\documentclass[10pt,twoside]{smfart}

\usepackage[latin1]{inputenc}
\usepackage[french]{babel}
\usepackage[all]{xy}
\usepackage[mathscr]{euscript}
\usepackage{mathrsfs,amsmath,amssymb,bbm,stmaryrd}

\newtheorem{thm}{Théorème}[section]
\newtheorem{prop}[thm]{Proposition}
\newtheorem{cor}[thm]{Corollaire}
\newtheorem{lm}[thm]{Lemme}

\newtheorem*{thm_i}{Théorème}
\newtheorem*{cor_i}{Corollaire}

\theoremstyle{definition}
\newtheorem{df}[thm]{Définition}

\newtheorem{num}[thm]{}

\theoremstyle{remark}
\newtheorem{rem}[thm]{Remarque}

\numberwithin{equation}{thm}
\renewcommand\theequation{\thethm.\alph{equation}}

\newcommand{\A}{\mathscr A}
\newcommand{\B}{\mathscr B}
\newcommand{\T}{\mathscr T}
\newcommand{\ab}{\mathscr Ab}
\newcommand{\sm}{\mathscr L_k}
\newcommand{\smc}{\mathscr L^{cor}_k}

\newcommand{\ftr}{Sh^{tr}(k)}
\newcommand{\ftrgr}{\mathbb N\!-\!\ftr}
\newcommand{\hftr}{HI(k)}

\newcommand{\hftrgr}{\mathbb Z\!-\!\hftr}
\newcommand{\tatemod}{\tatex *\!-\!\mathrm{mod}}
\newcommand{\Tmod}{T^*\!-\!\mathrm{mod}}

\newcommand{\hmtr} {HI_*(k)}
\newcommand{\hmtro} {HI^{(0)}(k)}
\newcommand{\modl} {\mathscr{M}Cycl(k)}
\newcommand{\Vdmme} {DM_{-}^{eff}\!(k)}
\newcommand{\DM} {DM\!(k)}
\newcommand{\DMe} {DM^{eff}\!(k)}
\newcommand{\DMb} {DM^{\circ}\!(k)}
\newcommand{\DMgme} {DM_{gm}^{eff}\!(k)}
\newcommand{\DMgmb} {DM_{gm}^{\circ}\!(k)}
\newcommand{\DMgm} {DM_{gm}\!(k)}
\newcommand{\DMgmo} {DM_{gm}^{(0)}\!(k)}
\newcommand{\SHtop} {\mathscr{SH}_{top}}

\newcommand{\NN} {{\mathbb N}}
\newcommand{\ZZ} {{\mathbb Z}}
\newcommand{\QQ} {{\mathbb Q}}

\newcommand{\cO}{\mathcal O}
\newcommand{\W}{\mathcal W}
\newcommand{\cS}{\mathcal S}

\renewcommand{\AA}{\mathbb A}
\newcommand{\GG}{\mathbb G_m}
\newcommand{\GGx}[1]{\mathbb G_{m,#1}}
\newcommand{\PP}{\mathbb P}

\newcommand{\un}{\mathbbm 1}
\newcommand{\tate}{S^1_t}
\newcommand{\tatex}[1]{S^{#1}_t}
\newcommand{\htatex}[1]{\hat S^{#1}_t}

\DeclareMathOperator{\Ker}{Ker}
\DeclareMathOperator{\coKer}{coKer}
\DeclareMathOperator{\Comp}{C}
\DeclareMathOperator{\Der}{D}

\DeclareMathOperator{\sing}{\underline C^*_{sing}}

\newcommand{\hlrep}{h_0}
\newcommand{\hhlrep}{\hat h}

\newcommand{\derL}{\mathrm L}
\newcommand{\derR}{\mathrm R}

\DeclareMathOperator{\uH} {\underline H}
\newcommand{\upH} {\smash{\breve{\underline{\mathrm{H}}}}}
\newcommand{\ugH}[1]{\underline H_{#1}}
\newcommand{\umH}{\smash{\hat{\underline{\mathrm{H}}}}}

\DeclareMathOperator{\otr}{\otimes^{tr}}
\DeclareMathOperator{\otrL}{\otimes^{tr,\derL}}
\DeclareMathOperator{\otrgr}{\hat \otimes^{tr}}
\DeclareMathOperator{\ohtr}{\otimes^{Htr}}
\DeclareMathOperator{\ohtrgr}{\hat \otimes^{Htr}}
\newcommand{\rep}[1]{\ZZ^{tr}(#1)}
\newcommand{\repNP}{\ZZ^{tr}}

\DeclareMathOperator{\Stab} {\Sigma^{\infty}}
\DeclareMathOperator{\Lop} {\Omega^{\infty}}
\DeclareMathOperator{\hStab} {\sigma^{\infty}}
\DeclareMathOperator{\hLop} {\omega^{\infty}}
\DeclareMathOperator{\Drap} {Drap}

\newcommand{\ilim}[1] { \underset{#1}{\varinjlim} \ }

\newcommand{\pplim}[1] { \underset{#1}{"\varprojlim"} \ }

\DeclareMathOperator{\tra}{{}^t \!}

\newcommand{\doto} {\bullet\!\!\!\longrightarrow}

\newcommand{\spec}[1] {\operatorname{\mathrm{Spec}}(#1)}

\DeclareMathOperator{\Hom}{Hom}
\DeclareMathOperator{\uHom}{\underline{Hom}}
\DeclareMathOperator{\RHom}{\derR \underline{Hom}}

\newcommand{\et}{\text{\'et}}
\newcommand{\nis}{\text{Nis}}
\newcommand{\zar}{\text{Zar}}

\DeclareMathOperator\pro{pro}

\newcommand{\C}{\mathscr C}
\newcommand{\cI}{\mathcal I}
\newcommand{\cM}{\mathcal M}
\newcommand{\sP}{\mathscr P}
\newcommand{\cX}{\mathcal X}
\newcommand{\cY}{\mathcal Y}
\DeclareMathOperator{\Tot}{Tot}
\newcommand{\E}{\mathbb E}
\newcommand{\F}{\mathbb F}
\newcommand{\cE}{\mathcal E}

\begin{document}

\sloppy
\setcounter{secnumdepth}{3}
\setcounter{tocdepth}{2}

\title{Modules homotopiques}
\alttitle{Homotopy modules}

\author{Frédéric Déglise}
\date{septembre 2009}

\address{CNRS, LAGA (UMR 7539),
 Institut Galil\'ee,
 Universit\'e Paris 13,
 99 avenue Jean-Baptiste Cl\'ement,
 93430 Villetaneuse
 FRANCE \\
 Tel.:  +33 1 49 40 35 77
}

\email{deglise@math.univ-paris13.fr}
\urladdr{http://www.math.univ-paris13.fr/~deglise/}
              
\email{deglise@math.univ-paris13.fr}

\begin{abstract}
En s'appuyant sur certains de nos précédents articles,
 nous comparons au-dessus d'un corps parfait $k$
 la catégorie des faisceaux invariants par homotopie avec transferts
 introduite par V.~Voevodsky à celle des modules de cycles introduite par M.~Rost:
 la première est une sous-catégorie pleine de la seconde.
Utilisant la construction récente par D.C~Cisinski et l'auteur d'un version
 non effective $DM(k)$ de la catégorie des complexes motiviques,
 nous étendons ce résultat et montrons que la catégorie des modules de cycles
 et le coeur d'une t-structure naturelle sur $DM(k)$ qui généralise la
 t-structure homotopique des complexes motiviques.
Nous illustrons ces résultats en prouvant que la suite spectrale du coniveau
 pour la cohomologie d'un $k$-schéma lisse représentée par un objet de $DM(k)$
 co\"incide avec la suite spectrale d'hypercohomologie pour la t-structure
 homotopique, généralisant un résultat classique de S.~Bloch et A.~Ogus.
\end{abstract}

\begin{altabstract}
Based on previous works,
 we compare over a perfect field $k$ 
 the category of homotopy invariant sheaves with transfers introduced
 by V.~Voevodsky and the category of cycle modules introduced by M.~Rost:
 the former is a full subcategory of the latter.
Using the recent construction by D.C.~Cisinski and the author of a non effective
 version $DM(k)$ of the category of motivic complexes,
 we show that cycle modules form the heart of a natural t-structure on $DM(k)$,
 generalising the homotopy t-structure on motivic complexes.
We enlighten these results by proving that the coniveau spectral sequence
 for the cohomology of a smooth $k$-scheme represented by an object of $DM(k)$
 coincides with the hypercohomology spectral sequence for the homotopy t-structure,
 generalising a classical result of S.~Bloch and A.~Ogus
\end{altabstract}

\subjclass{14F42, (14C15, 14C35)}
\keywords{motifs mixtes, complexes motiviques, modules de cycles, filtration par coniveau}
\thanks{L'auteur est partiellement financé par l'ANR, projet no.\ ANR-07-BLAN-0142 \og M\'ethodes \`a la Voevodsky, motifs mixtes et G\'eom\'etrie d'Arakelov \fg.}

\maketitle

\tableofcontents

\mainmatter

\subsection*{Introduction}

\paragraph{Théorie de Voevodsky}

Dans sa théorie des complexes motiviques sur un corps parfait $k$,
 V. Voevodsky introduit le concept central de faisceau Nisnevich 
 invariant par homotopie avec transferts,
 que nous appellerons simplement \underline{faisceau homotopique}.
Rappelons qu'un faisceau homotopique $F$ est un préfaisceau de groupes abéliens
 sur la catégorie des $k$-schémas algébriques lisses,
 fonctoriel par rapport aux correspondances finies à homotopie près, 
 qui est un faisceau pour la topologie de Nisnevich.
Un exemple central d'un tel faisceau est donné par le préfaisceau $\GG$ qui
 à un schéma lisse $X$ associe le groupe des sections globales inversibles sur $X$.
La catégorie des faisceaux homotopiques, notée ici $\hftr$,
 a de bonnes propriétés
 que l'on peut résumer essentiellement en disant que c'est une catégorie
 abélienne de Grothendieck, monoïdale symétrique fermée.

Un des points centraux de la théorie est la démonstration par Voevodsky
 que tout faisceau homotopique $F$ admet une \emph{résolution de Gersten}\footnote{
Les complexes du type \eqref{gersten}, ci-dessous, ont été introduits par Grothendieck
 sous le nom \emph{résolution de Cousin}, remplaçant la théorie des faisceaux homotopiques
 par celle des faisceaux cohérents.
Grâce à la suite spectrale associée à la filtration par coniveau d'après Grothendieck,
 ils ont été réintroduits un peu plus tard dans le contexte des théories cohomologiques,
 par Brown et Gersten en K-théorie et finallement par Bloch et Ogus dans une version 
 axiomatique. Notons que ces derniers auteurs parlent plutôt de 
 \og\emph{arithmetic resolution}\fg~
 et il semble que le terme de \emph{résolution de Gersten} se soit imposé par la suite.
La fonctorialité de la résolution de Gersten par rapport
 à un morphisme de schémas lisses a été traitée dans le cas
 de la K-théorie par H.~Gillet (voir \cite{Gillet}) puis étendue
 dans le cas des modules de cycles par M.~Rost.}.
Un cas particulier de ce résultat est le fait que pour tout schéma lisse $X$,
 le groupe abélien $F(X)$ admet une résolution par un complexe de la forme:
 \begin{equation}\tag{$\mathcal G$} \label{gersten}
C^*(X,\hat F_*):\bigoplus_{x \in X^{(0)}} \hat F(\kappa(x))
 \rightarrow \hdots \bigoplus_{x' \in X^{(n)}} \hat F_{-n}(\kappa(x'))
 \rightarrow \hdots
\end{equation}
Suivant Voevodsky, $F_{-n}=\uHom_{\hftr}(\GG^{\otimes n},F)$.
On a noté $X^{(n)}$ l'ensemble des points de codimension $n$ de $X$.
En degré $n$, ce complexe est formé par la somme des fibres du faisceau
$F_{-n}$ en les points du topos Nisnevich définis par $\kappa(x)$,
le corps résiduel de $x$.

Un corollaire de cette résolution de Gersten est que les faisceaux homotopiques
 sont essentiellement déterminés par leurs fibres en un corps de fonctions.
 La question centrale de cet article est de savoir jusqu'à quel point ils le sont.

\paragraph{Théorie de Rost}
 
Pour définir un complexe de Gersten, du type \eqref{gersten}, 
 on remarque qu'il faut essentiellement se donner un groupe abélien 
 pour chaque corps résiduel d'un point de $X$.
M.~Rost axiomatise cette situation en introduisant les \underline{modules de cycles}.
Un module de cycles est un foncteur $\phi$ 
 de la catégorie des corps de fonctions au-dessus de $k$ vers les groupes
 abéliens gradués,
 muni d'une fonctorialité étendue qui permet de définir un complexe $C^*(X,\phi)$ 
 du type \eqref{gersten}. 
Pour avoir une idée de cette fonctorialité,
 le lecteur peut se référer aux propriétés de la K-théorie de Milnor
  -- mais aussi à la théorie des modules galoisiens.
Rost note l'analogie entre ce complexe et le groupe des cycles de $X$
 -- comme l'avaient fait Bloch et Quillen avant lui -- et utilise le traitement de
 la théorie de l'intersection par Fulton pour montrer que la \emph{cohomologie} du complexe,
 notée $A^*(X,\phi)$, est naturelle en $X$ par rapport aux morphismes de schémas lisses.

\paragraph{Une comparaison}

Répondant à la question finale du premier paragraphe,
 nous comparons la théorie de Rost et celle de Voevodsky.
D'une manière vague, notre résultat principal affirme
 que l'association $F \mapsto \hat F_*$ définit un foncteur
 pleinement fidèle des faisceaux homotopiques dans les modules de cycles,
 avec pour quasi-inverse à gauche le foncteur $\phi \mapsto A^0(.,\phi)$.

Pour être plus précis dans la formulation de ce résultat,
 on est conduit à élargir la catégorie des faisceaux homotopiques.
On définit un \underline{module homotopique} $F_*$ comme un faisceau homotopique
$\ZZ$-gradué muni d'isomorphismes $\epsilon_n:F_n \rightarrow (F_{n+1})_{-1}$.
La catégorie obtenue, notée $\hmtr$, est encore abélienne de Grothendieck,
 symétrique monoïdale fermée.
De plus, elle contient comme sous-catégorie pleine la catégorie $\hftr$ --
si $F$ est un faisceau homotopique, le module homotopique associé 
 a pour valeur $\GG^{\otimes n} \otimes F$ (resp. $F_{-n}$) en degré $n \geq 0$
 (resp. $n<0$).

Dès lors, 
 on peut montrer que le système $\hat F_*$ 
 des fibres d'un module homotopique $F_*$
 en un corps de fonctions définit un module de cycles.
De plus, pour tout module de cycles $\phi$,
 le groupe $A^0(X,\phi)$,
  dépendant fonctoriellement d'un schéma lisse $X$,
   définit un module homotopique.
\begin{thm_i}[\emph{cf.} \ref{thm:main}]
Les deux associations décrites ci-dessus
 définissent des fonteurs quasi-inverses l'un de l'autre.
\end{thm_i}
La résolution de Gersten obtenue par Voevodsky est maintenant
 équivalente au résultat suivant:
\begin{cor_i}[\emph{cf.} \ref{iso_coh_gpe_chow}]
Si $F_*$ est un module homotopique et $X$ un schéma lisse,
 $H^n(X,F_*)=A^n(X,\hat F_*)$.\footnote{L'identification obtenue 
 ici est naturelle,
 non seulement par rapport au pullback,
 mais aussi par rapport aux correspondances finies
 et par rapport au pushout par un morphisme projectif.}
\end{cor_i}
 
\paragraph{L'interprétation motivique}

Rappelons qu'un complexe motivique est un complexe\footnote{Originellement,
 ces complexes sont supposés bornés supérieurement. 
Nous abandonnons cette hypothèse dans tout l'article suivant \cite{CD1}.}
 de faisceaux Nisnevich avec transferts dont les faisceaux de cohomologie
 sont des faisceaux homotopiques.
La catégorie des complexes motiviques $\DMe$ porte ainsi naturellement
 une t-structure au sens de Beilinson, Bernstein et Deligne dont le coeur 
 est la catégorie $\hftr$.

La catégorie $\DMe$ est triangulée monoïdale symétrique fermée.
Elle contient comme sous catégorie pleine la catégorie des motifs
 purs modulo équivalence rationnelle définie par Grothendieck.
C'est ainsi une catégorie \og effective \fg,
 dans le sens où le motif de Tate $\un(1)$ n'a pas de $\otimes$-inverse.
Suivant l'approche initiale de Grothendieck,
 on est conduit à introduire une version non effective
 des complexes motiviques ; c'est ce qui est fait par D.C.~Cisinski et l'auteur
  dans \cite{CD1}.
Il est naturel dans le contexte des complexes motiviques
 de remplacer la construction habituelle pour inverser $\un(1)$
 par l'approche des topologues pour définir la catégorie homotopique
 stable $\SHtop$.
La catégorie $\DM$, dont les objets seront appelés les \underline{spectres motiviques},
 est ainsi construite à partir du formalisme des spectres
 et des catégories de modèles.
C'est la catégorie monoïdale homotopique\footnote{c'est-à-dire la catégorie homotopique
 associée à une catégorie de modèles monoïdale.} \emph{universelle} 
 munie d'un foncteur dérivé monoïdal
$$
\Stab:\DMe \rightarrow \DM
$$
admettant un adjoint à droite $\Lop$ et telle que l'objet $\Stab \un(1)$
est $\otimes$-inversible.
Notons que dans le cadre des complexes motiviques, 
 le foncteur $\Stab$ est pleinement fidèle
 d'après le théorème de simplification de Voevodsky \cite{V3}. 

Dans cet article,
 nous montrons que l'on peut étendre la définition de la t-structure homotopique 
 à la catégorie $\DM$, de telle manière que le foncteur $\Lop$ est t-exact.
Le coeur de la t-structure homotopique sur $\DM$ est la catégorie $\hmtr$
 des modules homotopiques,
 qui est donc canoniquement identifiée à la catégorie des modules de cycles
 d'après le théorème \ref{thm:main} déjà cité.

Ceci nous permet de donner une interprétation frappante du module de cycles
 $\hat F_*$ associé à un module homotopique $F_*$, 
 à travers la notion de motifs génériques de \cite{Deg5}.\footnote{Cette
  notion a aussi été introduite par A.Beilinson dans \cite{Bei}.}
Le motif générique associé à un corps de fonctions $E$ est le pro-motif
 défini par tous les modèles lisses de $E$. 
On considère la catégorie $\DMgmo$ formée par tous les twists de motifs génériques 
 par $\un(n)[n]=\un\{n\}$ pour $n \in \ZZ$.
Alors, $\hat F_*$ est simplement la restriction du foncteur représenté par $F_*$
 dans $\DM$ à la catégorie $\DMgmo$. La catégorie $\DMgmo$ est une catégorie de 
 \og points \fg~pour les spectres motiviques, et la fonctorialité des modules
 de cycles est interprétée en termes de \emph{morphismes de spécialisations}
 entre ces points.
De ce point de vue, les modules homotopiques correspondent à des systèmes locaux
 où le groupoïde fondamental est remplacé par la catégorie $\DMgmo$.

\paragraph{Filtration par coniveau}

Comme on l'a déjà mentionné,
 la filtration par coniveau d'un schéma lisse $X$ est liée à la résolution de Gersten:
elle induit pour tout module homotopique $F_*$ une suite spectrale
convergente
$$
E_{1,c}^{p,q}(X,F_*) \Rightarrow H^{p+q}(X,F_*)
$$
dont le premier terme est concentré sur la ligne $q=0$
 et s'identifie \emph{canoniquement} au complexe $C^*(X,\hat F_*)$.
Sa dégénérescence explique donc le corollaire \ref{iso_coh_gpe_chow}.

Ce résultat peut être étendu aux spectres motiviques.
Pour un schéma lisse $X$ et un spectre motivique $\E$,
 on note $H^p(X,\E)$ le groupe des morphismes dans $\DM$
 entre le motif de $X$ et $E[p]$. On note aussi $\uH^q_*(\E)$
  (resp. $\umH^q_*(\E)$)
 le module homotopique (resp. module de cycles)
 correspondant au $q$-ième groupe de cohomologie pour la t-structure
 homotopique.
La filtration par coniveau sur $X$ induit une suite spectrale
$$
E_{1,c}^{p,q}=C^p(X,\umH^q_*(\E))_0 \Rightarrow H^{p+q}(X,\E).
$$
Par ailleurs, la filtration sur $\E$ pour la t-structure homotopique
détermine une suite spectrale
$$
E_{2,t}^{p,q}=H^p(X,\uH^q_*(\E))_0 \Rightarrow H^{p+q}(X,\E).
$$
\begin{thm_i}[\emph{cf.} \ref{thm:comparaison_ssp}]
L'identification du corollaire \ref{iso_coh_gpe_chow} détermine
 un isomorphisme canonique de suites spectrales
$$
E_{r,c}^{p,q} = E_{r,t}^{p,q}, \ r \geq 2.
$$
\end{thm_i}
On en déduit que la filtration par coniveau relativement à $X$ 
 sur $H^*(X,\E)$ coïncide avec la filtration pour la t-structure homotopique
 relativement à $\E$.
On en déduit aussi que la suite spectrale du coniveau satisfait
 de bonnes propriétés de fonctorialité par rapport au schéma lisse $X$:
 compatibilité aux pullbacks, pushouts, transferts
 et à l'action de la cohomologie motivique
 (\emph{cf.} \ref{fct_coniv} pour plus de précisions).

Ce résultat doit être comparé à un théorème de Bloch et Ogus
 montrant que la suite spectrale du coniveau 
 sur la cohomologie de De Rham coincide avec la suite spectrale 
 d'hypercohomologie Zariski du complexe de De Rham (\emph{cf.} \cite{BO}).
Plus précisément, on retrouve ce théorème grâce à la notion de
 \emph{théorie de Weil mixte} introduite par Cisinski et l'auteur
 dans \cite{CD2}.
 Toute théorie de Weil mixte correspond en effet à un spectre motivique
 $\E$ auquel on peut appliquer le théorème précédent, et la cohomologie
 de De Rham en caractéristique $0$ est un cas particulier de théorie
 de Weil mixte.
Pour retrouver le résultat sur la cohomologie de De Rham,
 il faut préciser que le module homotopique $\uH^q_*(\E)$
 coincide en degré $0$ avec la \emph{cohomologie non ramifiée}\footnote{Ce type de
 faisceau a été introduit pour la première fois dans \cite{BO} mais la terminologie
 utilisée ici est apparue un peu plus tard.} associée à $\E$,
 à savoir le faisceau Zariski associé au préfaisceau $X \mapsto H^q(X,\E)$.
De ce point de vue, les résultats de cet article apparaissent 
 dans le prolongement direct des travaux de Bloch et Ogus ; 
 la notion de module homotopique exprime ainsi 
 la \emph{structure} des faisceaux de cohomologie non ramifiée.
 
\paragraph{Plan du travail}
L'article est divisé en trois parties.

Dans la première partie,
 on introduit la catégorie des modules homotopiques (def. \ref{df:mod_htp})
 après quelques rappels sur les faisceaux homotopiques.
On rappelle aussi les grandes lignes de la théorie des modules de cycles
 et on établit le théorème central \ref{thm:main} cité précédemment en s'appuyant
 sur les articles \cite{Deg4} et \cite{Deg5} pour construire chaque foncteur
 de l'équivalence de catégories.
 On notera que tout le sel de ce théorème réside dans l'étude de la 
 fonctorialité de la résolution de Gersten. C'est ce qu'on exploite
 notamment en étudiant l'identification de \ref{iso_coh_gpe_chow}
 à la fin de cette partie.

Dans la seconde partie,
 on rappelle la théorie des complexes motiviques de Voevodsky
 en utilisant le travail effectué dans \cite{CD1} sur les catégories de modèles.
 Grâce à cela, on peut introduire la catégorie $\DM$ aisément
 et donner la définition de la t-structure homotopique sur cette catégorie,
 ainsi que l'identification du coeur.
 La fin de cette seconde partie est consacrée à la comparaison \ref{thm:comparaison_ssp}
 des suites spectrales du coniveau et de la t-structure homotopique,
 qui s'appuie de manière essentielle sur l'étude de la suite spectrale
 du coniveau déjà effectuée dans \cite{Deg6}.
On exploite ici la notion de \og filtration décalée \fg~introduite par Deligne,
 sur le modèle de \cite{Par}.
Dans un assez long préliminaire \ref{num:ssp},
 nous avons essayé de clarifier les diverses constructions des suites spectrales
 qui interviennent, l'une par un couple exact et l'autre par un complexe filtré.
C'est l'\emph{axiome de l'octaèdre} qui apparaît central ici et qui explique
 de manière frappante la dualité entre ces deux approches.

La troisième partie est consacrée à diverses applications et compléments.
Le premier de ceux-ci illustre la remarque déjà faite sur l'importance
 de l'étude de la fonctorialité de la résolution de Gersten: 
 elle nous permet de montrer
 que le morphisme de Gysin (\emph{cf.} \cite{Deg6})
  et la transposée d'un morphisme fini coïncident dans la catégorie $\DM$.
La deuxième application montre que pour une large classe de modules de cycles,
 la graduation structurale est bornée inférieurement. Cette classe est
 déterminée par une notion de constructibilité sur les modules homotopiques
 introduite à cette occasion, et comparée à d'autres définitions.
On montre ensuite en quoi la théorie des groupes de Chow à coefficients
 dans un module de cycle est particulièrement intéressante dans le cas des
 schémas singuliers, en les identifiant (dans les cas où on peut)
 avec l'homologie de Borel-Moore du module homotopique correspondant\footnote{On
 notera ainsi que dans la définition des différentielles du complexe $C^*(X,\phi)$,
 on utilise de manière centrale le procédé de normalisation pour désingulariser
 les sous-schémas fermés de $X$ en codimension $1$. 
 Cette résolution (faible) des singularités explique dans une certaine mesure 
 les bonnes propriétés des groupes $A^*(X,\phi)$.}.
On termine l'article sur une comparaison entre les motifs génériques
 et les motifs birationnels introduits par B.~Kahn et R.~Sujatha (\emph{cf.} \cite{KS})
 qui permet de compléter la fonctorialité des motifs génériques.
 
\paragraph{Perspective}
Comme on l'a déjà remarqué, la catégorie $\DM$ est analogue à la catégorie
 homotopique stable topologique. C'est dans le cadre de la catégorie homotopique
 stable des schémas que cette analogie prend tout son sens.
Ainsi, F.~Morel a introduit et largement étudié l'analogue des faisceaux
 homotopiques dans ce cadre (les faisceaux Nisnevich strictement invariants
 par homotopie, \cite{Mor3}). Il a aussi introduit l'analogue des
 modules homotopiques, interprétés comme objets du coeur de la catégorie
 homotopique stable des schémas dans \cite{Mor1}.
Sa notion de module homotopique est plus générale que la nôtre.\footnote{On 
 peut résoudre la contradiction entre les deux terminologies
 en utilisant l'expression \og module homotopique avec transferts \fg~pour 
 désigner les objets considérés ici.
 Une autre possibilité est d'appeller \og modules homotopiques généralisés \fg~les objets
 considérés par Morel, venant de la catégorie homotopique stable.}
Dans un travail en préparation,
 nous montrons une conjecture de Morel prédisant que les modules homotopiques
 introduits ici sont équivalents aux modules homotopiques orientables.
Cette conjecture s'inscrit plus généralement dans le tableau
 actuel qui veut que la catégorie $\DM$ s'identifie à la catégorie
 des spectres orientables avec loi de groupe formelle additive -- ceci
 est connu à coefficients rationnels et devrait se généraliser sur une base 
 quelconque (\emph{cf.} \cite{CD3}).\footnote{Ceci soulève aussi un problème de terminologie
 pour différencier les objets de la catégorie $\DM$ de ceux de la catégorie
 $SH(k)$. La terminologie de \og spectres motiviques \fg~a été employée 
 par plusieurs auteurs -- dont malheureusement Voevodsky -- pour désigner 
 les objets de $SH(k)$. Il nous semble plus opportun, 
 pour des raisons historiques évidentes, 
 de réserver cette terminologie aux objets de la catégorie $\DM$ et de parler de \og spectres
 motiviques généralisés \fg~pour désigner les objets de $SH(k)$. Cette terminologie
 présente l'avantage d'être conforme à celle utilisée en topologie algébrique.}

Le résultat concernant la comparaison de suites spectrales
 a été considéré indépendemment de nous par M.V.~Bondarko
 dans son travail sur les structures de poids (\emph{cf.} \cite{Bon}).
Notons pour terminer que notre démonstration s'appuie sur la 
\emph{preuve de Deligne}\footnote{Elle
est mentionnée par Bloch et Ogus dans \cite{BO} mais n'a pas été rédigée
par Deligne.} et que nous nous sommes particulièrement inspirés
de \cite{GS} et \cite{Par}.


\subsubsection*{Remerciements}

Mes remerciements vont en premier lieu à F.~Morel qui a dirigé ma thèse,
 dans laquelle le résultat central de cet article a été établit. L'influence de
 ses idées est partout dans ce texte. Je remercie aussi A.~Suslin et A.~Merkurjev
 qui ont été les rapporteurs de cette thèse et dont les rapports m'ont beaucoup aidés
 dans la rédaction présente,
  ainsi que D.C.~Cisinski pour sa relecture et son intérêt pour mon mémoire de thèse.
 Enfin, je remercie J.~Ayoub, A.~Beilinson, J.B.~Bost, B.~Kahn, J.~Riou, C.~Soulé
  et J.~Wildeshaus pour leur intérêt et des discussions autour du sujet de cet article.

\subsection*{Notations}

On fixe un corps parfait $k$. 
Tous les schémas considérés sont des $k$-schémas séparés.
Nous dirons qu'un schéma $X$ est lisse si il est lisse de type fini sur $k$.
La catégorie des schémas lisses est notée $\sm$.

Nous disons qu'un schéma $X$ est \emph{essentiellement de type fini}
s'il est localement isomorphe au spectre d'une $k$-algèbre
qui est une localisation d'une $k$-algèbre de type fini.

On appelle \emph{corps de fonctions} toute extension de corps
$E/k$ de degré de transcendance fini.
Un \emph{corps de fonctions valué} est un couple $(E,v)$ où $E$ est un corps
de fonctions et $v$ est une valuation sur $E$ dont l'anneau des entiers
est essentiellement de type fini sur $k$. \\
Un \emph{modèle} de $E/k$
est un $k$-schéma lisse connexe $X$ muni d'un $k$-isomorphisme
entre son corps des fonctions et $E$.
On définit le pro-schéma des modèles de $E$~:
$$
(E)=\pplim{A \subset E} \spec A
$$
où $A$ parcourt l'ensemble ordonné filtrant des sous-$k$-algèbres
de type fini de $E$ dont le corps des fractions est $E$.

Voici une liste des catégories principales utilisées dans ce texte:
\begin{itemize}
\item $\DMgme$ (resp. $\DMgm$) désigne la catégorie des motifs géométriques effectifs
 (resp. non nécessairement effectifs).
\item $\DMe$ désigne la catégorie des complexes motiviques
 (que l'on ne suppose pas nécessairement bornés inférieurement).
\item $\DM$ désigne la catégorie des spectres motiviques, version non effective
 de $\DMe$.
\item $\hftr$ (resp. $\hmtr$) désigne la catégorie des faisceaux
 (resp. modules) homotopiques. C'est le coeur de la t-structure homotopique
 sur $\DMe$ (resp. $\DM$).
\item $\modl$ désigne la catégorie des modules de cycles.
\end{itemize}

\part{Modules homotopiques et modules de cycles} 

\section{Modules homotopiques}

\subsection{Rappels sur les faisceaux avec transferts}

\num Soient $X$ et $Y$ des schémas lisses.
Rappelons qu'une \emph{correspondance finie} de $X$ vers $Y$ est un cycle
de $X \times Y$ dont le support est fini équidimensionel sur $X$.
La formule habituelle permet de définir un produit de composition
pour les correspondances finies qui donne lieu à une
catégorie additive $\smc$ (\emph{cf.} \cite[4.1.19]{Deg7}).
On obtient un foncteur $\gamma:\sm \rightarrow \smc$,
égal à l'identité sur les objets,
en associant à tout morphisme le cycle associé à son graphe.
La catégorie $\smc$ est enfin monoïdale symétrique.
Le produit tensoriel sur les objets est donné par le produit
cartésien des schémas lisses; sur les morphismes, il est induit
par le produit extérieur des cycles (\emph{cf.} \cite[4.1.23]{Deg7}).

\num Un \emph{faisceau avec transferts} est un foncteur 
$F:(\smc)^{op} \rightarrow \ab$ additif contravariant
tel que $F \circ \gamma$ est un faisceau Nisnevich.
On note $\ftr$ la catégorie des faisceaux avec transferts
munis des transformations naturelles.
Cette catégorie est abélienne de Grothendieck
 (\emph{cf.} \cite[4.2.8]{Deg7}).
Une famille génératrice est donnée par les faisceaux
représentables par un schéma lisse $X$~:
$$
\rep X:Y \mapsto c(Y,X).
$$
Il existe un unique  produit tensoriel symétrique $\otr$
sur $\ftr$ telle que le foncteur $\repNP$ est monoïdal
symétrique. La catégorie $\ftr$ est de plus
monoïdale symétrique fermée (\emph{cf.} \cite[4.2.14]{Deg7}).

\begin{df}
Un \emph{faisceau homotopique} est un faisceau avec
transferts $F$ invariant par homotopie~:
pour tout schéma lisse $X$, le morphisme induit par la 
projection canonique $F(X) \rightarrow F(\AA^1_X)$
est un isomorphisme.
\end{df}
On note $\hftr$ la sous-catégorie pleine de $\ftr$
formée des faisceaux homotopiques.
Le foncteur d'oubli évident $\cO:\hftr \rightarrow \ftr$
admet un adjoint à gauche $\hlrep:\ftr \rightarrow \hftr$,
$\hlrep(F)$ étant défini comme le faisceau associé au
préfaisceau
\begin{equation} \label{def:hlrep}
X \mapsto \coKer\Big( F(\AA^1_X)
 \xrightarrow{s_0^*-s_1^*} F(X) \Big)
\end{equation}
avec $s_0$ (resp. $s_1$) la section nulle (resp. unité)
de $\AA^1_X/X$ (\emph{cf.} \cite[4.4.4, 4.4.15]{Deg7}).
D'après \emph{loc. cit.}, le foncteur $\cO$ est exact.
La catégorie $\hftr$ est donc une sous-catégorie épaisse de $\ftr$.
En particulier, c'est une catégorie abélienne de Grothendieck
dont une famille génératrice est donnée par les faisceaux 
de la forme $\hlrep(X):=\hlrep(\rep X)$. On vérifie aisément
que le foncteur $\cO$ commute de plus à toutes les limites projectives
ce qui implique que $\hftr$ admet des limites projectives.

\num \label{foncteurs_fibres_fhtp}
Pour un corps de fonctions $E$,
on définit la fibre de $F$ en $E$ comme la limite inductive de l'application
de $F$ au pro-schéma $(E)$~:
$$
\hat F(E)=\ilim{A \subset E} F(\spec A)
$$
Les foncteurs $F \mapsto \hat F(E)$ forment une famille conservative
de foncteurs fibres\footnote{\emph{i.e.} exacts commutant aux 
limites inductives.} de $\hftr$ (\emph{cf.} \cite[4.4.7]{Deg7}).

\rem Ce dernier résultat repose sur la propriété
 très intéressante des faisceaux homotopiques suivante:
\begin{prop} \label{prop:h_0(imm_ouv)=epi}
Pour toute immersion ouverte dense $j:U \rightarrow X$ dans un schéma lisse,
 le morphisme induit
$$
j_*:\hlrep(U) \rightarrow \hlrep(X)
$$
est un épimorphisme dans $\hftr$.
\end{prop}
\noindent Cette proposition est une conséquence du corollaire 4.3.22 de \cite{Deg7}:
il existe un recouvrement ouvert $W \xrightarrow \pi X$ et une correspondance finie
$\alpha:W \rightarrow U$ telle que le diagramme suivant est commutatif à homotopie près
$$
\xymatrix@C=20pt@R=8pt{
& W\ar_{\alpha}[ld]\ar^{\pi}[d] \\
U\ar|j[r] & X.
}
$$
Elle implique en particulier que pour tout schéma lisse connexe
 $X$ de corps des fontions $E$, le morphisme canonique $F(X) \rightarrow \hat F(E)$
 est un monomorphisme.

\num \label{num:localisation}
Dans une catégorie abélienne de Grothendieck $\A$,
une classe de flèches $\W$ est dite localisante si~:
\begin{enumerate}
\item[(i)] $\W$ est stable par limite inductive.
\item[(ii)] Soit $f$ et $g$ des flèches composables de $\A$.
Si deux des constituants de $(f,g,gf)$ appartiennent à $\W$, 
le troisième appartient à $\W$.
\end{enumerate}
Si $\cS$ est un classe de flèche essentiellement petite,
on peut parler de la classe de flèches localisante engendrée
par $\cS$.
\begin{lm}
Il existe un unique produit tensoriel symétrique $\ohtr$ sur
$\hftr$ tel que le fonteur $\hlrep$ est monoïdal symétrique.
\end{lm}
\begin{proof}
D'après ce qui précède, $\hftr$ s'identifie à la localisation de 
la catégorie $\ftr$ par rapport à la classe de flèches localisante
engendrée par les morphismes $\rep{\AA^1_X} \rightarrow \rep X$
 pour un schéma lisse $X$ arbitraire.
Ainsi, pour tout schéma lisse $X$, $\W \otr \rep X \subset \W$.
Donc le produit tensoriel $\otr$ satisfait la propriété de localisation 
par rapport à $\W$ ce qui démontre le lemme.
\end{proof}
La catégorie $\hftr$ munie du produit tensoriel $\ohtr$ obtenu 
dans le lemme précédent est monoïdale symétrique fermée.
Ce produit tensoriel est caractérisé par la relation
$\hlrep(X) \ohtr \hlrep(Y)=\hlrep(X \times Y)$ déduite du lemme précédent.

\begin{df}
Soit $s:\{1\} \rightarrow \GG$ l'immersion du point unité.
On appelle \emph{sphère de Tate} le conoyau de $\hlrep(s)$ dans la
catégorie $\hftr$. On la note $\tate$.
\end{df}
D'après l'invariance par homotopie, on obtient encore une suite
exacte courte scindée dans $\hftr$~:
$$
0 \rightarrow \tate \rightarrow \hlrep(\GG)
 \xrightarrow{j_*} \hlrep(\AA^1_k)
 \rightarrow 0.
$$
où $j$ est l'immersion ouverte évidente. 
D'après la proposition 2.2.4 de \cite{Deg5},
la sphère de Tate est isomorphe en tant que faisceau au groupe
multiplicatif $\GG$ -- les transferts induits sur $\GG$ se calculent
à l'aide de l'application norme suivant \emph{loc. cit.}

\begin{lm} \label{lm:permutation_tate}
L'automorphisme de permutation des facteurs
 sur $\tate \ohtr \tate$ est égal à $-1$.
\end{lm}
\begin{proof}
Ce fait est bien connu du point de vue de la cohomologie
motivique. On donne ici une démonstration élémentaire basée
sur un argument de Suslin et Voevodsky.
Pour tout schéma affine lisse $X=\spec A$, 
 on obtient par définition des épimorphismes~:
$$
\GG(X) \otimes_\ZZ \GG(X)
 \xrightarrow \alpha \Gamma(X,\hlrep(\GG^2))
  \xrightarrow \beta \Gamma(X,\tate \ohtr \tate).
$$
On considère le sous-schéma fermé $H$ de
 $\AA^1 \times X \times \GG=\spec{A[t,u,u^{-1}]}$
définit par l'équation
$$
t(u-a)(u-b)+(1-t)(u-ab)(u-1)=0.
$$
Il est de codimension $1$ et domine $\AA^1_X$. C'est donc
une correspondance finie de $\AA^1_X$ dans $\GG$. 
D'après l'équation ci-dessus, $H \circ s_0=\{ab\}+\{1\}$
et $H \circ s_1=\{a\}+\{b\}$. Si $\delta$ désigne
l'immersion diagonale de $\GG$, l'homotopie $\delta \circ H$
montre donc la relation
$$
\alpha(ab,ab)+\alpha(1,1)=
\alpha(a,a)+\alpha(b,b).
$$
On en déduit $\beta \alpha(b,a)=-\beta \alpha(a,b)$
ce qui conclut d'après le lemme de Yoneda.
\end{proof}

\num \label{num:notation_graduation} 
Pour un entier $n \geq 0$,
 on note $\tatex n$ la puissance tensorielle
$n$-ième de $\tate$ dans $\hftr$.
Si $F$ est un faisceau homotopique,
 on pose $F_{-n}=\uHom_{\hftr}(\tatex n,F)$.
Par définition, pour tout schéma lisse $X$,
$$
F_{-1}(X)=F(\GG \times X)/F(X).
$$
Le foncteur $?_{-n}$ est le $n$-ième itéré
 du fonteur  $?_{-1}$.
Ainsi la proposition 3.4.3 de \cite{Deg5} entraine~:
\begin{lm} \label{lm:-1_exact}
L'endofoncteur
 $\hftr \rightarrow \hftr, F \mapsto F_{-n}$
  est exact.
\end{lm}

Le résultat suivant est un corollaire du théorème
de simplification de Voevodsky \cite{V3}.
\begin{prop} \label{cancellation}
L'endofoncteur $\hftr \rightarrow \hftr, F \mapsto \tatex n \ohtr F$ est 
pleinement fidèle.
\end{prop}
\begin{proof}
Il suffit de considérer le cas $n=1$.
La preuve anticipe la suite de l'exposé puisqu'elle utilise la catégorie
$\Vdmme$ des complexes motiviques de Voevodsky définie dans \cite{V3}.
Le théorème central de \emph{loc. cit.} affirme que le twist de Tate 
 est pleinement fidèle dans $\Vdmme$.
Il en résulte que le morphisme canonique 
$F \rightarrow \uHom_{\Vdmme}(\repNP(1)[1],F(1)[1])$
est un isomorphisme. D'après \cite[3.4.4]{Deg5},
le membre de droite est égal à $\uH^0(F(1)[1])_{-1}$.
Or par définition, $\uH^0(F(1)[1])=\tate \ohtr F$ et la transformation
naturelle correspondante $F \rightarrow (\tate \ohtr F)_{-1}$
est l'application d'adjonction.
\end{proof}

\subsection{Définition}

\num \label{num:cat_gr}
On note $\hftrgr$ la catégorie des faisceaux homotopiques $\ZZ$-gradués.
Pour un tel faisceau $F_*$ et un entier $n \in \ZZ$,
on note $F_*\{n\}$ le faisceau gradué dont la composante en degré $i$ est 
$F_{i+n}$. Si $F$ est un faisceau homotopique, on note encore $F\{n\}$
le faisceau gradué concentré en degré $-n$ égal à $F$.
La catégorie $\hftrgr$ est abélienne de Grothendieck
avec pour générateurs la famille $(\hlrep(X)\{i\})$ indexée par les schémas 
lisses $X$ et les entiers $i \in \ZZ$.

Cette catégorie est monoïdale symétrique~:
$$
\left(F_* \ohtrgr G_*\right)_n=\oplus_{p+q=n} F_p \ohtr G_q.
$$
Pour la symétrie, on adopte la convention donnée par la règle de Koszul~:
$$
\oplus_{p+q=n} F_p \ohtr G_q
 \xrightarrow{\sum (-1)^{pq}.\epsilon_{pq}} \oplus_{p+q=n} G_q \ohtr F_p
$$
où $\epsilon_{pq}$ désigne l'isomorphisme de symétrie pour la structure monoïdale
des faisceaux homotopiques.

On note $\tatex *$ le faisceau homotopique gradué concentré en degré positif,
égal en degré $n$ à $\tatex n$. Compte tenu de la règle de Koszul ci-dessus
et du lemme \ref{lm:permutation_tate}, c'est un monoïde commutatif
dans $\hftrgr$. On note $\tatemod$ la catégorie des modules sur $\tatex *$.
C'est une catégorie abélienne monoïdale de Grothendieck avec pour 
générateurs $(\tatex * \ohtr \hlrep(X)\{i\})$ pour $X$ un schéma lisse et 
$i \in \ZZ$.
Le morphisme structural d'un $\tatex *$-module $(F_*,\tau)$ est déterminé par 
 la suite de morphismes $\tate \ohtr F_n \xrightarrow{\tau_n} F_{n+1}$,
 appelés {\it morphismes de suspension}.
\begin{df} \label{df:mod_htp}
Un \emph{module homotopique} est un $\tatex *$-module $(F_*,\tau)$
tel que le morphisme adjoint à $\tau_n$
$$
\epsilon_n:F_n \rightarrow \uHom_{\hftr}(\tate,F_{n+1})=(F_{n+1})_{-1}
$$
est un isomorphisme. On note $\hmtr$ la sous-catégorie de $\tatemod$
formée des modules homotopiques.
\end{df}

\num \label{construction_base_hmtr}
Compte tenu du lemme \ref{lm:-1_exact}, le foncteur d'inclusion
$\hmtr \rightarrow \tatemod$ est exact et conservatif. 
Il admet de plus un adjoint
à gauche $L$ définit pour tout faisceau homotopique $F$ par la formule
$$
L\big(\tatex * \ohtr F\{i\}\big)_n=\begin{cases}
\tatex{n+i} \ohtr F & \text{si } n+i \geq 0 \\
F_{n+i} & \text{si } n+i \leq 0
\end{cases}
$$
en adoptant la notation de \ref{num:notation_graduation}.
Le fait que $L$ prend ses valeurs dans les faisceaux homotopiques
résulte de \ref{cancellation}. On pose plus simplement
$\hStab F\{i\}=L\big(\tatex * \ohtr F\{i\}\big)$.
La catégorie $\hmtr$ est donc une sous-catégorie abélienne
de $\tatemod$, avec pour générateurs la famille
\begin{equation} \label{notation:h_0*}
h_{0,*}=\hStab \hlrep(X)\{i\}
\end{equation}
pour un schéma lisse $X$ et un entier $i \in \ZZ$ --
 le symbole $*$ correspond à la graduation naturelle
 de module homotopique.

Si $(F_*,\tau)$ est un module homotopique, on pose 
$\hLop F_*=F_0$. On obtient ainsi un couple de foncteurs adjoints
$$
\hStab:\hftr \leftrightarrows \hmtr:\hLop
$$
tels que $\hStab$ est pleinement fidèle (prop. \ref{cancellation})
et $\hLop$ est exact (lemme \ref{lm:-1_exact}). 
Ainsi, pour tout schéma lisse $X$, tout module homotopique $F_*$
et tout $(n,i) \in \ZZ^2$,
\begin{equation} \label{calcul_Hom_hmtr}
\Hom_{\hmtr}(h_{0,*}(X),F_*\{i\}[n])=H^n_\nis(X;F_i).
\end{equation}
\begin{lm}
Il existe sur $\hmtr$ une unique structure monoïdale symétrique 
telle que le foncteur $L$ est monoïdal symétrique.
\end{lm}
\begin{proof}
Compte tenu de ce qui précède,
le foncteur $L$ est un foncteur de localisation:
pour tout schéma lisse $X$ et tout entier $n \in \ZZ$,
on obtient par définition
$(\tatex * \ohtr \hlrep(X)\{n\})_{-n+1}=\tate \ohtr \hlrep(X)$.
Par adjonction, l'identité de $\tate \ohtr \hlrep(X)$ induit
donc un morphisme de $\tatex *$-modules
$$
\tatex * \ohtr (\tate \ohtr \hlrep(X)\{n-1\})
 \rightarrow \tatex * \ohtr \hlrep(X)\{n\}.
$$
Utilisant à nouveau le jeu des adjonctions introduites ci-dessus,
$\hmtr$ est la localisation de $\tatemod$ par rapport à la 
classe de flèches localisante $\W$ (\emph{cf.} \ref{num:localisation}) 
engendrée par les morphismes précédents.
Pour tout couple $(Y,m)$, $Y$ schéma lisse, $m \in \ZZ$,
il est évident que $\W \ohtrgr (\tatex * \ohtr \hlrep(Y)\{m\}) \subset \W$.
Ainsi, $\ohtrgr$ vérifie la propriété de localisation par rapport à $\W$
ce qui conclut.
\end{proof}
La catégorie $\hmtr$ est donc monoïdale symétrique fermée
avec pour neutre le module homotopique $\tatex *$.
Le foncteur $\hStab$ est de plus monoïdal symétrique.
Enfin, l'objet $\hStab \tate$ est inversible pour le produit tensoriel
avec pour inverse $\hStab \ZZ^{tr}\{-1\}$.

\rem La catégorie $\hmtr$ est la catégorie monoïdale abélienne 
de Grothendieck {\it universelle} pour les propriétés qui viennent d'être 
énoncées.
La construction donnée ici est parfaitement analogue à la construction
de la catégorie des spectres en topologie algébrique, comme le suggère
nos notations -- en particulier pour le faisceau $\tate$ qui joue le
rôle de la sphère topologique.
La construction ici est facilitée parce que nous sommes
dans un cadre abélien et que la sphère $\tate$ est anti-commutative.
Le théorème de simplification \ref{cancellation} rend la construction
du foncteur $L$ plus facile mais n'est pas indispensable.

\subsection{Réalisation des motifs géométriques}
\label{realisation_mhtp}

Rappelons que la catégorie des motifs géométriques effectifs $\DMgme$ définie
 par Voevodsky
est l'enveloppe pseudo-abélienne de la localisation de la catégorie $K^b(\smc)$ 
des complexes de $\smc$ à équivalence d'homotopie près par la sous-catégorie 
triangulée épaisse engendrée par les complexes suivants~:
\begin{enumerate}
\item 
 $\hdots 0 \rightarrow U \cap V \rightarrow U \oplus V \rightarrow X 
  \rightarrow 0 \hdots$ \\
 pour un recouvrement ouvert $U \cup V$ d'un schéma lisse $X$.
\item $\hdots 0 \rightarrow \AA^1_X \rightarrow X \rightarrow 0 \hdots$ \\
 induit par la projection canonique pour un schéma lisse $X$.
\end{enumerate}
Rappelons que cette catégorie est triangulée monoïdale symétrique. 
Pour un schéma lisse $X$, on note simplement $M(X)$ le complexe concentré
en degré $0$ égal à $X$ vu dans $\DMgme$. \\
Pour tout complexe borné $C$ de $\smc$, on note $\rep C$ le complexe
de faisceau avec transferts évident. Pour un faisceau homotopique $F$,
posons $\varphi_F(C)=\Hom_{D(\ftr)}(\rep C,F)$.
Rappelons que pour un schéma lisse $X$,
 $\Hom_{D(\ftr)}(\rep X[-n],F)=H^n_\nis(X;F)$ (\emph{cf.} \cite[4.2.9]{Deg7}); 
la cohomologie Nisnevich de $F$ est de plus invariante par homotopie (\emph{cf.} \cite[4.5.1]{Deg7}).
On en déduit que le foncteur $\varphi_F$ ainsi défini se factorise et induit un foncteur
cohomologique encore noté $\varphi_F:\DMgme^{op} \rightarrow \ab$.

On définit le motif de Tate \emph{suspendu}\footnote{En effet, $\ZZ\{1\}=\ZZ(1)[1]$.}
 $\ZZ\{1\}$ comme le complexe
$$
\hdots \rightarrow\spec k \rightarrow \GG \rightarrow 0 \hdots
$$
où $\GG$ est placé en degré $0$, vu dans $\DMgme$. 
Avec une convention légèrement différente de celle de Voevodsky, adpatée à nos besoins,
on définit la catégorie des motifs géométriques $\DMgm$ comme la catégorie monoïdale
symétrique universelle obtenue en inversant $\ZZ\{1\}$ pour le produit tensoriel.
Un objet de $\DMgm$ est un couple $(C,n)$ où $C$ est un complexe de $\smc$
et $n$ un entier, noté suggestivement $C\{n\}$. 
Les morphismes sont définis par la formule
$$
\Hom_{\DMgm}(C\{n\},D\{m\})
 =\ilim{r \geq -n,-m} \Hom{\DMgme}(C\{r+n\},D\{r+n\}).
$$
Cette catégorie est de manière évidente équivalente à la catégorie définie 
dans \cite{V1} obtenue en inversant le motif de Tate $\ZZ(1)=\ZZ\{1\}[-1]$. 
Elle est donc triangulée monoïdale symétrique.

Considérons maintenant un faisceau homotopique $(F_*,\epsilon_*)$.
Pour tout motif géométrique $C\{n\}$, on pose 
$$
\varphi(C\{n\})=\ilim{r\geq -n} \Hom_{D(\ftr)}(\rep C\{r+n\},F_r)
$$
où les morphismes de transitions sont
\begin{align*}
\Hom_{D(\ftr)}(\rep C\{r+n\},F_r)
 &\xrightarrow{\epsilon_{r*}} \Hom_{D(\ftr)}(\rep C\{r+n\},(F_{r+1})_{-1}) \\
 &=\Hom_{D(\ftr)}(\rep C\{r+n+1\},F_{r+1}).
\end{align*}
Comme dans le cas des motifs effectifs, ceci induit un foncteur
de \emph{réalisation cohomologique} associé à $(F_*,\epsilon_*)$~:
$$\varphi:\DMgm^{op} \rightarrow \ab$$
Notons que ce foncteur est naturellement gradué
$\varphi_n(\rep C\{r\})=\varphi(\rep C\{r-n\})$
de sorte que,
 d'après le théorème de simplification \ref{cancellation},
 pour tout schéma lisse $X$, $\varphi_n(\rep X)=F_n(X)$

\rem Ce foncteur obtiendra une interprétation plus naturelle 
dans la partie \ref{part:motifs}. 
Notons qu'il résulte du théorème de simplification 
\ref{cancellation} que $\varphi(M(X)\{n\})=F_{-n}(X)$.

\section{Modules de cycles}

\subsection{Rappels}

Rappelons brièvement la théorie des modules de cycles sur $k$ due
à M.~Rost.
Tous ces rappels concernent la théorie telle qu'elle est exposée
dans \cite{Ros}. Toutefois, nous nous référons à \cite{Deg4}
pour des rappels plus détaillés car nous aurons à utiliser
les résultats supplémentaires qui y sont démontrés. \\
Un \emph{pré-module} de cycles $\phi$ (\emph{cf.} \cite[1.1]{Deg4})
 est la donnée pour tout corps de fonctions $E$  
 d'un groupe abélien $\ZZ$-gradué $\phi(E)$
  satisfaisant à la fonctorialité suivante~:
\begin{enumerate}
\item[\textbf{(D1)}] Pour toute extension de corps $f:E \rightarrow L$,
on se donne un morphisme appelé \emph{restriction}
$f_*:\phi(E) \rightarrow \phi(L)$ de degré $0$.
\item[\textbf{(D2)}] Pour toute extension finie de corps $f:E \rightarrow L$,
on se donne un morphisme appelé \emph{norme}
$f^*:\phi(L) \rightarrow \phi(E)$ de degré $0$.
\item[\textbf{(D3)}] Pour tout élément $\sigma \in K_r^M(E)$ du $r$-ième groupe
de $K$-théorie de Milnor de $E$, on se donne un morphisme
$\gamma_\sigma:\phi(E) \rightarrow \phi(E)$ de degré $r$.
\item[\textbf{(D4)}] Pour tout corps de fonctions valué $(E,v)$,
on se donne un morphisme appelé \emph{résidu}
$\partial_v:\phi(E) \rightarrow \phi(\kappa(v)))$ de degré $-1$.
\end{enumerate}
Considérant ces données, on introduit fréquemment un cinquième type
de morphisme, associé à un corps de fonctions valué $(E,v)$ et à une
uniformisante $\pi$ de $v$, de degré $0$,
$s_v^\pi=\partial_v \circ \gamma_\pi$, appelé \emph{spécialisation}. \\
Ces données sont soumises à un ensemble de relations 
 (\emph{cf.} \cite[par 1.1]{Deg4}). On peut se faire une idée de ces relations
en considérant le foncteur de K-théorie de Milnor qui est l'exemple
le plus simple de pré-module de cycles.

Considérons un schéma $X$ essentiellement de type fini sur $k$.
Soit $x,y$ deux points de $X$. Soit $Z$ l'adhérence réduite de $x$
dans $X$, $\tilde Z$ sa normalisation et $f:\tilde Z \rightarrow Z$
le morphisme canonique.
Supposons que $y$ est un point de codimension $1$ dans $Z$
et notons $\tilde Z_y^{(0)}$ l'ensemble des points génériques
de $f^{-1}(y)$.
Tout point $z \in \tilde Z_y^{(0)}$
correspond alors à une valuation $v_z$ sur $\kappa(x)$ de corps
résiduel $\kappa(z)$. 
On note encore $\varphi_z:\kappa(y) \rightarrow \kappa(z)$
 le morphisme induit par $f$.
On définit un morphisme $\partial_y^x:\phi(\kappa(x)) \rightarrow \phi(\kappa(y))$
par la formule suivante~:
$$
\partial_y^x=
\begin{cases}
\sum_{z \in \tilde Z_y^{(0)}} \varphi_z^* \circ \partial_{v_z}
 & \text{si } y \in Z^{(1)}, \\
0 & \text{sinon.}
\end{cases}
$$
Considérons ensuite le groupe abélien~:
$$
C^p(X;\phi)=\bigoplus_{x \in X^{(p)}} \phi(\kappa(x)).
$$
On dit que le pré-module de cycles $\phi$ est un \emph{module de cycles}
(\emph{cf.} \cite[1.3]{Deg4})
si pour tout schéma essentiellement de type fini $X$,
\begin{enumerate}
\item[\textbf{(FD)}] Le morphisme
$$
d^p_{X,\phi}:\sum_{x \in X^{(p)}, y \in X^{(p+1)}}
 \partial_y^x:C^p(X;\phi) \rightarrow C^{p+1}(X;\phi)
$$
est bien définit.
\item[\textbf{(C)}] La suite
$$
\hdots \rightarrow C^p(X;\phi) \xrightarrow{d_{X,\phi}^p} C^{p+1}(X;\phi)
  \rightarrow \hdots
$$
est un complexe.
\end{enumerate}
Les modules de cycles forment de manière évidente une catégorie
 que l'on note $\modl$.

On introduit une graduation sur le complexe de la propriété (C)~:
$$
C^p(X;\phi)_n=\bigoplus_{x \in X^{(p)}} M_{n-p}(\kappa(x)).
$$
On note $A^p(X;\phi)_n$ le $p$-ième groupe de cohomologie de ce complexe,
 appelé parfois \emph{groupe de Chow à coefficients dans $\phi$}.

Pour un schéma lisse $X$ de corps des fonctions $E$,
 le groupe $A^0(X;\phi)_n$ est donc le noyau de l'application bien
 définie
$$
\phi_n(E) \xrightarrow{\sum_{x \in X^{(1)}} \partial_x} \phi_{n-1}(\kappa(x))
$$
où $\partial_x$ désigne le morphisme résidu associé à la valuation sur
$E$ correspondant au point $x$.

\subsection{Fonctorialité}

\num \label{Gersten&plat_propre}
Le complexe gradué $C^*(X;\phi)_*$ est contravariant en $X$ par rapport
aux morphismes plats (\emph{cf.} \cite[par. 1.3]{Deg4}). 
Il est covariant par rapport aux morphismes propres équidimensionnels
(\emph{loc. cit.}).

\num Dans \cite[3.18]{Deg4}, nous avons prolongé le travail original de Rost
et nous avons associé à tout morphisme $f:Y \rightarrow X$
localement d'intersection complète (\cite[3.12]{Deg4})
tel que $Y$ est lissifiable (\cite[3.13]{Deg4})
un \emph{morphisme de Gysin}
$$
f^*:C^*(X;\phi) \rightarrow C^*(Y;\phi)
$$
qui est un composé de morphisme de complexe et d'inverse formel
d'un morphisme de complexes qui est un quasi-isomorphisme (plus précisément,
il s'agit de l'inverse formel d'un morphisme $p^*$ pour $p$ la projection
d'un fibré vectoriel). Pour désigner une telle flèche formelle,
on utilise la notation $f^*:X \doto Y$.

Ce morphisme de Gysin $f^*$ satisfait les propriétés suivantes~:
\begin{enumerate}
\item Lorsque $f$ est de plus plat, $f^*$ coincide avec le pullback plat
évoqué plus haut.
\item Si $g:Z \rightarrow Y$ est un morphisme localement d'intersection complète
avec $Z$ lissifiable, $(fg)^*=g^* f^*$.
\end{enumerate}

Dans le cas où $f$ est une immersion fermée régulière, l'hypothèse que $Y$
est lissifiable est inutile ;
le morphisme $f^*$ est défini en utilisant
la déformation au cône normal,
 suivant l'idée originale de Rost (\emph{cf.} \cite[3.3]{Deg4}).
On utilisera par ailleurs le résultat suivant dû à Rost (\cite[(12.4)]{Ros})
qui décrit partiellement ce morphisme de Gysin~:
\begin{prop} \label{Gysin&specialisation}
Soit $X$ un schéma intègre de corps des fonctions $E$,
 et $i:Z \rightarrow X$ l'immersion fermée d'un diviseur
  principal régulier irréductible paramétré par $\pi \in \mathcal O_X(X)$.
Soit $v$ la valuation de $E$ correspondant au diviseur $Z$.
Alors, le morphisme
$i^*:A^0(X;\phi) \rightarrow A^0(Z;\phi)$
est la restriction de
$s_v^\pi:\phi(E) \rightarrow \phi(\kappa(v))$.
\end{prop}

\num \label{Gysin_raffine}
A tout carré cartésien
$$
\xymatrix@=10pt{
Y'\ar^j[r]\ar_g[d]\ar@{}|\Delta[rd] & X'\ar^f[d] \\
Y\ar_i[r] & X
}
$$
tel que $i$ est une immersion fermée régulière,
on associe un \emph{morphisme de Gysin raffiné}
$\Delta^*:X' \doto Y'$.
Ce morphisme $\Delta^*$ vérifie les propriétés suivantes~:
\begin{enumerate}
\item Si $j$ est régulière et le morphisme des cônes
normaux $N_{Y'}(X') \rightarrow g^{-1} N_Y(X)$
est un isomorphisme, $\Delta^*=j^*$.
\item Si $f$ est propre, $i^*f_*=g_*\Delta^*$.
%
\end{enumerate}
De plus, si l'immersion canonique 
$C_{Y'}(X') \rightarrow g^{-1} N_Y(X)$
du cône de $j$ dans le fibré normal de $i$
 est de codimension pure égale à $e$,
le morphisme $\Delta^*$ est de degré cohomologique $e$.

\num \label{transferts_complexe_Rost}
Pour tout couple de schémas lisses $(X,Y)$ et 
 pour toute correspondance finie $\alpha \in c(X,Y)$,
  on définit un morphisme $\alpha^*:Y \doto X$ (\emph{cf.} \cite[6.9]{Deg4}).
  
On peut décrire ce dernier comme suit.
Supposons que $\alpha$ est la classe d'un sous-schéma fermé irréductible
$Z$ de $X \times Y$. Considèrons les morphismes:
$$
X \stackrel p \leftarrow Z \xrightarrow i Z \times X \times Y
\xrightarrow q Y
$$
où $p$ et $q$ désignent les projections canoniques et $i$ le
graphe de l'immersion fermée $Z \rightarrow X \times Y$.
Alors,
\begin{equation} \label{transferts_Chow}
\alpha^*=p_* i^* q^*
\end{equation}
où $i^*$ désigne le morphisme de Gysin de l'immersion fermée
régulière $i$, 
$q^*$ le pullback plat
 et $p_*$ le pushout fini.

La propriété $(\beta \alpha)^*=\alpha^* \beta^*$ est démontrée dans
\cite[6.5]{Deg4}.

\subsection{Suite exacte de localisation}
\label{sec:se_loc_modl}

La suite exacte de localisation n'est pas étudiée (ni rappelée) dans \cite{Deg4}.
Nous la rappelons maintenant suivant \cite{Ros}
 et démontrons un résultat supplémentaire concernant sa fonctorialité. \\
Pour une immersion fermée $i:Z \rightarrow X$ purement de codimension $c$,
d'immersion ouverte complémentaire $j:U \rightarrow X$,
on obtient en utilisant la fonctorialité rappelée ci-dessus
 une suite exacte courte scindée de complexes
\begin{equation} \label{eq:se_loc_modl}
0 \rightarrow C^{p-c}(Z;\phi)_{n-c} \xrightarrow{i_*} C^p(X;\phi)_n
 \xrightarrow{j^*} C^p(U;\phi)_n \rightarrow 0.
\end{equation}
On en déduit une suite exacte longue de localisation
\begin{equation} \label{localisation2}
\hdots \rightarrow A^{p-c}(Z;\phi)_{n-c} \xrightarrow{i_*} A^p(X;\phi)_n
 \xrightarrow{j^*} A^p(U;\phi)_n
  \xrightarrow{\partial_Z^U} A^{p-c+1}(Z;\phi)_{n-c} \rightarrow \hdots
\end{equation}
où le morphisme $\partial_Z^U$ est définit au niveau des complexes
par la formule $\sum_{x \in U^{(p)}, z \in Z^{(p-c+1)}} \partial_z^x$.

Cette suite est naturelle par rapport au pushout propre
et au pullback plat. La proposition suivante est nouvelle~:
\begin{prop} \label{localisation&Gysin}
Considérons un carré cartésien
$$
\xymatrix@=10pt{
T\ar^{\iota'}[r]\ar_k[d]\ar@{}|\Delta[rd] & Z\ar^i[d] \\
Y\ar_\iota[r] & X
}
$$
tel que $\iota$ est une immersion fermée régulière.
Supposons que $i$ (resp. $k$) est une immersion fermée
d'immersion ouverte complémentaire $j:U \rightarrow X$
(resp. $l:V \rightarrow X$).
Notons $h:V \rightarrow U$ le morphisme induit par $\iota$.
Supposons enfin que $i$ (resp. $k$) est
de codimension pure égale à $c$ (resp. $d$).
Alors, le diagramme suivant est commutatif~:
$$
\xymatrix@R=10pt@C=12pt{
\hdots\ar[r] &  A^{p-c}(Z;\phi)_{n-c}\ar^-{i_*}[r]\ar^{\Delta^*}[d]
 & A^p(X;\phi)_n\ar^-{j^*}[r]\ar^{\iota^*}[d]
  & A^p(U;\phi)_n\ar^-{\partial_Z^U}[r]\ar^{h^*}[d]
   & A^{p-c+1}(Z;\phi)_{n-c}\ar[r]\ar^{\Delta^*}[d] & \hdots \\
\hdots\ar[r] &  A^{p-d}(T;\phi)_{n-d}\ar^-{k_*}[r]
 & A^p(Y;\phi)_n\ar^-{l^*}[r]
  & A^p(V;\phi)_n\ar^-{\partial_T^V}[r]
   & A^{p-d+1}(T;\phi)_{n-d}\ar[r] & \hdots
}
$$
\end{prop}

\rem 
\begin{enumerate}
\item On peut généraliser la proposition précédente au cas des morphismes
de Gysin raffinés comme dans la proposition 4.5 de \cite{Deg4}.
Nous laissons au lecteur le soin de formuler cette généralisation.
\item Alors que l'hypothèse sur la codimension pure de $i$ est naturelle,
 celle sur $k$ ne l'est pas, en particulier dans un cas non transverse. 
Elle ne nous sert qu'à exprimer les degrés cohomologiques de tous les 
morphismes et peut aisément être supprimée si on accepte des morphismes
non homogènes par rapport au degré cohomologique.
\end{enumerate}

\begin{proof}
Il suffit de reprendre la preuve de la proposition 4.5
de {\it loc. cit.}
dans le cas du diagramme commutatif~:
$$
\xymatrix@=9pt{
T\ar^{}[rrr]\ar[rdd]\ar^/6pt/k[rrd]
 &&& Z\ar^i[rrd]\ar[rdd] && \\
&& Y\ar^/-20pt/{}[rrr]\ar@{=}[ld] &&& X\ar@{=}[ld] \\
 & Y\ar_\iota[rrr] &&& X. & \\
}
$$
On obtient ainsi un diagramme commutatif\footnote{Il y a une
faute de frappe dans le diagramme commutatif de \emph{loc. cit.}
Il faut lire $t^*N_ZX$ au lieu de $N_YX$.}, avec
les notations analogues de \emph{loc. cit.}
$$
\xymatrix@R=16pt@C=22pt{
Z\ar@{{*}->}^{\sigma'}[r]
          \ar@{{*}->}_{i_*}[d]\ar@{}|{(1)}[rd]
 & C_TZ\ar@{{*}->}^{\nu'_*}[r]
          \ar@{{*}->}^{k''_*}[d]\ar@{}|{(2)}[rd]
 & k^*N_YX
          \ar_{k'_*}[d]\ar@{}|{(3)}[rd]
 & T\ar@{{*}->}^{k_*}[d]\ar@{{*}->}_-{p_T^*}[l] \\
X\ar@{{*}->}_{\sigma}[r]
          \ar@{{*}->}_{j^*}[d]\ar@{}|{(1')}[rd]
 & N_YX\ar@{=}[r]
          \ar@{{*}->}^{l'^*}[d]\ar@{}|{(2')}[rd]
 & N_YX
          \ar_{l'^*}[d]\ar@{}|{(3')}[rd]
 & Y\ar@{{*}->}|/-6pt/{p^*}[l]\ar@{{*}->}^{l^*}[d] \\
U\ar@{{*}->}_{\sigma_U}[r]
 & N_UV\ar@{=}[r] 
 & N_UV & U\ar@{{*}->}^-{p_U^*}[l].
}
$$
Les carrés (1), (2), (3) sont commutatifs d'après \emph{loc. cit.}
et les carrés (1'), (2'), (3') le sont pour des raisons triviales.
Les flèches $\doto$ qui apparaissent dans ce diagramme
sont bien des morphismes de complexes et induisent donc des morphismes
de suite exacte longue de localisation. Il suffit alors d'appliquer
le fait que les morphismes $p^*$, $p_T^*$ et $p_U^*$ sont 
des quasi-isomorphismes pour conclure.
\end{proof}

\begin{cor} \label{cor:fonctorialite_loc_transverse}
Considérons un carré cartésien
$$
\xymatrix@=10pt{
T\ar^{g}[r]\ar_k[d]\ar@{}|\Delta[rd] & Z\ar^i[d] \\
Y\ar_f[r] & X
}
$$
de schémas lisses tels que $i$ (resp. $k$)
est une immersion fermée de codimension pure égale à $c$,
d'immersion ouverte complémentaire $j:U \rightarrow X$
(resp. $l:V \rightarrow X$).
Notons $h:V \rightarrow U$ le morphisme induit par $f$.
Alors, le diagramme suivant est commutatif~:
$$
\xymatrix@R=10pt@C=12pt{
\hdots\ar[r] &  A^{p-c}(Z;\phi)_{n-c}\ar^-{i_*}[r]\ar^{g^*}[d]
 & A^p(X;\phi)_n\ar^-{j^*}[r]\ar^{f^*}[d]
  & A^p(U;\phi)_n\ar^-{\partial_Z^U}[r]\ar^{h^*}[d]
   & A^{p-c+1}(Z;\phi)_{n-c}\ar[r]\ar^{g^*}[d] & \hdots \\
\hdots\ar[r] &  A^{p-c}(T;\phi)_{n-c}\ar^-{k_*}[r]
 & A^p(Y;\phi)_n\ar^-{l^*}[r]
  & A^p(V;\phi)_n\ar^-{\partial_T^V}[r]
   & A^{p-c+1}(T;\phi)_{n-c}\ar[r] & \hdots
}
$$
\end{cor}

\rem \label{rem:rappels_pf}
 Dans l'article \cite{Deg5},
 une \emph{paire fermée} est un couple $(X,Z)$ tel que $X$ est un schéma
 lisse et $Z$ un sous-schéma fermé.
 On dit que $(X,Z)$ est lisse (resp. de codimension $n$)
 si $Z$ est lisse (resp. purement de codimension $n$ dans $X$). \\
Si $i:Z \rightarrow X$ est l'immersion fermée associée,
 un \emph{morphisme de paires fermées} $(f,g)$ est un carré commutatif
$$
\xymatrix@=10pt{
T\ar^{g}[r]\ar_k[d] & Z\ar^i[d] \\
Y\ar_f[r] & X
}
$$
qui est topologiquement cartésien.
On dit que $(f,g)$ est \emph{cartésien}
(resp. \emph{transverse}) quand le carré est cartésien 
(resp. et le morphisme induit sur les cônes normaux
$C_TY \rightarrow g^{-1}C_ZX$ est un isomorphisme).
\footnote{Lorsque $(X,Z)$ est lisse de codimension $n$
 le fait que le morphisme $(f,g)$ est transverse
 entraîne que $(Y,T)$ est lisse de codimension $n$
  ($k$ est régulier).} \\
Le corollaire précédent montre que la suite de localisation
associée à un module de cycles $\phi$ et une paire
fermée $(X,Z)$ est naturelle par rapport aux morphismes
transverses.
 
\subsection{Module homotopique associé}

\num \label{mcycl->mhtp}
Considérons un module de cycles $\phi$.
D'après \ref{transferts_complexe_Rost},
 $A^0(.;\phi)_*$ définit un préfaisceau gradué avec transferts. 
D'après \cite[6.9]{Deg4}, c'est un faisceau homotopique gradué. 
On le note $F^\phi_*$
et on lui définit une structure de module homotopique comme suit: \\
Soit $X$ un schéma lisse. 
On considère le début de la suite exacte longue de localisation 
\eqref{localisation2}
associée à la section nulle $X \rightarrow \AA^1_X$~:
$$
0 \rightarrow F^\phi_n(\AA^1_X) \xrightarrow{j^*_X} F^\phi_n(\GG \times X)
 \xrightarrow{\partial_0^X} F^\phi_{n-1}(X) \rightarrow \hdots
$$
On peut décrire le morphisme $\partial_0^X$ si $X$ est connexe de corps
des fonctions $E$ comme étant induit par le morphisme
$$
\partial_0^E:\phi_n(E(t)) \rightarrow \phi_{n-1}(E)
$$
associé à la valuation standard de $E(t)$. \\
Soit $s_1:X \rightarrow \GG \times X$ la section unité.
Rappelons que $(F^\phi_n)_{-1}(X)=\Ker(s_1^*)$.
Or par invariance par homotopie de $F^\phi_n$,
 le morphisme canonique $\Ker(s_1^*) \rightarrow\coKer(j^*)$ est un isomorphisme.
Ainsi, le morphisme $\partial_0^X$ induit un morphisme
$$
\epsilon_{n,X}:(F^\phi_n)_{-1}(X) \rightarrow F^\phi_{n-1}(X).
$$
On vérifie que la suite de localisation précédente est compatible
 aux transferts en $X$,
comme cela résulte de la description des transferts rappelée en
\ref{transferts_complexe_Rost} et du corollaire
\ref{cor:fonctorialite_loc_transverse}.
Ainsi, $\epsilon_n$ définit un morphisme de faisceaux homotopiques.
Pour tout corps de fonctions $E$, $A^1(\AA^1_E;\phi)=0$ 
 (\emph{cf.} \cite[(2.2)(H)]{Ros}). Donc la fibre de $\epsilon_n$ en $E$
  est un isomorphisme ce qui implique que c'est un
   isomorphisme de faisceaux homotopiques d'après \ref{foncteurs_fibres_fhtp}.

Ainsi, $(F^\phi_*,\epsilon^{-1}_*)$ définit un module homotopique qui dépend
fonctoriellement de $\phi$.

\section{Equivalence de catégories}

\subsection{Transformée générique}

Considérons un couple $(E,n)$
 formé d'un corps de fonctions $E$ et d'un entier relatif $n$.
Rappelons que l'on a associé dans \cite[3.3.1]{Deg5} au couple $(E,n)$
 un \emph{motif générique}
$$
M(E)\{n\}=\pplim{A \subset E} M(\spec A)\{n\}
$$
dans la catégorie des pro-objets de $\DMgm$.
On note $\DMgmo$ la catégorie des motifs génériques.

\num \label{mhtp->mcycl}
Considérons un module homotopique $(F_*,\epsilon_*)$ ainsi
que le foncteur de réalisation $\varphi:\DMgm^{op} \rightarrow \ab$
qui lui est associé dans la section \ref{realisation_mhtp}.
On note $\hat \varphi$ le prolongement évident de $\varphi$
à la catégorie des pro-objets.
Il résulte de \cite[6.2.1]{Deg5} que la restriction de $\hat \varphi$
à la catégorie $\DMgmo$ est un module de cycles, que l'on note
$\hat F_*$ et que l'on appelle la \emph{transformée générique} de $F_*$.

Rappelons brièvement certaines parties de la construction de \cite{Deg5}.
Notons d'abord que pour tout motif générique $M(E)\{n\}$,
$\hat \varphi(M(E)\{n\})=\hat F_{-n}(E)$ n'est autre que la fibre de $F_{-n}$
en $E$ (\emph{cf.} \ref{foncteurs_fibres_fhtp}). La transformée $\hat F_*$
s'interprète donc comme le \emph{système des fibres} de $F_*$. Ce sont
les \emph{morphismes de spécialisation} entre ces fibres qui donnent
la structure de pré-module de cycles~:
\begin{enumerate}
\item[(D1)] Fonctorialité évidente de $F_*$.
\item[(D2)] (\cite[5.2]{Deg5}) Pour une extension finie $L/E$,
on trouve des modèles respectifs $X$ et $Y$ de $E$ et $L$ ainsi
qu'un morphisme fini surjectif $f:Y \rightarrow X$ dont l'extension
induite des corps de fonctions est isomorphe à $L/E$.
Le graphe de $f$ vu comme cycle de $X \times Y$ définit une correspondance
finie de $X$ vers $Y$ notée $\tra f$ -- la \emph{transposée} de $f$.
On en déduit un morphisme $(\tra f)^*:F_*(X) \rightarrow F_*(Y)$.
On montre que ce morphisme est compatible à la restriction à
 un ouvert de $X$ et il induit donc la fonctorialité attendue.
\item[(D3)] (\cite[5.3]{Deg5}) Soit $E$ un corps de fonctions
 et $x \in E^{\times}$ une unité.
Considérons un modèle $X$ de $E$ munit d'une section inversible
$X \rightarrow \GG$ qui correspond à $x$.
Considérons l'immersion fermée $s_x:X \rightarrow \GG \times X$ induite
par cette section. On en déduit un morphisme
$$
\gamma_x:F_{n-1}(X) \xrightarrow{\epsilon_{n-1}} (F_{n})_{-1}(X)
 \xrightarrow \nu F_{n}(\GG \times X) \xrightarrow{s_x^*}
  F_{n}(X)
$$
où $\nu$ est l'inclusion canonique.
Ce morphisme est compatible à la restriction suivant un ouvert de
$X$ et induit la donnée D3 pour $\hat F_*$.
\item[(D4)] (\cite[5.4]{Deg5}) Soit $(E,v)$ un corps de fonctions valué.
On peut trouver un schéma lisse $X$ munit d'un point $x$ de codimension $1$
tel que l'adhérence réduite $Z$ de $x$ dans $X$ est lisse et l'anneau local
$\mathcal O_{X,x}$ est isomorphe à l'anneau des entiers de $v$.
On pose $U=X-Z$, $j:U \rightarrow X$ l'immersion ouverte évidente. 
Rappelons que le motif relatif $M_Z(X)$ de la paire $(X,Z)$
est définie comme l'objet de $\DMgme$ représenté par le complexe
concentré en degré $0$ et $-1$ avec pour seule différentielle non
nulle le morphisme $j$. Ce motif relatif s'inscrit naturellement
dans le triangle distingué
$$
M_Z(X)[-1] \xrightarrow{\partial'_{X,Z}}
M(U) \xrightarrow{j_*} M(X) \xrightarrow{+1}
$$
On a définit dans \cite[sec. 2.2.5]{Deg5} un \emph{isomorphisme de pureté}
$$
\mathfrak p_{X,Z}:M_Z(X) \rightarrow M(Z)(1)[2].
$$
On en déduit un morphisme
\begin{align*}
\partial_{X,Z}:F_{n}(U)&=\varphi_n(M(U))
 \xrightarrow{\varphi_n(\partial'_{X,Z})} \varphi_n(M_Z(X)[-1]) \\
& \xrightarrow{(\varphi_n(\mathfrak p_{X,Z}^{-1})} \varphi_n(M(Z)\{1\})=(F_n)_{-1}(Z)
 \xrightarrow{\epsilon_n^{-1}} F_{n-1}(Z),
\end{align*}
ayant posé $\varphi_n(\cM)=\varphi(\cM\{-n\})$ pour un motif $\cM$.
Le morphisme résidu du module de cycles $\hat F_*$ est donné par la limite inductive
des morphismes $\partial_{U,Z \cap U}$ suivant les voisinages ouverts $U$ de $x$ dans
$X$.
\end{enumerate}

\subsection{Théorème et démonstration}

\num \label{iso_Rost}
Considérons un module de cycles $\phi$ et $X$ un schéma lisse.
D'après \cite[6.5]{Ros}, on dispose pour tout entier $n \in \ZZ$ 
d'un isomorphisme canonique
$A^p(X;\phi)=H^p_\mathrm{Zar}(X;F^\phi)$. \\
On rappelle la construction de cet isomorphisme tout en le généralisant
au cas de la topologie Nisnevich. 
Notons $X_\nis$ le petit site Nisnevich de $X$.
Les morphismes de $X_\nis$ étant étales,
on obtient, 
en utilisant la fonctorialité rappelée dans \ref{Gersten&plat_propre},
un préfaisceau $V/X \mapsto C^*(V;\phi)$ sur $X_\nis$
noté $C^*_X(\phi)$.
On vérifie que c'est un faisceau Nisnevich
 (voir \cite{Deg5}, preuve de 6.10).
On note $F^\phi_X$ le faisceau $V/X \mapsto A^0(V;\phi)$ sur $X_\nis$.
D'après \cite[6.1]{Ros}, le morphisme évident
$F^\phi_X \rightarrow C^*_X(\phi)$ est un quasi-isomorphisme.
Il induit donc un isomorphisme
$$
H^p_\nis(X;F^\phi_X) \rightarrow H^p_\nis(X;C^*_X(\phi)).
$$
Notons par ailleurs que le complexe $C^*_X(\phi)$ vérifie la propriété
de Brown-Gersten au sens de \cite[1.1.9]{CD2} (voir à nouveau \cite{Deg5}, 
 preuve de 6.10).
D'après la démonstration de \cite[1.1.10]{CD2}, on en déduit que le morphisme
canonique
$$
H^p(C^*(X;\phi)) \rightarrow H^p_\nis(X;C^*_X(\phi))
$$
est un isomorphisme. Ces deux isomorphismes définissent comme annoncé~:
\begin{equation} \label{eq:iso_gpe_chow&coh_nis}
\rho_X:A^p(X;\phi) \xrightarrow \sim H^p_\nis(X;F^\phi).
\end{equation}
Notons par ailleurs que $\rho_X$ est naturel
 par rapport aux morphismes de schémas. 
Considérons d'abord le cas d'un morphisme plat $f:Y \rightarrow X$
 de schémas lisses. Dans ce cas, on déduit suivant \ref{Gersten&plat_propre} 
 un morphisme de complexes
$$f^*:C^*(X;\phi) \rightarrow C^*(Y;\phi)$$
qui est naturel en $X$ par rapport aux morphismes étales. 
La transformation naturelle sur $X_\nis$ correspondante
définit un morphisme dans la catégorie dérivée des faisceaux abéliens
sur $X_\nis$:
$$
\eta_f:C_X^*(\phi) \rightarrow f_*C_Y^*(\phi))=\derR f_*C_Y^*(\phi).
$$
(La dernière identification résulte du fait que
 $C_Y^*(\phi)$ vérifie la propriété de Brown-Gersten.)
Par ailleurs, la structure de faisceau sur $\sm$ de $F^\phi$ définit
 une transformation naturelle
$F_X^\phi \rightarrow f_*F_Y^\phi$
qui se dérive (quitte à prendre une résolution injective de $F^\phi$)
et induit une transformation naturelle dans la catégorie dérivée
des faisceaux abéliens sur $X_\nis$
$$
F^\phi_X \xrightarrow{\tau_f} \derR f_*(F^\phi).
$$
Par définition de la structure de faisceau sur $F^\phi$,
 le diagramme suivant est commutatif:
$$
\xymatrix@=12pt{
F^\phi_X\ar[r]\ar_{\tau_f}[d] & C^*_X(\phi)\ar^{\eta_f}[d] \\
\derR f_*F^\phi_Y\ar[r] & \derR f_*C^*_Y(\phi).
}
$$
On en déduit la naturalité de $\rho$ par rapport aux morphismes
plats. \\
Remarquons que si $f$ est la projection d'un fibré vectoriel,
 $\eta_f$ est un quasi-isomorphisme.
Il reste à considérer le cas d'une immersion fermée $i:Z \rightarrow X$
entre schémas lisses. Notons $N$ le fibré normal associé à $i$.
La spécialisation au fibré normal définie par Rost
(\emph{cf.} \cite[2.1]{Deg4}) est un morphisme de complexes
$$\sigma_ZX:C^*(X;\phi) \rightarrow C^*(N;\phi)
$$
qui est de plus naturel en $X$ par rapport aux morphismes étales
 (\emph{cf.} \cite[2.2]{Deg4}). Notons $\nu$ le morphisme composé
$$
N \xrightarrow p Z \xrightarrow i X.
$$
On en déduit dans la catégorie dérivée un morphisme canonique
$$
\sigma_i:C_X^*(\phi) \rightarrow \derR \nu_*C_N^*(\phi).
$$
Puisque le morphisme $\eta_p$ est un quasi-isomorphisme,
 on obtient alors un morphisme canonique dans la catégorie
 dérivée
$$
\eta_i:C_X^*(\phi) \rightarrow \derR i_*C_Z^*(\phi).
$$
Comme précédemment, on vérifie que le morphisme $\rho_i$
est compatible avec la transformation naturelle 
$\tau_i:F_X^\phi \rightarrow \derR i_*F_Z^\phi$ induite
par la structure de faisceau sur $\sm$ de $F^\phi$,
ce qui permet d'obtenir la fonctorialité de $\rho$
par rapport aux immersions fermées.

Considérons par ailleurs le foncteur de réalisation 
$$
\varphi:\DMgm^{op} \rightarrow \ab
$$
associé au module homotopique $F^\phi$
suivant la section \ref{realisation_mhtp}.
L'isomorphisme $\rho_X$ correspond par définition à un isomorphisme:
$$
A^p(X,\phi)_n \rightarrow \varphi_n(M(X)[-p]).
$$
Considérons de plus une immersion fermée $i:Z \rightarrow X$ entre schémas lisses
 et $j:U \rightarrow X$ l'immersion ouverte du complémentaire.
 Supposons que $i$ est de codimension pure égale à $c$.
 On déduit de la suite exacte de localisation \eqref{eq:se_loc_modl}
 une unique flèche pointillée qui fait commuter le diagramme de
 complexes suivant (on utilise à nouveau le fait que $C_X^*(\phi)$
 vérifie la propriété de Brown-Gersten):
$$
\xymatrix@R=20pt{
0\ar[r] & C^*(Z,\phi)_{n-c}[-c]\ar^-{i_*}[r]\ar@{-->}_{(1)}[d] 
 & C^*(X,\phi)_n\ar^{j^*}[r]\ar[d]
 & C^*(U,\phi)_n\ar[r]\ar[d]
 & 0 \\
0\ar[r] & \derR \Gamma_Z(X,C_X^*(\phi))_{n}\ar[r]
 & \derR \Gamma(X,C_X^*(\phi))_n\ar^{j^*}[r]
 & \derR \Gamma(U,C_X^*(\phi))_n\ar[r]
 & 0.
}
$$
La flèche (1) est un quasi-isomorphisme,
 puisqu'il en est de même des deux autres flèches verticales.
Considérons le motif relatif $M_Z(X)$ associé la paire fermée $(X,Z)$
-- \emph{cf.} \ref{mhtp->mcycl}, (D4).
En utilisant l'isomorphisme (1) et l'identification canonique $H^p_Z(X;F^\phi)_n=\phi_n(M_Z(X)[-p])$,
 on obtient un diagramme commutatif:
$$
\xymatrix@R=20pt{
A^{p-1}(U,\phi)_{n}\ar^-{\partial_Z^U}[r]\ar_{\rho_U}[d]
 & A^{p-c}(Z,\phi)_{n-c}\ar^-{i_*}[r]\ar^{\rho'_{X,Z}}[d] 
 & A^p(X,\phi)_{n}\ar^{\rho_X}[d] \\
\varphi_n(M(U)[-p])\ar[r]
 & \varphi_n(M_Z(X)[-p])\ar[r]
 & \varphi_n(M(X)[-p])
}
$$
dans lequel les flèches verticales sont des isomorphismes.
Le morphisme $\rho'_{X,Z}$ est de plus naturel en $(X,Z)$
par rapport aux morphismes transverses (définis en \ref{rem:rappels_pf}).
Cela résulte en effet du corollaire \ref{cor:fonctorialite_loc_transverse},
 ou plus précisément
 du diagramme commutatif apparaissant
 dans la démonstration de \ref{localisation&Gysin},
 en utilisant d'une part l'unicité de la flèche pointillée (1)
 et d'autre part la description de la fontorialité dérivée de $C_X^*(\phi)$
 établie ci-dessus -- \emph{i.e.} les transformations naturelles $\tau_f$ et $\tau_i$. \\
Comme conséquence de cette construction, on obtient le lemme clé suivant:
\begin{lm} \label{lm:comp_Gersten_loc}
Reprenons les hypothèses introduites ci-dessus.
Considérons le triangle de Gysin (\emph{cf.} \cite[2.3.1]{Deg5})
 associé à $(X,Z)$:
$$
M(U) \rightarrow M(X) \xrightarrow{i^*} M(Z)(c)[2c]
 \xrightarrow{\partial_{X,Z}} M(U)[1].
$$
Alors, le diagramme suivant est commutatif:
$$
\xymatrix@R=12pt@C=40pt{
A^{p-1}(U,\phi)_{n}\ar^-{\partial_Z^U}[r]\ar_{\rho_U}[dd]
 & A^{p-c}(Z,\phi)_{n-c}\ar^-{i_*}[r]\ar^{\rho_{Z}}[d] 
 & A^p(X,\phi)_{n}\ar^{\rho_X}[dd] \\
 & \varphi_{n-c}(M(Z)[c-p])\ar@{=}[d]
 & \\
\varphi_n(M(U)[-p])\ar^-{\varphi_n(\partial_{X,Z})}[r]
 & \varphi_{n}(M(Z)(c)[2c-p])\ar^-{\varphi_n(i^*)}[r]
 & \varphi_n(M(X)[-p]).
}
$$
\end{lm}
\begin{proof}
Considérons l'isomorphisme de pureté définit dans \cite[sec. 2.2.5]{Deg5}
$$
\mathfrak p_{X,Z}:M_Z(X) \rightarrow M(Z)(c)[2c].
$$
Dès lors, d'après ce qui précède, l'isomorphisme composé
\begin{align*}
\rho_{X,Z}:A^{p-c}(Z,\phi)_{n-c}
 &\xrightarrow{\rho'_{X,Z}} \varphi_n(M_Z(X)[-p]) \\
 &\xrightarrow{\varphi(\mathfrak p_{X,Z})} \varphi_n(M(Z)(c)[2c-p])
 =\varphi_{n-c}(M(Z)[c-p])
\end{align*}
s'inscrit dans le diagramme commutatif:
$$
\xymatrix@R=14pt@C=34pt{
A^{p-1}(U,\phi)_{n}\ar^-{\partial_Z^U}[r]\ar_{\rho_U}[dd]
 & A^{p-c}(Z,\phi)_{n-c}\ar^-{i_*}[r]\ar^{\rho_{X,Z}}[d] 
 & A^p(X,\phi)_{n}\ar^{\rho_X}[dd] \\
 & \varphi_{n-c}(M(Z)[c-p])\ar@{=}[d]
 & \\
\varphi_n(M(U)[-p])\ar^-{\varphi_n(\partial_{X,Z})}[r]
 & \varphi_{n}(M(Z)(c)[2c-p])\ar^-{\varphi_n(i^*)}[r]
 & \varphi_n(M(X)[-p]).
}
$$
Il s'agit de voir que $\rho_{X,Z}=\rho_Z$. Notons que d'après ce qui précède,
le morphisme $\rho_{X,Z}-\rho_Z$ est naturel en $(X,Z)$ par rapport
 aux morphisme transverses (définis en \ref{rem:rappels_pf}). 
Soit $P_ZX$ la complétion projective du fibré normal de $Z$ dans $X$.
Considérons l'éclatement $B_Z(\AA^1_X)$ de $Z \times \{0\}$ dans $X$, 
ainsi que le diagramme de déformation classique qui
lui est associé
$$
(X,Z) \xrightarrow{(d,i_1)} (B_Z(\AA^1_X),\AA^1_Z)
 \xleftarrow{(d',i_0)} (P_ZX,Z).
$$
Les carrés correspondants à $(d,i_1)$ et $(d',i_0)$ sont transverses.
On est donc réduit au cas où $(X,Z)=(P_ZX,Z)$.
Dans ce cas, l'immersion fermée $i$ admet une rétraction et
le morphismes $\rho_{X,Z}$ (resp. $\rho_Z$) est déterminé
de manière unique par $\rho_X$.
\end{proof}

\begin{thm} \label{thm:main}
Les foncteurs
$$
\begin{array}{rcl}
\hmtr & \leftrightarrows & \modl \\
F_* & \mapsto & \hat F_* \\
F^\phi_* & \mapsfrom & \phi
\end{array}
$$
sont des équivalences de catégories quasi-inverses
l'une de l'autre.
\end{thm}
\begin{proof}
Il s'agit de construire les deux isomorphismes naturels qui réalisent
l'équivalence. \\
\underline{Premier isomorphisme}~:
Considérons un module de cycles $\phi$, $F^\phi_*$ le module homotopique
associé. Par définition, pour tout corps de fonctions $E$,
il existe une flèche canonique
$$
a_E:\hat F_n^\phi(E)=\ilim{A \subset E} A^0(\spec A;\phi)_n \rightarrow \phi_n(E).
$$
C'est trivialement un isomorphisme et il reste à montrer que $a$ définit
un morphisme de modules de cycles. La compatibilité à (D1) est évidente.
La compatibilité à (D2) résulte du fait que pour un morphisme fini surjectif
$f:Y \rightarrow X$, le morphisme $A^0(\tra f;\phi)$ est le pushout $f_*$
propre (\emph{cf.} \cite[6.6]{Deg5}).

\smallskip

\noindent \textit{Compatibilité à (D3)}~:
On reprend les notations du point (D3) de \ref{mhtp->mcycl} pour le module homotopique $F^\phi_*$
et pour une unité $x \in E^\times$. On considère la flèche canonique
$$
a'_E:\hat F_n^\phi(\GG \times (E))
 =\ilim{A \subset E} A^0\big(\spec{A[t,t^{-1}]};\phi\big)_n
  \longrightarrow \phi_n(E(t)).
$$
Pour tout $E$-point $y$ de $\spec{E[t]}$,
 on note $v_y$ la valuation de $E(t)$ correspondante,
  d'uniformisante $t-y$.
D'après la proposition \ref{Gysin&specialisation}, 
le diagramme suivant est commutatif~:
$$
\xymatrix@=12pt{
\hat F_n^\phi(\GG \times (E))\ar^-{s_x^*}[r]\ar_{a'_E}[d]
 & \hat F_n^\phi(E)\ar^{a_E}[d] \\
\phi_n(E(t))\ar^-{s_{v_x}^{t-x}}[r] & \phi_n(E).
}
$$

Par définition du morphisme structural $\epsilon_*$ de $F^\phi_*$
(\emph{cf.} \ref{mcycl->mhtp}), le morphisme 
$\nu':\hat F_{n-1}^\phi(E) \xrightarrow {\epsilon_{n-1}} (\hat F_n^\phi)_{-1}(E)
 \xrightarrow \nu \bar F_n^\phi(\GG \times (E))$
est la section de la suite exacte courte
$$
0 \rightarrow \hat F_n^\phi(E) \xrightarrow{p^*} \hat F_n^\phi(\GG \times (E)) 
 \xrightarrow{\partial} \hat F_{n-1}^\phi(E) \rightarrow 0
$$
qui correspond à la rétraction $s_1^*$ de $p^*$, 
 pour $s_1:(E) \rightarrow \GG \times (E)$ la section unité de
  la projection $p:\GG \times (E) \rightarrow (E)$.
En particulier, $\nu'$ est caractérisé par les propriétés
 $\partial \nu'=1$ et $s_1^* \nu'=0$.

Notons $\varphi:E \rightarrow E(t)$ l'inclusion canonique. On peut vérifier
en utilisant les relations des pré-modules de cycles les formules suivantes~:
\begin{align*}
(1) & \ \forall \rho \in \phi_n(E), \partial_{v_0}(\{t\}.\varphi_*(\rho))=\rho. \\
(2) & \ \forall y \in E^\times,
 \forall \rho \in \phi_n(E),\partial_{v_y}(\{t\}.\varphi_*(\rho))=0. \\
(3) & \ \forall y \in E^\times,
 \forall \rho \in \phi_n(E),s_{v_y}^{t-y}(\{t-y\}.\varphi_*(\rho))=\{y\}.\rho.
\end{align*}
D'après (2), l'application $\phi_n(E) \rightarrow \phi_n(E(t)),
 \rho \mapsto \{t\}.\varphi_*(\rho)$ induit une unique flèche pointillée
  rendant le diagramme suivant commutatif~:
$$
\xymatrix@R=12pt@C=20pt{
\hat F_n^\phi(E)\ar@{-->}[r]\ar_{a_E}[d]
 & \hat F_n^\phi(\GG \times (E))\ar^{a'_E}[d] \\
\phi_n(E(t))\ar^-{\{t\}.\varphi_*}[r] & \phi_n(E).
}
$$
D'après la relation (1) et la relation (3) avec $y=1$,
 cette flèche pointillée satisfait les deux propriétés caractérisant $\nu'$.
On déduit donc de la relation (3) avec $y=x$ que
 $\nu' \circ s_x^*(\rho)=\{x\}.\rho$ ce qui prouve la relation attendue.
 
\smallskip

\noindent \textit{Compatibilité à (D4)}~: 
Considérons les notations du point (D4) dans \ref{mhtp->mcycl}.
La compatibilité au résidu est alors une conséquence directe
 du lemme \ref{lm:comp_Gersten_loc} appliqué,
 pour tout voisinage ouvert $U$ de $x$ dans $X$,
 à l'immersion fermée $i:Z \cap U \rightarrow U$ dans le cas $c=1$, $p=1$. \\
\underline{Deuxième isomorphisme}~: 
Considérons un module homotopique $(F_*,\epsilon_*)$.
Pour tout schéma lisse $X$, en considérant la limite inductive des 
morphismes de restriction $F(X) \rightarrow F(U)$
 pour les ouverts $U$ de $X$,
  on obtient une flèche $F_*(X) \rightarrow C^0(X;\hat F_*)$
  qui induit par définition des différentielles
   un morphisme
   $b_X:F_*(X) \rightarrow A^0(X;\hat F_*)$
    homogène de degré $0$. \\
Le point clé est de montrer que cette flèche est naturelle
par rapport aux correspondances finies.
Soit $\alpha \in c(X,Y)$ une correspondance finie entre schémas lisses,
que l'on peut supposer connexes.
Rappelons que pour tout ouvert dense $j:U \rightarrow X$,
le morphisme $j^*:A^0(X;\hat F_*) \rightarrow A^0(U;\hat F_*)$
est injectif d'après la suite exacte de localisation \eqref{localisation2}.
Ainsi, on peut remplacer $\alpha$ par $\alpha \circ j$ et $X$ par $U$.
Par additivité, on se ramène encore au cas où $\alpha$ est la classe
d'un sous-schéma fermé intègre $Z$ de $X \times Y$, fini et dominant sur $X$.
Dès lors, $\alpha \circ j=[Z \times_X U]$. Donc puisque $k$ est parfait,
quitte à réduire $X$, on peut supposer que $Z$ est lisse sur $k$.
Rappelons que d'après \ref{transferts_complexe_Rost}, $\alpha^*=p_* i^* q^*$
pour les morphismes évidents suivants
$$
X \stackrel p \leftarrow Z \xrightarrow i Z \times X \times Y
\xrightarrow q Y.
$$
On est donc ramené à vérifier la naturalité dans les trois cas suivants~: \\
\textit{Premier cas}~: Si $\alpha=q$ est un morphisme plat,
 la compatibilité résulte alors de la définition du pullback plat
 sur $A^0(.;\hat F_*)$ est de la définition de D1. \\
\textit{Deuxième cas}~: Si $\alpha=\tra p$,
 $p:Z \rightarrow X$ morphisme fini surjectif entre schémas lisses.
Ce cas résulte de la définition du pushout propre sur $A^0$ et de la définition
de D2. \\
\textit{Troisième cas}~: Supposons $\alpha=i$, pour $i:Z \rightarrow X$ immersion
fermée régulière entre schémas lisses. Comme on l'a déjà vu,
l'assertion est locale en $X$.
On se réduit donc en factorisant $i$ au cas de codimension $1$. 
On peut aussi supposer que
 $Z$ est un diviseur principal paramétré par $\pi \in \mathcal O_{X}(U)$,
 pour $U=X-Z$.
D'après la proposition \ref{Gysin&specialisation}, on est ramené à
montrer que le diagramme suivant est commutatif~:
$$
\xymatrix@R=10pt@C=20pt{
F_*(X)\ar[d]\ar^{i^*}[r] & F_*(Z)\ar[d] \\
\hat F_*(\kappa(X))\ar^-{s_v^\pi}[r]
 & \hat F_*(\kappa(Z)).
}
$$
Tenant compte de la naturalité du morphisme structural $\epsilon_*$
 du module homotopique $F_*$, 
 on se ramène à la commutativité du diagramme~:
$$
\xymatrix@R=10pt@C=24pt{
\varphi(M(X)\{1\})\ar_{j^*}[d]\ar^-{i^*}[rrr] &&& \varphi(M(Z)\{1\})\ar@{=}[d] \\
\varphi(M(U)\{1\})\ar^-\nu[r]
 & \varphi(M(\GG \times U))\ar^-{\gamma_\pi^*}[r]
 & \varphi(M(U))\ar^-{\partial_{X,Z}}[r]
 & \varphi(M(Z)\{1\}) &
}
$$
où $\nu$ est l'inclusion canonique, 
$\gamma_\pi$ est induit par $\pi:U \rightarrow \GG$
 et $\partial_{X,Z}=\partial'_{X,Z} \circ \mathfrak p_{X,Z}^{-1}$ avec les notations
 de \ref{mhtp->mcycl}(D4) est le morphisme résidu au niveau des motifs.
Or la commutativité de ce diagramme résulte exactement de \cite[2.6.5]{Deg5}.

Le morphisme $b:F_* \rightarrow A^0(.;\hat F_*)$ est donc un morphisme
de faisceaux avec transferts. Or, il est évident que le morphisme induit sur
les fibres en un corps de fontions quelconque est un isomorphisme.
Il en résulte (\emph{cf.} \ref{foncteurs_fibres_fhtp}) que $b$ est un isomorphisme.
Enfin, on établit facilement la compatibilité de $b$ avec
les morphismes structuraux des modules homotopiques $F_*$
et $A^0(.;\hat F_*)$ compte tenu de la construction \ref{mcycl->mhtp} --
on utilise simplement la fonctorialité de $b$ par rapport à
$j_X:\GG \times X \rightarrow \AA^1_X$
et $s_1:X \rightarrow \GG \times X$.
\end{proof}

\num Le théorème précédent montre que la catégorie des modules
de cycles est monoïdale symétrique avec pour élément neutre le
foncteur de K-théorie de Milnor. Le produit tensoriel est
de plus compatible au foncteur de décalage de la graduation
des modules de cycles -- \emph{i.e.} le foncteur
noté $\{\pm 1\}$ dans $\hmtr$.

A tout schéma lisse $X$, on associe un module de cycles
$$
\hhlrep_{0,*}(X)=(h_{0,*}(X))^\wedge.
$$
D'après le théorème précédent,
la famille de modules de cycles $(\hhlrep_{0,*}(X)\{n\})$
pour un schéma lisse $X$ et un entier $n \in \ZZ$
est génératrice dans la catégorie abélienne $\modl$. \\
Notons que ces générateurs caractérisent le produit tensoriel
des modules de cycles:
$$
\hhlrep_{0,*}(X)\{n\} \otimes \hhlrep_{0,*}(Y)\{m\}
 =\hhlrep_{0,*}(X \times Y)\{n+m\}.
$$
On peut enfin donner une formule explicite pour
 calculer ces modules de cycles.
Considérons pour tous schémas lisses $X$ et $Y$ le groupe
$$
\pi(Y,X)=\coKer\big(c(\AA^1_Y,X)
 \xrightarrow{s_0^*-s_1^*} c(Y,X)\big).
$$
Notons que ce groupe s'étend de manière évidente aux schémas
réguliers essentiellement de type fini sur $k$
et que l'on dipose d'un théorème de commutation aux limites projectives
de schémas pour ces groupes étendus (\emph{cf.} \cite[4.1.24]{Deg7}).
Par ailleurs, si $E$ est un corps de fonctions,
 $\pi(\spec E,X)=CH_0(X_E)$, groupe de Chow des $0$-cycles
 de $X$ étendu à $E$.

On déduit de tout cela les calculs suivants:
Pour tout corps de fonctions $E$
 et tout schéma lisse $X$, 
$$
\hhlrep_{0,0}(X).E=CH_0(X_E).
$$
De plus, pour tout entier $n>0$,
\begin{align*}
\hhlrep_{0,n}(X).E&=
\coKer\big(
\oplus_{i=0}^n CH_0(\GG^{n-1} \times X_E)
 \rightarrow CH_0(\GG^n \times X_E)\big) \\
\hhlrep_{0,-n}(X).E&=
\Ker\big(
\pi(\GGx E^n,X)
 \rightarrow\oplus_{i=0}^n \pi(\GGx E^{n-1},X)\big)
\end{align*}
où les flèches sont induites par les injections évidentes
$\GG^i \times \{1\} \times \GG^{n-1-i} \rightarrow \GG^n$.

\subsection{Résolution de Gersten et cohomologie}

Soit $F_*$ un module homotopique, $\phi=\hat F_*$ sa transformée générique.
Considérons l'isomorphisme $b:F_* \rightarrow F^{\phi}_*$ qui lui
est associé d'après le théorème précédent. Compte tenu de l'isomorphisme
\eqref{eq:iso_gpe_chow&coh_nis}, on en déduit un isomorphisme
$$
\epsilon_X:H^n_\nis(X;F_*)
 \xrightarrow{b_*} H^n_\nis(X;F^{\phi}_*)
 \xrightarrow{\rho_X^{-1}} A^n(X;\hat F_*).
$$
\begin{prop} \label{iso_coh_gpe_chow}
Avec les notations ci-dessus, $\epsilon_X$ est un isomorphisme
compatible avec les transferts par rapport à $X$.
\end{prop}
\begin{proof}
Par définition, on est ramené à prouver la compatibilité de 
$\rho_X^{-1}:H^n_\nis(X;F^\phi_*) \rightarrow A^n(X;\phi)$ 
par rapport aux transferts en $X$ pour un module de cycles
arbitraire $\phi$. \\
Pour un schéma simplicial $\cX$, on note $F^\phi(\cX)$ 
le complexe des sommes alternées 
associé au groupe abélien cosimplicial évident.
Rappelons l'isomorphisme canonique
$$
H^n_\nis(X;F_*^\phi)=\ilim{\cX/X} H^n F^\phi_*(\cX)
$$
où la limite inductive parcourt les hyper-recouvrements
Nisnevich de $X$. \\
Pour un schéma simplicial $\cX$, on note $\Tot C^*(\cX;\phi)$
le complexe total associé au bicomplexe évident.
Pour un hyper-recouvrement Nisnevich $\cX/X$,
on obtient alors des morphismes de complexes,
naturels en $\cX$,
$$
F^\phi_*(\cX) \xrightarrow{(1)} \Tot C^*(\cX;\phi)
 \xleftarrow{(2)} C^*(X;\phi).
$$
Suivant la construction de \eqref{eq:iso_gpe_chow&coh_nis},
le morphisme $\rho_X^{-1}$ est égal à~:
$$
\ilim{\cX/X} H^nF^\phi_*(\cX)
 \rightarrow \ilim{\cX/X} H^n\Tot C^*(\cX;\phi)
 \xrightarrow{(2')} H^nC^*(X;\phi)
$$
où la flèche $(2')$ est la réciproque de la limite des
flèches de type $H^n(2)$. \\
Pour vérifier la compatibilité de cet isomorphisme
avec les transferts, on est alors ramené au lemme suivant~:
\begin{lm}
Soit $\alpha \in c(Y,X)$ une correspondance finie
entre deux schémas lisses. 
Alors, pour tout hyper-recouvrement $\cX$ Nisnevich de $X$,
 il existe un hyper-recouvrement Nisnevich $\cY$ de $Y$
 et une transformation naturelle $\tilde \alpha$
 faisant commuter le diagramme
$$
\xymatrix@=10pt{
\rep \cY\ar^{\tilde \alpha}[r]\ar[d] & \rep \cX\ar[d] \\
\rep Y\ar^\alpha[r] & \rep X
}
$$ 
où $\rep \cX$ (resp. $\rep \cY$) désigne le complexe
de faisceaux associé au faisceau simplicial évident.
\end{lm}
Notons que d'après \cite[4.2.9]{Deg7},
le morphisme canonique $\rep \cX \rightarrow \rep X$ est un quasi-isomorphisme.
Le lemme en résulte dès lors,
 suivant une construction par induction.

Ce lemme nous permet de conclure. En effet,
considérant $\alpha$ et $\tilde \alpha$ comme ci-dessus,
les morphismes du type $H^n(1)$ et $H^n(2)$
sont naturels par rapport à $\alpha$ et $\tilde \alpha$,
comme il résulte de la compatibilité des groupes $A^n(.;\phi)$
par rapport au produit de composition des correspondances.
\end{proof}

\num \label{iso_Bloch}
Considérons le module homotopique $\tatex *$.
Suivant \cite[3.4]{SV2}, pour tout corps de fontions $E$,
 $\tatex *(E) \simeq K_*^M(E)$. 
Cet isomorphisme est de plus compatible
 aux structures de module de cycles.
Pour la norme, cela résulte de \cite[3.4.1]{SV2}.
Pour le résidu associé à un corps de fontions valué $(E,v)$,
on se réduit à montrer que $\partial_v(\pi)=1$ pour
le module de cycle $\htatex *$, 
ce qui résulte de \cite[2.6.5]{Deg7}.

On en déduit l'isomorphisme de Bloch\footnote{En effet, 
d'après l'isomorphisme que l'on vient d'expliciter,
le faisceau gradué $\tatex *$ est le faisceau de
K-théorie de Milnor non ramifié.} 
pour tout schéma lisse $X$~:
$$
\epsilon_X^\mathrm B:H^n_\nis(X;\tatex n)
 \rightarrow A^n(X;K_*^M)_n=CH^n(X).
$$
Cet isomorphisme est naturel par rapport aux transferts.

Rappelons que pour tout module de cycles $\phi$,
il existe un accouplement de modules de cycles 
$K_*^M \times \phi \rightarrow \phi$
au sens de \cite[1.2]{Ros}.
Il induit d'après \cite[par. 14]{Ros} un accouplement
$$
A^m(X;\phi)_r \otimes CH^n(X)
 \rightarrow A^{m+n}(X;\phi)_{r+n}.
$$
Considérant un module homotopique $F_*$,
on dispose d'un (iso)morphisme de modules homotopiques
$S^*_t \otimes F_* \rightarrow F_*$.
Pour un schéma lisse $X$, de diagonale $\delta:X \rightarrow X \times X$,
on en déduit un accouplement
$$
H^m(X;F_*)_r \otimes H^n(X;\tatex *)_n \rightarrow H^{m+n}(X;F_*)_{r+n}
$$
définit en associant aux morphismes
 $a:h_{0,*}(X) \rightarrow \tatex *\{n\}[n]$
  et $b:h_{0,*}(X) \rightarrow F_*\{r\}[m]$ 
 la composée
$$
h_{0,*}(X) \xrightarrow{\delta_*} h_{0,*}(X) \otimes h_{0,*}(X)
 \xrightarrow{a \otimes b} \tatex * \otimes F_*\{n+r\}[n+m]
 \xrightarrow \sim F_*\{n+r\}[n+m].
$$
Nous laissons au lecteur le soin de vérifier la compatibilité
suivante~:
\begin{lm} \label{iso_Bloch&produit}
Avec les notations introduites ci-dessus, le diagramme suivant est commutatif~:
$$
\xymatrix@R=18pt@C=28pt{
H^m(X;F_*)_r \otimes H^n(X;\tatex *)_n
  \ar[r]\ar_{\epsilon_X \otimes \epsilon_X^\mathrm B}[d]
 & H^{n+m}(X;F_*)_{n+r}\ar^{\epsilon_X}[d] \\
A^m(X;\phi)_r \otimes CH^n(X)\ar[r] & A^{m+n}(X;\hat F_*)_{n+r}
}
$$
\end{lm}
Ainsi, l'isomorphisme $\epsilon_X^\mathrm B$ est compatible au produit,
et l'isomorphisme $\epsilon_X$ est compatible aux structures
de module décrites ci-dessus.

\num Notons $\varphi:\DMgm^{op} \rightarrow \ab$ le foncteur de
réalisation associé à $F_*$ (\emph{cf.} section \ref{realisation_mhtp}). 
D'après la proposition précédente,
le foncteur $\varphi$ prolonge le foncteur $A^*(.;\hat F_*)$.
Ainsi, on a étendu canoniquement la cohomologie à coefficients
dans un module de cycles quelconque en un foncteur de réalisation
triangulé de $\DMgm$.
Nous notons encore
$$
\epsilon_X:\varphi(M(X)\{-r\}[-n]) \rightarrow A^n(X;\hat F_*)_r
$$
l'isomorphisme qui se déduit par construction de \ref{iso_coh_gpe_chow}.

Soit $f:Y \rightarrow X$ un morphisme projectif
entre schémas lisses, de dimension relative constante $d$.
Dans \cite[2.7]{Deg6}, on a construit 
$f^*:M(X)(d)[2d] \rightarrow M(Y)$, 
morphisme de Gysin associé à $f$ dans $\DMgm$.
\begin{prop} \label{Gysin&pushout}
Considérons les notations introduites ci-dessus.
Alors, le carré suivant est commutatif~:
$$
\xymatrix@R=10pt@C=30pt{
\varphi\big(M(X)\{d-r\}[d-n]\big)\ar^-{\varphi(f^*)}[r]\ar_{\epsilon_X}[d]
 & \varphi\big(M(Y)\{-r\}[-n]\big)\ar^{\epsilon_Y}[d] \\
A^{n-d}(X;\hat F_*)_{r-d}\ar^-{f_*}[r] & A^n(Y;\hat F_*)_r
}
$$
\end{prop}
\begin{proof}
Dans cette preuve, on utilisera particulièrement le lemme suivant~:
\begin{lm}
Soit $X$ un schéma lisse et $E/X$ un fibré vectoriel de rang $n$.
Soit $p:P \rightarrow X$ le fibré projectif associé, et
 $\lambda$ le fibré inversible canonique sur $P$ tel que
  $\lambda \subset p^{-1}(E)$.
On note $c=c_1(\lambda) \in CH^1(X)$ la première classe de Chern
de $\lambda$.

Alors, le morphisme suivant est un isomorphisme~:
$$
\begin{array}{rcl}
\bigoplus_{i=0}^n A^*(X;\hat F_*) & \rightarrow & A^*(P;\hat F_*) \\
x_i & \mapsto & p^*(x_i).c^i.
\end{array}
$$
en utilisant la structure de $CH^*(X)$-module (ici notée à droite) 
de $A^*(X;\hat F_*)$ rappelée en \ref{iso_Bloch}.
\end{lm}
Pour obtenir ce lemme,
il suffit d'appliquer le théorème du fibré projectif dans $\DMgm$
(\emph{cf.} \cite[2.5.1]{V1}) et de regarder son image par $\varphi$
compte tenu du lemme \ref{iso_Bloch&produit}.

Soit $E$ un fibré vectoriel sur $X$
 et $P$ sa complétion projective.
On déduit de ce lemme le cas où $f=p$.
En effet, d'après la formule de projection
$$
p_*(p^*(x_i).c^i)=x_i.p_*(c^i)
$$
pour les groupes de Chow à coefficients (\emph{cf.} \cite[5.9]{Deg4}), 
 on déduit que $p_*$ 
est la projection évidente à travers
le théorème du fibré projectif.
L'analogue de ce calcul pour $\varphi(p^*)$
résulte des définitions de \cite{Deg6}.

Compte tenu de la définition du morphisme de Gysin et du cas
que l'on vient de traiter, nous sommes ramenés au cas où $f=i:Z \rightarrow X$
est une immersion fermée,
 que l'on peut supposer de codimension pure égale à $c$.
Ce cas est alors une conséquence directe du lemme
 \ref{lm:comp_Gersten_loc}.
%
\end{proof}

\part{Motifs mixtes triangulés}
\label{part:motifs}
\section{Cas effectif}

\subsection{Complexes de faisceaux avec transferts}

\num \label{cat_modele_ftr}
On note $\Comp(\ftr)$ la catégorie des complexes 
(non nécessairement bornés)
de faisceaux avec transferts.

Pour tout faisceau avec transferts $F$ et tout entier $n \in \ZZ$,
on note $D^nF$ le complexe concentré en degrés $n$ et $n+1$ dont la
seule diférentielle non nulle est l'identité de $F$.
On note $S^nF$ le complexe concentré en degré $n$ égal à $F$.
On définit la classe des \emph{cofibrations} comme la plus petite classe
stable par pushout, composition transfinie et rétracte contenant
les morphismes de complexes évidents 
$S^{n+1} \rep X \rightarrow D^n \rep X$.

On dit qu'un complexe de faisceau avec transferts $C$ est 
\emph{$\nis$-local}
si pour tout schéma lisse $X$ et tout entier $n \in \ZZ$, le morphisme
canonique
$$
H^n\big( C(X) \big) \rightarrow H^n_\nis(X,C)
$$
est un isomorphisme. On dit qu'un morphisme $\phi:C \rightarrow D$ de 
$\Comp(\ftr)$ est une fibration si c'est un épimorphisme dans la catégorie
$\Comp(\ftr)$ et son noyau est $\nis$-local.
\begin{prop} \label{prop:modele_c_ftr}
La catégorie $\Comp(\ftr)$ avec pour équivalences faibles
les quasi-isomorphismes, munie des cofibrations et des fibrations
définies ci-dessus, est une catégorie de modèles.
\end{prop}
Cela résulte du théorème 1.7 de \cite{CD1}, compte tenu de l'exemple 1.6.
Rappelons qu'un point essentiel dans cette proposition est le résultat
suivant (voir \cite[4.2.9]{Deg7}) dû à Voevodsky~:
\begin{prop}
Pour tout complexe de faisceaux avec transferts $C$, tout schéma lisse $X$
et tout entier $n \in \ZZ$,
$$
\Hom_{D(\ftr)}(\rep X,C[n])=H^n_\nis(X,C).
$$
\end{prop}

\rem Par définition, un complexe de faisceaux avec transferts $C$
est fibrant (\emph{i.e.} $\nis$-local) si il est fibrant au sens de 
\cite[1.1.5]{CD2}
en tant que complexe de faisceaux Nisnevich (ayant oublié les transferts).
Il résulte donc de \emph{op. cit.} prop. 1.1.10 que cette propriété 
est équivalente à la suivante~:
\begin{enumerate}
\item[(BG)] Pour tout carré cartésien 
$\xymatrix@=10pt{W\ar[r]\ar[d] & V\ar^f[d] \\ U\ar^j[r] & X}$
tel que $j$ est une immersion ouverte, $f$ est étale et $f^{-1}(X-U) \rightarrow X-U$
est un isomorphisme topologique, le carré
$$\xymatrix@=10pt{C(X)\ar^{f^*}[r]\ar_{j^*}[d] & C(V)\ar[d] \\ C(U)\ar[r] & C(W)}$$
est homotopiquement cartésien.
\end{enumerate}

\num La catégorie $\Comp(\ftr)$ est symétrique monoïdale.
Le produit tensoriel de deux complexes $C$ et $D$ est donné par la formule
\begin{align*}
(C \otr D)^n=\oplus_{p+q=n} C^p \otr D^q \\
d(x \otr y)=dx \otr y + (-1)^{\deg(x)} x \otr dy.
\end{align*}
L'isomorphisme de symétrie est égal en degré $n$ à la somme des morphismes
$$
C^p \otr D^q \xrightarrow{(-1)^{pq}.\epsilon} D^q \otr C^p
$$
où $\epsilon$ est l'isomorphisme de symétrie du produit tensoriel des faisceaux
avec transferts.

Ce produit tensoriel est exact à droite. Comme la catégorie $\Comp(\ftr)$ est 
abélienne de Grothendieck, on en déduit l'existence formelle du foncteur
Hom interne $\uHom$ adjoint à droite de $\otr$. Explicitement, ce foncteur
est donné par la formule:
$$
\uHom(C,D)=\mathrm{Tot}^\pi \uHom_{\ftr}(C,D)
$$
où $\mathrm{Tot}^\pi$ désigne le complexe total produit d'un bicomplexe
de faisceaux avec transferts.
\begin{prop}
La catégorie $\Comp(\ftr)$ munie des structures introduites ci-dessus 
 est une catégorie de modèles symétrique monoïdale.
\end{prop}
Cela résulte de la proposition 2.3 de \cite{CD1}, 
compte tenu de l'exemple 2.4.
On obtient donc un produit tensoriel dérivé~: soit $C$ (resp. $D$)
un complexe de faisceaux avec transferts et $C' \rightarrow C$ 
(resp. $D' \rightarrow D$) une résolution cofibrante.
Alors, $C \otrL D=C' \otr D'$.
Ainsi, puisque pour tout schéma lisse $X$ le faisceau $\rep X$ placé
en degré $0$ est cofibrant, $\rep X \otrL \rep Y=\rep{X \times Y}$. \\
On peut définir aussi le Hom interne dérivé, en considérant
de plus une résolution fibrante $D \rightarrow D''$~:
$$
\RHom(C,D)=\uHom(C',D'').
$$

\subsection{Complexes motiviques}

\num On dit qu'un complexe de faisceaux avec transferts $C$ est 
\emph{$\AA^1$-local} si pour tout schéma lisse $X$
 et tout entier $n \in \ZZ$,
le morphisme induit par la projection canonique
$$
H^n_\nis(X,C) \rightarrow H^n_\nis(\AA^1_X,C)
$$
est un isomorphisme.

On dit qu'un morphisme $f:C \rightarrow D$ de $\Comp(\ftr)$
est une \emph{$\AA^1$-équivalence} si pour tout complexe
$\AA^1$-local $E$, le morphisme induit
$$
\Hom_{\Der(\ftr)}(D,E) \rightarrow \Hom_{\Der(\ftr)}(C,E)
$$
est un isomorphisme.
On dit aussi que $f$ est une \emph{$\AA^1$-fibration}
si c'est un épimorphisme de complexe et son noyau
est à la fois $\nis$-local et $\AA^1$-local.
\begin{prop} \label{A^1-eff-modele}
\begin{enumerate}
\item La catégorie $\Comp(\ftr)$ avec pour équivalences faibles
les $\AA^1$-équivalences, pour fibrations les $\AA^1$-fibrations
et pour cofibrations celles définies en \ref{cat_modele_ftr}
est une catégorie de modèles symétrique monoïdale.
\item La catégorie homotopique associée s'identifie à 
la localisation de la catégorie triangulée $\Der(\ftr)$ par
rapport à la catégorie localisante engendrée par les morphismes
$\rep {\AA^1_X} \rightarrow \rep X$.
\item Cette catégorie homotopique s'identifie encore à la sous-catégorie
pleine de $\Der(\ftr)$ formée des complexes $\AA^1$-locaux.
\end{enumerate}
\end{prop}
Cela résulte de \cite[3.5]{CD1} compte tenu de l'exemple 3.15.
Rappelons que la catégorie de modèles de cette proposition est la localisation
de Bousfield à gauche de la catégorie de modèles de \ref{prop:modele_c_ftr}.
On l'appelle dans la suite la catégorie de modèles $\AA^1$-locale
sur $\Comp(\ftr)$.

\begin{df}
On note $\DMe$ la sous-catégorie pleine de $\Der(\ftr)$ formée
des complexes $\AA^1$-locaux.
\end{df}
On en déduit donc une adjonction de catégories triangulées:
\begin{equation} \label{eq:adj_D_DM}
L_{\AA^1}:\Der(\ftr) \leftrightarrows \DMe:\cO
\end{equation}
avec $\cO$ le foncteur d'inclusion.
La catégorie $\DMe$ est triangulée monoïdale symétrique fermée
et le foncteur $L_{\AA^1}$ est monoïdal. On notera simplement
$\otimes$ le produit tensoriel de $\DMe$ (il s'agit
du produit tensoriel dérivé de $\otr$ pour la structure de
catégorie de modèles $\AA^1$-locale sur $\Comp(\ftr)$).
On note aussi $\un$ l'objet unité de la catégorie monoïdale
$\DMe$\footnote{Voevodsky le note $\mathbb Z$ dans \cite{V1}.}.
Pour tout schéma lisse $X$, on pose $M(X)=L_{\AA^1}(\rep X)$.

\num Si $C$ est un complexe de faisceaux avec transferts,
 et $n$ un entier, on note $\upH^n(C)$ le préfaisceau sur $\smc$
 qui à $X$ associe $H^n(C(X))$.
On note $\uH^n(C)$ le faisceau avec transferts associé à $\upH^n(C)$
 (\emph{cf.} \cite[4.2.7]{Deg7})\footnote{Si l'on oublie les transferts,
  $\uH^n(C)$ est le $n$-ième faisceau de cohomologie de $C$.}.
Dans notre contexte, le théorème suivant donne une reformulation
 des résultats principaux de Voevodsky concernant les complexes
 motiviques (\emph{cf.} \cite{V1}):
\begin{thm} \label{thm:Voevodsky}
Soit $C$ un complexe de faisceaux avec transferts. Les conditions
suivantes sont équivalentes~:
\begin{enumerate}
\item[(i)] $C$ est $\AA^1$-local.
\item[(ii)] Pour tout entier $n \in \ZZ$, $\uH^n(C)$ est $\AA^1$-local.
\item[(iii)] Pour tout entier $n \in \ZZ$, $\uH^n(C)$ est invariant par homotopie.
\end{enumerate}
\end{thm}
\begin{proof}
L'équivalence de (i) et (ii) résulte
 de la suite spectrale d'hypercohomologie Nisnevich.
L'implication $(ii) \Rightarrow (iii)$ est évidente
 et sa réciproque résulte du théorème de Voevodsky
 \cite[4.5.1]{Deg7}. 
\end{proof}
D'un point de vue terminologique, un 
\emph{complexe de faisceaux avec transferts $\AA^1$-local}
en notre sens est donc un \emph{complexe motivique} au sens de Voevodsky 
(à condition qu'il soit borné supérieurement). 
L'intérêt de la définition que nous avons adoptée est 
qu'elle garde un sens pour des bases plus générales qu'un corps parfait 
(\emph{cf.} \cite{CD1} où la construction présentée ici est généralisée
au cas d'une base régulière).

\num \label{complexe_Suslin}
Notons finalement que Voevodsky obtient une formule explicite pour le foncteur
$L_{\AA^1}$ grâce au \emph{complexe des chaînes singulières de Suslin} $\sing$
(\emph{cf.} \cite[3.2.3]{V1}). Pour un complexe $K$ de faisceau avec transferts,
$$
\sing(K)=\uHom(\rep{\Delta^*},K)
$$
où $\Delta^*$ est le schéma cosimplicial standard et $\rep{\Delta^*}$
 et le complexe de faisceaux avec transferts associé\footnote{A priori,
 la formule ci-dessus donne donc un complexe homologique. Ce que l'on note
 $\sing(K)$ est le complexe cohomologique obtenu en inversant la graduation.}.

D'après \cite[3.2.6]{V1}, on obtient un foncteur pleinement fidèle
$$
\DMgme \rightarrow \DMe.
$$
On obtient aussi ce résultat dans \cite{CD1} (\emph{cf.} exemple 5.5) 
où l'on démontre de plus que $\DMgme$ est la sous-catégorie pleine
de $\DMe$ formée des complexes $K$ qui sont compacts 
-- \emph{i.e.} le foncteur $\Hom_{\DMe}(K,.)$ commute aux
sommes directes quelconques.

\subsection{t-structure homotopique}

La catégorie $\Der(\ftr)$ porte naturellement une t-structure,
 celle dont les tronqués négatifs sont donnés par la formule
$$
\tau_{\leq 0}(K)^n=\begin{cases}
K^n & \text{si } n<0 \\
\Ker(d) &  \text{si } n=0 \\
0 & \text{sinon.}
\end{cases}
$$
Le foncteur pleinement fidèle $\cO$ de l'adjonction \eqref{eq:adj_D_DM} 
 permet de transporter cette t-structure à $\DMe$.
Un complexe $\AA^1$-local $K$ est dit \emph{positif} (resp. \emph{négatif})
si pour tout $n<0$ (resp. $n>0$), $\uH^n(K)=0$.
Notons que les foncteurs de troncation négatif (ci-dessus)
et positif respectent les complexes $\AA^1$-locaux 
d'après le théorème \ref{thm:Voevodsky}.

Suivant Voevodsky, on appelle la t-structure sur $\DMe$
ainsi définie la \emph{t-structure homotopique}.
Notons que cette t-structure est non dégénérée --
comme c'est le cas pour la t-structure canonique sur $\Der(\ftr)$.

Il est immédiat que le coeur de la t-structure homotopique
est équivalent à la catégorie des faisceaux homotopiques $\hftr$
(introduite dans la partie précédente), ceci grâce au foncteur
\begin{equation} \label{H^0_DMe}
\uH^0:\DMe \rightarrow \hftr.
\end{equation}
Si $F$ est un faisceau avec transferts,
on obtient de plus $\uH^0(L_{\AA^1} F)=h_0(F)$
 avec les notations de la première partie,
 compte tenu de \ref{complexe_Suslin}.
De même, pour tout schéma lisse $X$, $h_0(X)=\uH^0(M(X))$.

\num Soit $T$ le faisceau avec transferts 
conoyau du morphism $\rep k \xrightarrow{s_*} \rep{\GG}$
induit par la section unité. Ainsi, $\tate=\uH^0(T)$
et on en déduit que si $F$ est un faisceau homotopique,
$F_{-1}=\uHom(T,F)$ où le Hom interne est calculé dans
la catégorie des faisceaux avec transferts.

Pour tout préfaisceau avec transferts $F$,
 on pose plus généralement $\Omega F=\uHom(T,F)$.
Si $F$ est un faisceau (resp. faisceau homotopique),
 $\Omega F$ est un faisceau (resp. faisceau homotopique).
Soit $a$ le foncteur faisceau avec transferts associé
 définit dans \cite[4.2.7]{Deg7}.
Pour tout préfaisceau avec transferts $F$,
 on en déduit un morphisme canonique
$$
Ex_F:a \circ \Omega(F) \rightarrow \Omega \circ a(F).
$$
\begin{lm}
Le morphisme $Ex_F$ est un isomorphisme.
\end{lm}
\begin{proof}
D'après \cite[4.4.14]{Deg7}, la source et le but de $Ex_F$
 sont des faisceau homotopiques. 
D'après le paragraphe \ref{foncteurs_fibres_fhtp}, 
 il suffit donc de vérifier que $Ex_F$ induit un isomorphisme 
 sur les fibres au point définit par un corps de fonctions $E$.
Par définition du foncteur $\Omega$,
 on est ramené à montrer l'égalité
 $a F(\GG \times_k (E))=F(\GG \times_k (E))$.
Mais cela résulte de \cite[4.4.10]{Deg7}
 appliqué au préfaisceau homotopique $\hat p^*F$
  pour $p:\spec E \rightarrow \spec k$
  (voir 4.2.18 et 4.2.21).
\end{proof}

Si $C$ est un complexe de faisceau avec transferts,
 on note $\Omega C$ le complexe obtenu en appliquant
 $\Omega$ en chaque degré.
\begin{cor} \label{cor_uH&Omega}
Soit $C$ un complexe $\AA^1$-local de faisceaux avec transferts.
Alors, pour tout entier $n \in \ZZ$, $\uH^n(\Omega C)=\uH^n(C)_{-1}$.
Le complexe $\Omega C$ est $\AA^1$-local.
\end{cor}
\begin{proof}
La première assertion est un corollaire du lemme précédent
 compte tenu du fait que $\upH^n$ commute à $\Omega$.
  La deuxième assertion en résulte.
\end{proof}

En particulier, $\Omega$ préserve les quasi-isomorphismes entre complexes
motiviques. On en déduit donc un endofoncteur $\DMe \rightarrow \DMe$
 que l'on note encore $\Omega$.
\begin{cor} \label{-1_sur_complexes}
L'endofoncteur $\Omega$ est $t$-exact. Il induit le foncteur $?_{-1}$
 sur le coeur de $\DMe$. Pour tout complexe motivique $C$,
$$
\Omega(C)=\uHom_{\DMe}(T,C).
$$
\end{cor}

\num \label{twist_Tate_suspendu}
Rappelons que le motif de Tate $\un(1)$ dans la catégorie $\DMe$
est définit par: $M(\PP^1_k)=\un \oplus \un(1)[2]$.
De plus, $L_{\AA^1} T=\un(1)[1]$. Il est commode d'introduire 
la notation redondante $\un\{1\}:=\un(1)[1]$.
Pour un entier $n \geq 0$, on pose
$\un\{n\}=\un\{1\}^{\otimes,n}$ et pour un complexe motivique $C$, 
 $C\{n\}=\un\{n\} \otimes C$.
Notons que l'on peut énoncer le théorème de simplification de Voevodsky 
\cite[4.10]{V3} sous
la forme suivante:
\begin{thm}[Voevodsky]\label{simplification_DM}
Pour tout complexe motivique $C$,
 le morphisme d'ajonction 
$$
C \rightarrow \uHom_{\DMe}(\un\{1\},C\{1\})=\Omega(C\{1\})
$$
est un isomorphisme.
\end{thm}

\section{Cas stable}

\subsection{Spectres motiviques}

Avec Denis-Charles Cisinski,
 nous avons construit dans \cite[ex. 6.25]{CD1}
 la catégorie triangulée monoïdale symétrique fermée $\DM$
 munie d'une adjonction de catégories triangulées
$$
\Stab:\DMe \leftrightarrows \DM:\Lop
$$
telle que $\Stab$ est monoïdal symétrique
 et envoie le motif de Tate sur un objet inversible de $\DM$\footnote{
Cette catégorie est la catégorie triangulée monoïdale
universelle pour ce problème si l'on se restreint au catégories triangulées
qui sont des catégories homotopiques associées à une catégorie de
modèles monoïdale (\emph{cf.} \cite{Hov})}. On rappelle ici une construction
plus simple de cette catégorie utilisant les spectres non symétriques,
dont le défaut est de ne pas fournir de construction directe du produit tensoriel.
Nous indiquons pourquoi cette définition simplifiée est équivalente
 à celle de {\it loc. cit.} dans la proposition \ref{prop:eq_sp_sym&non_sym}.

\num De manière analogue à \ref{num:cat_gr},
on note $\ftrgr$ la categorie des faisceaux avec transferts
 positivement gradués.
Elle est abélienne de Grothendieck
avec pour générateurs la famille $(\rep X\{-i\})$ indexée
par les schémas lisses $X$ et les entier $i \geq 0$.
Elle est de plus monoïdale symétrique (fermée) avec la définition
analogue à celle de \emph{loc. cit.} du produit tensoriel, noté ici $\otrgr$.

Soit $T^*$ le faisceau avec transferts positivement gradué 
égal à $T^{\otr,n}$ en degré $n$. C'est un monoïde dans $\ftrgr$.
La catégorie des modules à gauche sur $T^*$ forme une catégorie
abélienne que l'on note $\Tmod$.
Le foncteur $T^*$-module libre $F_T$ induit une adjonction de catégories
\begin{equation} \label{adj_ftrgr_Tmod}
F_T:\ftrgr \rightarrow \Tmod:\cO_T
\end{equation}
et $\cO_T$ est exact, conservatif. Pour un schéma lisse $X$,
et un entier $i \geq 0$,
on note $T^*(X)\{-i\}$ l'image de $\rep X\{-i\}$ par $F_T$.
La catégorie $\Tmod$ est abélienne de Grothendieck avec pour générateurs
la famille $(T^*(X)\{-i\})$ où $X$ est un schéma lisse
 et $i \geq 0$ un entier.
Notons enfin que pour tout entier $i \in \NN$,
 on dispose encore d'une adjonction
\begin{equation} \label{adj_ftr_Tmod}
F_i:\ftr \leftrightarrows \Tmod:Ev_i
\end{equation}
telle que $Ev_i(F_*)=F_i$ et pour tout schéma lisse $X$,
 $F_i(\rep X)=T^*(X)\{-i\}$.
 
Du fait que le monoïde $T^*$ n'est pas commutatif,
 la catégorie $\Tmod$ n'est pas monoïdale symétrique.
On a malgré tout une action à gauche de $\ftr$ sur $\Tmod$
par la formule:
\begin{equation} \label{left_actn_ftr}
(G \otr F_*)_n=G \otr F_n
\end{equation}
la multiplication étant donnée grâce à l'isomorphisme
de symétrie du produit tensoriel sur $\ftr$.
Notons enfin que pour tout faisceau avec transferts $G$,
l'endofoncteur $G \otr ?$ de $\Tmod$ admet un adjoint
à droite.

\num Un complexe de $\Tmod$ est appelé un \emph{spectre de Tate}.
Un tel objet est donc une famille $(E_n)_{n \in \NN}$
de complexes de faisceaux avec transferts muni de morphismes:
$$
\mu_n:T \otrgr E_n \rightarrow E_{n+1}.
$$
ou encore, de manière équivalente, de morphismes
$$
\epsilon_n:E_n \rightarrow \Omega E_{n+1}.
$$
La catégorie dérivée de $\Tmod$ est bien définie
 et on obtient à l'aide de \eqref{adj_ftrgr_Tmod} un foncteur conservatif
\begin{equation} \label{O_DTmod_Dftrgr}
\cO_T:\Der(\Tmod) \rightarrow \Der(\ftrgr).
\end{equation}
Soit $(f_n:E_n \rightarrow F_n)_{n \in \NN}$ un morphisme
de spectres de Tate. 
On dit que $f$ est un fibration (resp équivalence faible) 
 si pour tout $n \in \NN$, $f_n$ est une fibration
 au sens de \ref{cat_modele_ftr} (resp. quasi-isomorphisme).
\begin{lm} \label{lm:modele_stable}
La classe des fibrations et des équivalences faibles définies
 ci-dessus induisent une structure de catégorie
 de modèles
  sur $\Comp(\Tmod)$ pour laquelle la catégorie homotopique 
  associée est la catégorie $\Der(\Tmod)$.
\end{lm}
\begin{proof}
Il s'agit d'un cas particulier de \cite[1.14]{Hov}
 appliqué à la catégorie de modèles de \ref{prop:modele_c_ftr}\footnote{
On peut aussi utiliser \cite[1.7]{CD1} avec
 la bonne structure de descente.}.
\end{proof}
Pour cette structure de catégorie de modèles
 et celle considérée dans \ref{prop:modele_c_ftr},
 le couple de foncteurs \eqref{adj_ftr_Tmod}
  est de manière évidente une adjonction de Quillen qui se dérive donc
  et induit une adjonction de catégories triangulées:
\begin{equation} \label{adj_Dftr_DTmod}
\derL F_i:\Der(\ftr) \leftrightarrows \Der(\Tmod):\derR Ev_i.
\end{equation}
Soit $X$ un schéma lisse.
Comme $\rep X$ est cofibrant,
 $\derL F_i(\rep X)=T^*(X)\{-i\}$.
On en déduit donc que pour tout spectre motivique $E$,
\begin{equation} \label{Hom_Der_Tmod}
\Hom_{\Der(\Tmod)}(T^*(X)\{-i\},E[n])=H^n(X,E_i).
\end{equation}
Par ailleurs, on vérifie en utilisant ces mêmes structures
de catégorie de modèles que l'action à gauche de $\ftr$ sur $\Tmod$ 
définie en \eqref{left_actn_ftr} s'étend en une action à gauche
de $\Der(\ftr)$ sur $\Der(\Tmod)$.

\num On peut définir la catégorie dérivée $\AA^1$-localisée
de $\Tmod$ en inversant la classe de flèches $\W_{\AA^1}$ engendrée 
par
$$
T^*(\AA^1_X)\{-i\} \rightarrow T^*(X)\{-i\}
$$
pour un schéma lisse $X$ et un entier $i \in \ZZ$
dans la catégorie $\Der(\Tmod)$.
On la note $\Der_{\AA^1}(\Tmod)$. On dispose donc
d'un foncteur canonique
$$
\pi:\Der(\Tmod) \rightarrow \Der_{\AA^1}(\Tmod).
$$
On peut considérer la localisation de Bousfield à gauche
 de la catégorie de modèles sur $\Comp(\ftr)$
 définie dans le lemme précédent dont la catégorie
 homotopique associée est $\Der_{\AA^1}(\Tmod)$\footnote{Ce qui
 est une façon de montrer que la catégorie localisée 
 $\Der_{\AA^1}(\Tmod)$ est bien définie.}.
On l'appelle la catégorie de modèles $\AA^1$-locale.
On obtient par la théorie des localisations de Bousfield
 un adjoint à droite du foncteur $\pi$
$$
\Der_{\AA^1}(\Tmod) \rightarrow \Der(\Tmod)
$$
qui est pleinement fidèle et dont l'image essentielle
s'identifie aux objets $\W_{\AA^1}$-locaux,
que l'on appelle spectres $\AA^1$-locaux.
Compte tenu de \eqref{Hom_Der_Tmod}, 
un spectre motivique $(E_n)_{n \in \NN}$ est $\AA^1$-local si
pour tout $n \in \NN$, le complexe $E_n$ est $\AA^1$-local.
Il est immédiat que l'adjonction \eqref{adj_Dftr_DTmod}
passe aux catégories $\AA^1$-localisées et induit
\begin{equation} \label{adj_DAftr_DATmod}
\derL_{\AA^1} F_i:\DMe \leftrightarrows \Der_{\AA^1}(\Tmod):\derR_{\AA^1} Ev_i.
\end{equation}

\num Pour tout schéma lisse $X$ et tout entier $i \geq 0$,
on obtient (en utilisant la symétrie du produit tensoriel sur $\ftr$)
un morphisme canonique de faisceaux gradués
$$
\left(T \otr T^*(X)\right)\{-i-1\} \rightarrow T^*(X)\{-i\}
$$
compatible à la structure de $T^*$-module.
Notons $\W_{\mathfrak S}$ la classe de flèches induites
dans la catégorie des spectres de Tate.
\begin{df}
On note $\DM$
 la localisation de la catégorie $\Der_{\AA^1}(\Tmod)$
 par rapport à la classe de flèches $\W_{\mathfrak S}$.
\end{df}
En utilisant à nouveau une localisation de Bousfield
à gauche de la catégorie de modèles $\AA^1$-locale 
sur $\Comp(\Tmod)$ par rapport à $\W_\mathfrak S$,
on obtient la catégorie de modèles dite \emph{stable}
sur $\Comp(\Tmod)$. Il est immédiat que celle-ci coïncide
avec celle introduite par Hovey dans \cite[3.3]{Hov} à partir
de la catégorie de modèles $\AA^1$-locale sur $\Comp(\ftr)$.
Appliquant \cite[3.4]{Hov}, on en déduit 
qu'un spectre de Tate $(E_n)$ est fibrant pour la structure
stable si et seulement si pour tout $n \in \NN$, 
le complexe $E_n$ est Nisnevich fibrant, $\AA^1$-local
et le morphisme structural $E_n \rightarrow \Omega E_{n+1}$
est un quasi-isomorphisme.
\begin{df}
Un spectre de Tate $(E_n)$ est appelé un \emph{spectre motivique}
si pour tout $n \in \ZZ$, $E_n$ est $\AA^1$-local
et si le morphisme structural $E_n \rightarrow \Omega E_{n+1}$
est un quasi-isomorphisme.
\end{df}
La catégorie $\DM$ est donc équivalente à la sous-catégorie pleine de
$\Der(\Tmod)$ formée des spectres motiviques. 
On considèrera désormais que 
 les objets de $\DM$ sont des spectres motiviques.

L'adjonction \eqref{adj_DAftr_DATmod} pour $i=0$ permet de définir
une adjonction
\begin{equation} \label{adj_DMe_DM}
\Stab:\DMe \leftrightarrows \DM:\Lop,
\end{equation}
qui n'est autre que l'adjonction dérivée de \eqref{adj_ftr_Tmod} pour
 $i=0$.
Puisqu'un spectre motivique $(E_n)$ est fibrant pour la structure
 de catégorie de modèles précédente, 
 on en déduit que le foncteur $\Lop$ associe à $(E_n)$ 
 le complexe motivique $E_0$.

\num \label{num:simplification}
Soit $C$ un complexe motivique, et $C' \rightarrow C$ une
 résolution cofibrante au sens de la catégorie de modèles
 \ref{A^1-eff-modele}.
On peut donner la description suivante du foncteur $\Stab$:
pour un entier $n \geq 0$, on définit a priori le foncteur
$\Stab$ en posant
$$
(\Stab C)_n=C\{n\}, n \geq 0
$$
où $C\{n\}$ désigne le complexe motivique 
$$
L_{\AA^1}(C \otrL T^n)=L_{\AA^1}(C' \otr T^n)
$$
conformément à la notation \ref{twist_Tate_suspendu}.
On définit de plus une structure de spectre de Tate sur $\Stab C$ 
grâce au morphisme d'adjonction suivant, pour un entier $n \geq 0$,
$$
C\{n\} \rightarrow \Omega (C\{n+1\}).
$$
Il résulte du théorème de simplification \ref{simplification_DM}
 que $\Stab C$ est alors un spectre motivique.
Par définition, $\Stab C$ est donc une résolution fibrante
 du spectre de Tate $F_0(C')$, ce qui montre qu'il coincide
 avec l'adjoint à gauche de \eqref{adj_DMe_DM}.
De cette description, on déduit que le morphisme d'adjonction
$$
C \rightarrow \Omega^\infty \Stab C
$$
est un isomorphisme: le foncteur $\Stab$ est pleinement fidèle.

\num L'action à gauche de la catégorie monoïdale $\Der(\ftr)$ sur $\Der(\Tmod)$
 peut être étendue en une action à gauche de $\DMe$ sur $\DM$. 
D'après \cite[3.8]{Hov}, l'action du complexe motivique $\un\{1\}$ sur $\DM$
 est inversible et le foncteur quasi-inverse est induit par le foncteur
 de décalage qui à un $\Omega$-spectre $E$ associe le spectre
 motivique $E\{-1\}$ - qui est encore un spectre motivique.
\begin{prop} \label{prop:eq_sp_sym&non_sym}
La catégorie $\DM$ introduite ici ainsi que l'adjonction
\eqref{adj_DMe_DM} coïncident avec celles de \cite[ex. 6.25]{CD1}.
\end{prop}
Il en résulte que la catégorie triangulée $\DM$ est munie d'une 
structure monoïdale symétrique fermée telle que le foncteur $\Stab$
est monoïdal symétrique. Par rapport au cas général traité dans
\emph{loc. cit.} on a obtenu ici que le foncteur $\Stab$ est pleinement
fidèle en utilisant le théorème de simplification \ref{simplification_DM}.
\begin{proof}
Il s'agit essentiellement de comparer la construction de $\DM$
donnée ici par les spectres motiviques avec celle \emph{loc. cit.} donnée
par les spectres motiviques symétriques.
D'après la preuve du lemme \ref{lm:permutation_tate},
 il existe une $\AA^1$-équivalence d'homotopie forte entre la permutation
 cyclique des facteurs de $T^3$ et l'identité dans la catégorie
 $\Comp(\ftr)$.
Il en résulte que $T$ est un objet symétrique de la catégorie
 de modèle $\AA^1$-locale $\Comp(\ftr)$ au sens de la définition
 \cite[9.2]{Hov}.
Dès lors, on peut appliquer \cite[9.4]{Hov} ce qui permet de
conclure.
\end{proof}

\num \label{inclusion_DMgm_DM}
Notons pour terminer que le foncteur pleinement fidèle
monoidal $\DMgme \rightarrow \DMe \rightarrow \DM$
s'étend par propriété universelle en un foncteur pleinement fidèle monoidal
$$
\DMgm \rightarrow \DM.
$$
L'image essentielle
de ce foncteur coïncide avec la catégorie des objets compacts
de $\DM$ d'après \cite{CD1}.

\subsection{t-structure homotopique}

\num \label{notation_H_DM}
Soit $E=(E_n,\epsilon_n:E_n \rightarrow \Omega E_{n+1})_{n \geq 0}$
 un spectre motivique.
On lui associe un unique module homotopique $\ugH*^0(M)$ tel que
pour tout $n \geq 0$,
$\ugH n^0(E)=\uH^0(E_n)$ avec la notation de \eqref{H^0_DMe}
et avec pour morphismes structuraux
$$
\uH^0(E_n)
 \xrightarrow{H^0(\epsilon_n)}
  \uH^0(\Omega E_{n+1}) = [\uH^0(E_{n+1})]_{-1}
$$
en appliquant le corollaire \ref{cor_uH&Omega}.
La graduation négative de $\ugH n^0(E)$ est définie de manière
 tautologique.
Pour tout entier $m \in \ZZ$, on pose $\ugH*^m(E)=\ugH*^0(E[m])$.
\begin{lm}
Les foncteurs définis ci-dessus vérifient les propriétés
 suivantes:
\begin{enumerate}
\item[(i)] Le foncteur $\ugH*^0:\DM \rightarrow \hmtr$
 est un foncteur cohomologique qui commute aux sommes quelconques.
\item[(ii)] La famille de foncteurs $(\ugH*^m)_{m \in \ZZ}$ est 
 conservative.
\end{enumerate}
\end{lm}
\begin{proof}
La catégorie $\Der(\ftrgr)$ porte naturellement une t-structure,
 puisque $\ftrgr$ est abélienne.
On note ${\underline{\tilde{H}}}_*^0$ le foncteur cohomologique associé.
Par définition, on obtient un carré commutatif
$$
\xymatrix{
\DM\ar^-{(2)}[r]\ar_{\ugH*^0}[d] & \Der(\ftrgr)\ar^{{\underline{\tilde{H}}}_*^0}[d] \\
\hmtr\ar^-{(1)}[r] & \NN\!-\!\ftr.
}
$$
où la flèche $(1)$ est la flèche d'oubli évidente
 et la flèche $(2)$ est induite par le foncteur \eqref{O_DTmod_Dftrgr}.
D'après \ref{construction_base_hmtr}, le foncteur (1)
 est exact et conservatif. Dès lors, comme ${\underline{\tilde{H}}}_*^0$
 envoie les triangles distingués sur des suites exactes
 longues, il en est de même de $\ugH*^0$.
Le fait que $\ugH*^0$ commute aux sommes infinies est 
 maintenant évident, ce qui implique (i).
L'assertion (ii) résulte maintenant du fait que $(2)$ est conservatif
 et du fait que ${\underline{\tilde{H}}}_*^0$ est le foncteur cohomologique
 d'une t-structure sur $\Der(\ftrgr)$.
\end{proof}

On dit qu'un spectre motivique est positif (resp. négatif)
 si pour tout $n<0$ (resp. $n>0$), $\ugH*^n(E)=0$.
Soit $\tau_{\leq 0}$ le foncteur de troncation négative pour la
 t-structure homotopique sur $\DMe$.
On vérifie en utilisant à nouveau le corollaire
\ref{cor_uH&Omega} que l'application de $\tau_{\leq 0}$
 degré par degré à un spectre motivique $E$ définit
 un spectre motivique négatif $\tau_{\leq 0} E$,
 avec un morphisme canonique:
$$
\tau_{\leq 0} E \rightarrow E.
$$
\begin{cor}
La catégorie $\DM$, munie de la notion d'objets négatifs
et positifs introduite ci-dessus, est une t-structure
dont le foncteur de troncation négatif est le foncteur $\tau_{\leq 0}$
introduit ci-dessus et dont le foncteur cohomologique associé est
le foncteur $\ugH*^0$.
\end{cor}
On appelle cette t-structure la t-structure homotopique
 sur $\DM$. 
Notons que d'après la définition de \ref{notation_H_DM},
le diagramme suivant est commutatif:
$$
\xymatrix@C=30pt@R=12pt{
\DM\ar^{\uH^0_*}[r]\ar_{\Lop}[d] & \hmtr\ar^{\hLop}[d] \\
\DMe\ar^-{\uH^0}[r] & \hftr.
}
$$
Il en résulte que $\Lop$ est $t$-exact.
 
\subsection{Coeur homotopique}

\begin{num} \label{num:mhtr_rep}
D'après la construction précédente,
 la catégorie des modules homotopiques
 $\hmtr$ est le coeur de $\DM$ pour la t-structure homotopique.

Avec les notations de la première partie,
 le module homotopique représenté par un schéma lisse $X$
 est égal à $h_{0,*}(X)=\uH^0_*(M(X))$.
Pour tout entier $n \in \ZZ$, on obtient aussi
 une identification de modules homotopiques
 $h_{0,*}(X)\{n\}=\uH_*^0\big(M(X)\{n\}\big)$\footnote{On fera attention
  à la confusion possible introduite par cette notation: étant donné un module
  homotopique $F_*$ et un entier $n>0$, l'objet twisté $F_*\{n\}$ 
  n'est pas le même s'il est calculé dans la catégorie $\hmtr$ ou $\DM$.
  C'est pourquoi on précise ici que $h_{0,*}(X)\{n\}$ est considéré
  comme un module homotopique.}.

Notons plus généralement que tout schéma algébrique $X$
 définit d'après \cite{V1} un complexe motivique $\sing \rep X$.
Pour tout entier $i\geq 0$,
 on obtient un module homotopique:
$$
h_{i,*}(X):=\uH_*^{-i}(\Stab \sing \rep X).
$$
\end{num}

Si $E$ est un spectre motivique,
 on lui associe pour tout entier $p \in \ZZ$
 un module de cycles $\umH^p_*(E)$
 obtenu en appliquant le foncteur \emph{transformée générique}
 de \ref{thm:main} au module homotopique $\uH^p_*(E)$.
Pour tout corps de fonction $L$, et tout entier $n \in \ZZ$,
$$
\umH^p_n(E).L=\ilim{A \subset L} \Hom_{\DM}(M(\spec A),E\{n\}[p]).
$$
Compte tenu du théorème \ref{thm:main},
 on obtient donc le corollaire suivant:
\begin{cor} \label{cor:motifs->modl}
La catégorie des modules de cycles est le coeur
 de la t-structure homotopique sur $\DM$,
 via le foncteur $\umH^0_*$.
\end{cor}

Ainsi, on peut associer à tout schéma algébrique $X$
 et tout entier $i \in \NN$
 un module de cycles $\hat h_{i,*}(X)$.
Pour tout corps de fontions $L$,
le gradué de degré $0$ de ce module de cycles est donné par l'homologie de Suslin de $X$:
$$
\hat h_{i,0}(X).L=H_i^{sing}(X_L/L)
$$
avec les notations de \cite{SV1}. \\
Si $X$ est projectif lisse connexe de dimension $d$,
 le motif $M(X)=\Stab \sing \rep X$ dans $\DM$
 est fortement dualisable avec pour dual fort $M(X)(-d)[-2d]$
 d'après \cite[2.16]{Deg6}.
Il en résulte que pour tout corps de fonctions $L$,
\begin{equation} \label{eq:rep_dual_cohm}
\hat h_{i,n}(X).L=H^{2d+i+n,d+n}_\cM(X_L),
\end{equation}
où $H^{s,t}_\cM(X_L)$ désigne la cohomologie motivique
 de $X$ étendu à $L$ en degré $s$ et twist $t$.
 
\num Considérons un corps de fonctions $L$ et un entier $n \in \ZZ$.
On lui associe un pro-objet de $\hmtr$:
$$
h_{0,*}(L)\{n\}=\pplim{A \subset L} h_{0,*}(\spec A)\{n\},
$$
où $A$ parcourt l'ensemble ordonné filtrant des sous-$k$-algèbres
de type fini de $L$ dont le corps des fractions est $L$.
\begin{lm}
Pour tout spectre motivique $E$, et tout entier $i \in \ZZ$,
$$
\Hom_{\pro\!-\!\DM}(M(L)\{n\},E[i])
 =\Hom_{\pro\!-\!\hmtr}(h_0(L)\{n\},\uH^i_*(E)).
$$
\end{lm}
Utilisant le théorème de simplification,
 on se réduit au cas effectif qui résulte de \cite[3.4.4]{Deg6}
 -- il s'agit essentiellement du fait que le pro-objet $(L)$
  est un point pour la topologie de Nisnevich.

Notons $\hmtro$ la sous-catégorie pleine de $\pro\!-\!\hmtr$
 formée des objets de la forme $h_{0,*}(L)\{n\}$ pour un corps
 de fonctions $L$ et un entier $n \in \ZZ$.
Alors, d'après le lemme précédent, le foncteur canonique
$$
\DMgmo \rightarrow \hmtro, M(L)\{n\} \mapsto \uH^0_*(M(L)\{n\})
$$
est une équivalence de catégories\footnote{Le cas effectif avait
 déjà été traité dans \cite[3.4.7]{Deg5}.}.

Ainsi, les motifs génériques apparaissent comme des pro-objets
 de la catégorie abélienne $\hmtr$, qui définissent des foncteurs
 fibres de cette catégorie
  (\emph{i.e.} exacts, commutant aux sommes infinies).
Cette interprétation des motifs génériques est donc très proche
 de la notion de points d'un topos.
La transformée générique d'un module homotopique $F_*$ est finalement
 donnée par la \emph{restriction de $F_*$ à cette catégorie de points}.

\rem Notons que les morphismes de spécialisations correspondants aux données (D1)
 à (D4) des modules de cycles sont donc définis dans la catégorie
 abélienne des pro-modules homotopiques.
A titre d'illustration, considérons un corps de fonctions valué $(L,v)$.
Soit $\kappa_v$ son corps résiduel
 et $\varphi:\mathcal O_v \rightarrow L$ l'inclusion de son anneau des
 entiers.
On obtient par application du fonteur $H^0_*$ au triangle de Gysin
 une suite exacte courte dans $\pro\!-\!\hmtr$:
$$
h_{0,*}(\kappa_v)\{1\}
 \xrightarrow{\partial_v} h_{0,*}(L)
  \xrightarrow{\varphi^*} h_{0,*}(\mathcal O_v) \rightarrow 0.
$$

\section{Suite spectrale de Gersten et t-structure homotopique}

\num \label{num:ssp}
Soit $\A$ une catégorie abélienne.

Rappelons qu'un couple exact dans $\A$ est la donnée d'objets
bigradués $E_1^{p,q}$ et $D_1^{p,q}$, pour des indices $(p,q) \in \ZZ^2$
et des morphismes homogènes
$$
\xymatrix@C=20pt@R=22pt{
D_1\ar^{(1,-1)}_\alpha[rr]
 &&  D_1\ar^{(0,0)}_/8pt/\beta[ld] \\
 & E_1\ar^{(1,0)}_/-6pt/\gamma[lu] &
}
$$
dont les bidegrés sont indiqués sur le diagramme.
Rappelons (\emph{cf.} \cite[2.3]{McC}) que l'on associe à un tel couple
exact une suite spectrale dont la première page est $E_1^{p,q}$
avec pour différentielles les morphismes
$d_1=\gamma \circ \beta$.

Considérons maintenant un complexe $K$ de $\A$.
On suppose donnés pour tout entier $p\in \ZZ$ les complexes et morphismes
suivants:
\begin{equation} \label{proto_ce}
\xymatrix@C=16pt@R=22pt{
& F^{p+1}K\ar@{^(->}^{f^p}[rr]\ar@{^(->}_{i^{p+1}}[ld]
 && F^pK\ar@{^(->}^{i^{p}}[rd]\ar@{->>}^{\pi^p}[ld] & \\
K\ar@{->>}_{k^{p+1}}[rd] && G^pK\ar@{^(->}_{\tilde \pi^p}[ld]
 && K\ar@{->>}^{k^p}[ld] \\
& T^{p+1}K\ar@{->>}_{\tilde f^p}[rr] & & T^pK
}
\end{equation}
On demande que pour tout $p\in \ZZ$,
 les couples $(i^p,k^p)$, $(f^p,\pi^p)$, $(\tilde f^p,\tilde \pi^p)$
 forment des suites exactes courtes dans la catégorie abélienne
 $\Comp(\A)$.
On notera en particulier que $F^*K$ (resp. $T^*K$) 
 définit une filtration (resp. cofiltration) de $K$.
Bien entendu, l'une détermine l'autre.

Il en résulte que, passant à la catégorie dérivée
 $\T:=\Der(\A)$,
 on obtient le diagramme suivant
\begin{equation} \label{octaedre}
\xymatrix@R=16pt@R=22pt{
& F^{p+1}K\ar^{f^p}[rr]\ar_{i^{p+1}}[ld]
 && F^pK\ar^{i^{p}}[rd]\ar^{\pi^p}[ld] & \\
K\ar_{k^{p+1}}[rd]\ar@{}|{*}[r]
 && G^pK\ar_{\tilde \pi^p}[ld]\ar^{^{+1}}[lu]\ar@{}|{*}[d]\ar@{}|{*}[u]
 && K\ar^{k^p}[ld]\ar@{}|{*}[l] \\
& T^{p+1}K\ar_{\tilde f^p}[rr]\ar^/6pt/{^{+1}}[uu]
 && T^pK\ar_/6pt/{^{+1}}[uu]\ar^/6pt/{^{+1}}[lu]
}
\end{equation}
dans lequel les triangles marqués d'une étoile sont distingués
 et les autres sont commutatifs. Autrement dit, on obtient un \emph{octaèdre}
 dans la catégorie triangulée $\T$.

Supposons donné par ailleurs un foncteur (co)homologique
 $\varphi:\Der(\A) \rightarrow \B$, et posons
 $\varphi^n=\varphi(.[n])$.
On peut alors définir des objets bigradués:
$$
D_1^{p,q}=\varphi^{p+q}(F^pK),
 \tilde D_1^{p,q}=\varphi^{p+q}(T^{p+1}K),
E_1^{p,q}=\varphi^{p+q}(G^pK),
 E_\infty^{p,q}=\varphi^{p+q}(K).
$$
pour $(p,q) \in \ZZ^2$.
Par application du foncteur cohomologique $\varphi$ 
 au diagramme précédent, on obtient un diagramme
 commutatif d'objets bigradués, 
 formé de morphismes homogènes,
$$
\xymatrix@R=18pt@C=24pt{
& D_1\ar^{\alpha}[rr]\ar_{}[ld]
 && D_1\ar^{}[rd]\ar^{\beta}[ld] & \\
E_\infty\ar_{}[rd]\ar@{}|{}[r]
 && E_1\ar_{\tilde \gamma}[ld]\ar^{\gamma}[lu]\ar@{}|/3pt/{(2)}[d]\ar@{}|/3pt/{(1)}[u]
 && E_\infty\ar^{}[ld]\ar@{}|{}[l] \\
& \tilde D_1\ar_{\tilde \alpha}[rr]\ar^{r}[uu]
 && \tilde D_1\ar_{r}[uu]\ar_{\tilde \beta}[lu]
}
$$
dans lequel les triangles $(1)$ et $(2)$ forment un couple exact
(avec les conventions rappelées ci-dessus).
Ce diagramme commutatif\footnote{On reconnaitra un cas particulier
 de \og système de Rees \fg~suivant la terminologie introduite 
 par Eilenberg et Moore
 (\emph{cf.} \cite[3.1]{McC}).}
 montre que les suites spectrales
 associées respectivement à (1) et (2) sont \emph{égales}
 --- l'assertion concernant les différentielles de la première page 
  est par exemple immédiate.
Notons enfin que lorsque la filtration $F^*K$ est bornée\footnote{Il
 est probable si la suite spectrale dégénère en $(p,q)$,
  $E_r^{p,q} \simeq E_\infty^{p,q}$ pour $r>>0$.}
(ou ce qui revient au même $T^*K$), le terme à l'infini
de cette suite spectrale est ce que nous avons noté $E_\infty$
et la suite spectrale converge.

\rem
\begin{enumerate}
\item Un cas particulier fondamental est celui où le foncteur 
$\varphi$ est le foncteur (co)homologique $H^0:\Der(\A) \rightarrow \A$ canonique.
La suite spectrale obtenue ci-dessus est alors la suite spectrale
du complexe filtré $(K,F^*K)$ (\emph{cf.} \cite{Del}).
\item Suivant la construction ci-dessus, on reconnait donc dans
la donnée d'un système de Rees la trace d'un octaèdre,
en l'occurence le diagramme \eqref{octaedre}.
\item Plus généralement, c'est la famille des diagrammes \eqref{octaedre} 
qui est fondamentale pour définir la suite spectrale précédente. Ainsi,
le procédé décrit ci-dessus à partir de ces diagrammes a un sens
dans nimporte quelle catégorie triangulée, indépendamment de la manière
dont on a donné naissance à ces diagrammes. Si on se place par
ailleurs dans une catégorie \og triangulée enrichie \fg~$\T$
-- c'est-à-dire une $\infty$-catégorie stable dans le sens de \cite{Lurie}, 
comme par exemple la catégorie homotopique d'une DG-catégorie, 
ou encore la catégorie homotopique d'une catégorie de modèles stable --
on peut associer canoniquement à la donnée des morphismes $(f^p,i^p)$
(resp. $(\tilde f^p,\tilde i^p)$),
des diagrammes du type \eqref{octaedre} en utilisant le foncteur
\emph{cofibre homotopique} (resp. \emph{fibre homotopique}).
\item Remarquons que tout foncteur triangulé $\psi:\T' \rightarrow \T$
envoie une famille de diagrammes \eqref{octaedre} sur une famille
du même type. 
En général, le foncteur (co)homologique $\varphi$ se décompose
en $H^0 \circ \psi$ où $\psi$ est un foncteur triangulé et
$H^0$ est le foncteur (co)homologique d'une t-structure donnée sur $\T$.
Dans le cadre qui suit, $\psi$ est un foncteur dérivé (à droite).
Les catégories $\T$ et $\T'$ sont les catégories homotopiques
de catégories de modèles stables. Il y a lieu dans ce cas
de remplacer l'hypothèse que $\tilde f^p$ est induit par
un épimorphisme de complexes par l'hypothèse que c'est une fibration\footnote{
On voir alors $G^p$ comme la cofibre homotopique de $\tilde f^p$,
qui est un objet fibrant. L'avantage est qu'il n'y a alors par lieu de dériver
le foncteur de Quillen à droite sous-jacent à $\psi$.}
pour la catégorie de modèles sous-jacente à $\T$.
Ce cadre correspond à une \og tour de fibrations \fg~et à la suite spectrale 
qui lui est associée en topologie algébrique. La question de la convergence
de cette suite spectrale est alors reliée au problème de savoir si la flèche
canonique
$$
K \rightarrow \mathrm{holim}_{p \in \ZZ} T^pK
$$
est une équivalence faible.

Remarquons que dualement,
si $\psi$ est un foncteur dérivé à gauche, il y a lieu de supposer
que $f^p$ est une cofibration ; ce cas correspond dans le cadre d'une catégorie
de modèles abstraite au cas particuliers des complexes filtrés dans une catégorie
dérivée. Dans ce cas, la convergence est reliée à la flèche canonique
$$
\mathrm{hocolim}_{p \in \ZZ} F^pK \rightarrow K.
$$
\end{enumerate}

\num \label{2ssp}
Pour les deux prochains exemples, on fixe un schéma lisse $X$
 et un spectre motivique $\E$. On pose $\E_0=\Omega^\infty \E$
 vu comme complexe de faisceau Nisnevich sur $\sm$.

\noindent \textit{Suite spectrale n°1}.--
La t-structure homotopique sur $\DM$
permet d'obtenir un diagramme du type \eqref{octaedre} dans la
catégorie triangulée $\DM$:
$$
\xymatrix@C=16pt@R=20pt{
& \tau_{\leq -p-1} \E\ar[rr]\ar[ld]
 && \tau_{\leq -p} \E\ar[rd]\ar[ld] & \\
\E\ar[rd]\ar@{}|{*}[r]
 && \uH^{-p}_*(\E)[p]\ar[ld]\ar_/-5pt/{_{+1}}[lu]
      \ar@{}|{*}[u]\ar@{}|{*}[d]
 && \E\ar[ld]\ar@{}|{*}[l] \\
& \tau_{> -p-1}\ar[rr]\ar^/8pt/{_{+1}}[uu]
 && \tau_{> -p-1}\ar_/8pt/{_{+1}}[uu]\ar^/2pt/{_{+1}}[lu]
}
$$
Si $X$ est un schéma lisse, on obtient donc après application
du foncteur (co)homologique $\Hom(M(X),.)$ un couple exact et une
suite spectrale:
$$
'E_{1,t}^{p,q}=\Hom(M(X),\uH^{-p}_*(\E)[2p+q])
 \Rightarrow \Hom(M(X),\E[p+q])
$$
qui est la suite spectrale d'hypercohomologie associée
à la t-structure homotopique.

Suivant l'usage, on renumérote cette suite spectrale pour qu'elle
commence au terme $E_2$ suivant la règle $E_{2,t}^{p,q}={}'E_{1,t}^{-q,p+2q}$.
Un petit calcul donne alors la forme finale suivante pour cette
suite spectrale:
$$
E_{2,t}^{p,q}=H^p(X,\uH^q_0 \E) \Rightarrow H^{p+q}(X,\E_0).
$$
Remarquons que cette suite spectrale est convergente puisqu'elle
est concentrée dans la colonne $0 \leq p \leq \dim X$.

\noindent \textit{Suite spectrale n°2}.--
Rappelons\footnote{Cette définition est classique; 
on se réfère à \cite[section 3]{Deg5}, pour plus de détails.} 
qu'un drapeau de $X$ est une suite
décroissante $(Z^p)_{p \in \NN}$ de sous-schémas fermés de $X$ telle
que $\mathrm{codim}_X(Z^p) \geq p$. L'ensemble des drapeaux de $X$,
 ordonné par l'inclusion termes à termes, est filtrant.
Etant donné un tel drapeau,
on peut considérer le diagramme suivant dans la catégorie des faisceaux avec transferts:
$$
\xymatrix@C=-10pt@R=20pt{
& \rep{X-Z^p}\ar@{^(->}^{f_{p*}}[rr]\ar@{^(->}_/3pt/{i_{p*}}[ld]
 && \rep{X-Z^{p+1}}\ar@{^(->}^/3pt/{i_{p+1*}}[rd]\ar@{->>}^/-10pt/{\pi_p}[ld] & \\
\rep X\ar@{->>}_/-6pt/{k_{p}}[rd]
 && \rep{X-Z^{p+1}/X-Z^p}\ar@{^(->}_/6pt/{\tilde \pi_p}[ld]
 && \rep X\ar@{->>}^/-6pt/{k_{p+1}}[ld] \\
& \rep{X/X-Z^p}\ar@{->>}_{\tilde f_p}[rr] & & \rep{X/X-Z^{p+1}}
}
$$
Les morphismes $f^p$ et $i^p$ désignent les immersions ouvertes canoniques.
On obtient donc un diagramme du type \eqref{proto_ce}, à ceci près que
la filtration donnée par $f_{p*}$ est décroissante.
Ce diagramme est naturellement fonctoriel (contravariant) par rapport
 à l'inclusion des drapeaux.
Il induit donc un diagramme commutatif du type \eqref{octaedre}
dans la catégorie $\Der(\ftr)$. En prenant son image par le foncteur
triangulé $\Der(\ftr) \rightarrow \DMe \rightarrow \DM$,
on obtient donc un diagramme de la forme suivante:
$$
\xymatrix@C=-10pt@R=20pt{
& M(X-Z^p)\ar^{f_{p*}}[rr]\ar_/3pt/{i_{p*}}[ld]
 && M(X-Z^{p+1})\ar^/3pt/{i_{p+1*}}[rd]\ar^/-10pt/{\pi_p}[ld] & \\
M(X)\ar_/-6pt/{k_{p}}[rd]\ar@{}|{*}[r]
 && M(X-Z^{p+1}/X-Z^p)\ar_/6pt/{\tilde \pi_p}[ld]\ar_/-5pt/{_{+1}}[lu]
      \ar@{}|{*}[u]\ar@{}|{*}[d]
 && M(X)\ar^/-6pt/{k_{p+1}}[ld]\ar@{}|{*}[l] \\
& M(X/X-Z^p)\ar_{\tilde f_p}[rr]\ar^/8pt/{_{+1}}[uu]
 && M(X/X-Z^{p+1})\ar_/8pt/{_{+1}}[uu]\ar^/2pt/{_{+1}}[lu]
}
$$
Considérons un spectre motivique $\E$. Appliquant le foncteur
$\derR \Hom_{\DM}(.,\E)$ au diagramme précédent,
on obtient un diagramme dans la catégorie triangulée
$\Der(\ab)$ qui est précisément de la forme \eqref{octaedre}
où l'on a posé:
\begin{align*}
K&=\derR \Hom(M(X),\E), 
&F^pK&=\derR \Hom(M(X/X-Z^p),\E), \\
T^pK&=\derR \Hom(M(X-Z^p),\E),
&G^pK&=\derR \Hom(M(X-Z^{p+1}/X-Z^{p}),\E).
\end{align*}
Le diagramme ainsi obtenu est naturel covariant par
rapport aux inclusions de drapeaux.
Comme les limites inductives filtrantes sont exactes
dans $\Der(\ab)$, on en déduit un diagramme de la forme 
\eqref{octaedre} avec:
\begin{align*}
K&=\derR \Hom(M(X),\E), \\
F^p_cK&=\ilim{Z^* \in \Drap(X)} \derR \Hom(M(X/X-Z^p),\E) \\
T^p_cK&=\ilim{Z^* \in \Drap(X)}\derR \Hom(M(X-Z^p),\E), \\
G^p_cK&=\ilim{Z^* \in \Drap(X)}\derR \Hom(M(X-Z^{p+1}/X-Z^{p}),\E).
\end{align*}
Après application du foncteur cohomologique canonique,
on en déduit donc une suite spectrale:
$$
E_{1,c}^{p,q}=\ilim{Z^* \in \Drap(X)} \E^{p+q}(X-Z^{p+1}/X-Z^{p})
 \Rightarrow \E^{p+q}(X)
$$
qui n'est autre que le suite spectrale du coniveau associée
à la théorie cohomologique représentée par $\E$ (voir \cite{BO}
pour la référence originale). Notons que cette suite spectrale
est convergente puisqu'un drapeau de $X$ est de longeur
bornée par la dimension de $X$.

Dans \cite[4.5]{Deg5}, on démontre qu'il existe un isomorphisme
canonique de complexes de groupes abéliens:
$$
E_{1,c}^{*,q} \simeq C^*(X,\umH^q_*(\E))_0
$$
où le membre de droite est le complexe de cycles (en degré $0$) 
à coefficients dans le module de cycles $\umH^q_*(\E)$ 
(\emph{cf.} corollaire \ref{cor:motifs->modl}).

On obtient donc la forme suivante de la suite spectrale
du coniveau:
$$
E_{2,c}^{p,q}=A^p(X,\umH^q_*(\E))_0 \Rightarrow H^{p+q}(X,\E_0)
$$

D'après la proposition \ref{iso_coh_gpe_chow},
on en déduit donc un isomorphisme
$$
\varphi_2^{p,q}:E_{2,t}^{p,q} \rightarrow E_{2,c}^{p,q}.
$$
\begin{thm} \label{thm:comparaison_ssp}
La famille d'isomorphismes $\varphi_2^{p,q}$ est compatible
aux différentielles des deux suites spectrales définies
ci-dessus.

De plus, elle induit de proche en proche des isomorphismes
compatibles aux différentielles
$\varphi_r^{p,q}:E_{r,t}^{p,q} \rightarrow E_{r,c}^{p,q}$
 pour tout $r\geq 2$.
\end{thm}
\begin{proof}
Puisque les deux suites spectrales sont concentrées en degrés
 $0 \leq p \leq \dim X$, on peut supposer que $E$
 est borné inférieurement
 pour la t-structure homotopique.

Soit $X_\nis$ le site des $X$-schémas étales muni de la topologie
de Nisnevich. Soit $\tilde X_\nis$ la catégorie des
faisceaux abéliens sur $X_\nis$.
Si $V/X$ est un schéma étale, on note $\ZZ_X(V)$
 le faisceau de groupes abéliens libres sur $X_\nis$
 représenté par $V$. \\
Soit $K$ la restriction de $\E_0$ à $X_\nis$. 
Par hypothèse, $K$ est borné inférieurement. Il existe
donc une résolution de Cartan-Eilenberg $I \rightarrow K$
où $I$ est un complexe quasi-isomorphe à $K$, dont les composantes
sont des objets injectifs de $\tilde X_\nis$.
Ainsi, pour tout $X$-schéma étale $V$,
 on dispose d'un isomorphisme canonique:
\begin{equation}\label{IetE}
\Hom_{\Der(\tilde X_\nis)}(\ZZ_X(V),I[n])
 \rightarrow \Hom_{\DM}(M(V),\E[n]).
\end{equation}

Puisque le foncteur de restriction à $X_\nis$ est exact,
la suite spectrale $E_{2,t}^{p,q}$ est canoniquement
isomorphe à la suite spectrale d'hypercohomologie Nisnevich 
du complexe $I$ -- \emph{i.e.} associée au foncteur cohomologique
$\derR \Gamma$, foncteur dérivé du foncteur sections globales. 


Soit $(Z^p)_{p \in \NN}$ un drapeau de $X$. On obtient alors un diagramme
du type \eqref{proto_ce} dans la catégorie $\tilde X_\nis$:
\newcommand{\repz}[1]{\ZZ_X(#1)}
$$
\xymatrix@C=-10pt@R=20pt{
& \repz{X-Z^p}\ar@{^(->}^{f_{p*}}[rr]\ar@{^(->}_/3pt/{i_{p*}}[ld]
 && \repz{X-Z^{p+1}}\ar@{^(->}^/3pt/{i_{p+1*}}[rd]\ar@{->>}^/-10pt/{\pi_p}[ld] & \\
\repz X\ar@{->>}_/-6pt/{k_{p}}[rd]
 && \repz{X-Z^{p+1}/X-Z^p}\ar@{^(->}_/6pt/{\tilde \pi_p}[ld]
 && \repz X\ar@{->>}^/-6pt/{k_{p+1}}[ld] \\
& \repz{X/X-Z^p}\ar@{->>}_{\tilde f_p}[rr] & & \repz{X/X-Z^{p+1}}
}
$$
Ce digramme donne lieu à son tour à une octaèdre (du type \eqref{octaedre})
 dans la catégorie dérivée $\Der(\tilde X_\nis)$.
D'après \eqref{IetE}, la suite spectrale associée à ce dernier
diagramme pour le foncteur cohomologique $\Hom(.,I)$ est canoniquement
isomorphe à la suite spectrale $E_{1,c}^{p,q}$.
Notons $F^p_c(I)$ le noyau de l'épimorphisme
$$
I=\uHom(\ZZ_X(X),I) \xrightarrow{i_p^*} \uHom(\ZZ_X(X-Z^{p+1}),I)
$$
et posons $G^p=F^p_c(I)/F^{p+1}_c(I)$.
On obtient alors un octaèdre dans $\Der(X_\nis)$:
$$
\xymatrix@C=10pt@R=24pt{
& F^{p+1}_c(I)\ar[rr]\ar[ld]
 && F^p_c(I)\ar[rd]\ar[ld] & \\
I\ar|/-12pt/{i^*_{p+1}}[rd]\ar@{}|{*}[r]
 && G^p\ar[ld]\ar_/-5pt/{_{+1}}[lu]
      \ar@{}|{*}[u]\ar@{}|{*}[d]
 && I\ar|/-12pt/{i^*_{p}}[ld]\ar@{}|{*}[l] \\
& \uHom(\ZZ_X(X-Z^{p+1}),I)\ar_{f_p^*}[rr]\ar^/8pt/{_{+1}}[uu]
 && \uHom(\ZZ_X(X-Z^{p}),I)\ar_/8pt/{_{+1}}[uu]\ar^/8pt/{_{+1}}[lu]
}
$$
qui montre que la suite spectrale d'hypercohomologie du complexe
filtré $(I,F^p_c)$ est canoniquement isomorphisme à $E_{1,c}^{p,q}$.
Ainsi, on peut calculer le $(p+q)$-ème faisceau de cohomologie du 
complexe $G^p$:
\begin{equation*}
\uH^{p+q}(G^p)=C^p(.,\umH^q_*(\E))_0
\end{equation*}
où le complexe de cycles à coefficients dans $\umH^q(\E)$
est vu comme un faisceau Nisnevich sur $X_\nis$ -- rappelons
en effet que ce dernier est fonctoriel par rapport aux morphismes
étales.
Si l'on considère $\cE^{p,q}_{1,c}$ la suite spectrale du
complexe filtrée $(I,F^p_c)$ dans la catégorie abélienne
$\tilde X_\nis$, on obtient même un isomorphisme de complexes
de faisceaux:
\begin{equation}  \label{pf:calcul_grad}
\cE^{*,q}_{1,c} = C^*(.,\umH^q_*(\E))_0.
\end{equation}

Notons $F_{triv}(I)$ la filtration triviale sur $I$:
$$
F_{triv}^p(I)=\begin{cases}
I & \text{si } p<0 \\
0 & \text{sinon,}
\end{cases}
$$
et $\cE^{p,q}_{r,triv}$ la suite spectrale
 dans $\tilde X_{nis}$ qui lui est associée.
Evidemment, $\cE^{*,q}_{1,triv}$
 est un complexe concentré en degré $0$ 
 égal au faisceau $\uH^q(I)$.
 
On peut alors considérer le morphisme canonique
de complexes filtrés:
$$
\varphi':(I,F_{triv}) \rightarrow (I,F_c).
$$

Il résulte du calcul \eqref{pf:calcul_grad}
que le morphisme de complexes de faisceaux induit par $\varphi'$
$$
\cE_{1,triv}^{*,q} \rightarrow \cE_{1,c}^{*,q}
$$
est un quasi-isomorphisme pour tout $q \in \ZZ$.
Evalué en un $X$-schéma étale $V$, ce quasi-isomorphisme correspond au
morphisme d'augmentation canonique
$$
\Gamma(V,\uH^q_0(\E)) \rightarrow C^*(V,\umH^q_*(\E))_0.
$$
Il en résulte que le morphisme induit sur les termes de la
deuxième page
\begin{equation}  \label{pf:ident_ssps}
\cE_{2,triv}^{p,q} \rightarrow \cE_{2,c}^{p,q}
\end{equation}
est le morphisme nul si $p \neq 0$, et correspond
pour $p=0$ à l'isomorphisme canonique de la proposition
\ref{iso_coh_gpe_chow}
$$
\uH^q_0(\E) \rightarrow A^0(.,\umH^q_*(\E))_0.
$$
Soit $Dec$ le foncteur de \og décalage de la filtration \fg~définit
dans \cite[1.3.3]{Del}. Le morphisme $\varphi'$ induit donc
un morphisme entre les filtrations décalées:
$$
\varphi'':(I,Dec(F_{triv})) \rightarrow (I,Dec(F_c))
$$
Notons que d'après \cite[1.4.6]{Del}, $Dec(F_{triv})$
 est la filtration \emph{canonique} sur $I$ --
 qui correspond à la filtration pour la t-structure
 homotopique sur $\E$ d'après le choix de $I$.
Il résulte de \cite[1.3.15]{Del} et du calcul précédent
que $\varphi''$ est un quasi-isomorphisme de complexes
filtrés. Il induit donc un isomorphisme au niveau des
couples exacts associés dans la catégorie $D(\tilde X_{\nis})$
et a fortiori un isomorphisme de suite spectrales 
d'hypercohomologies. L'isomorphisme ainsi obtenu sur la première page 
des suites spectrales (\emph{cf.} \cite[1.3.4]{Del})
est de la forme
$$
(\varphi'')_*^{p,q}:E_{2,t}^{p,q} \rightarrow E_{2,c}^{p,q}.
$$
Compte tenu de l'identification de l'isomorphisme \eqref{pf:ident_ssps}
 obtenue ci-dessus, $(\varphi'')_*^{p,q}$ s'identifie à $\varphi_2^{p,q}$
 ce qui permet de conclure.
\end{proof}

\num \label{fct_coniv}
Notons $E_{r,c}^{p,q}(X,\E)$ la suite spectrale du coniveau associée
 au schéma lisse $X$ et au motif $\E$ comme ci-dessus.
Le corollaire principal de la proposition précédente est la fonctorialité
 de la suite spectrale du coniveau. 
Les cas les plus importants sont la compatibilité
 au pullback, pushout et action de la cohomologie motivique:
\begin{itemize}
 \item Un morphisme, ou même une correspondance finie,
 $f:Y \rightarrow X$ entre schémas lisses induit un morphisme
 de suites spectrales:
$$
f^*:E_{r,c}^{p,q}(X,\E) \rightarrow E_{r,c}^{p,q}(Y,\E)
$$
qui converge vers le morphisme $f^*:H^{p+q}(X,\E) \rightarrow H^{p+q}(Y,\E)$.
\item Soit $f:Y \rightarrow X$ un morphisme projectif de dimension relative pure $n$
 entre schémas lisses. Rappelons que l'on associe à $f$ un morphisme de Gysin
$f^*:M(X)(n)[2n] \rightarrow M(Y)$ (\emph{cf.} \cite[2.7]{Deg6}).
On en déduit un morphisme de suites spectrales:
$$
f_*:E_{r,c}^{p,q}(Y,\E) \rightarrow E_{r,c}^{p-2n,q}(X,\E(-n))
$$
qui converge vers le morphisme $f_*:H^{p+q}(Y,\E) \rightarrow H^{p+q-2n}(X,\E(-n))$.
\item Considérons une classe $x \in H_\cM^{i,n}(X)$ dans la cohomologie motivique
 de $X$. On en déduit un morphisme
$$
\gamma_x:E_{r,c}^{p,q}(X,\E) \rightarrow E_{r,c}^{p+i,q}(X,\E(n))
$$
qui converge vers $\gamma_x:H^{p+q}(X,\E) \rightarrow H^{p+q+i}(X,\E(n))$.
\end{itemize}

\begin{num}
Fixons un corps $K$ de caractéristique $0$.
Si $V$ est un $K$-espace vectoriel, on note $V^\vee$ le dual de $V$. \\
Considérons une théorie de Weil mixte $E$ à coefficients dans $K$,
 au sens de \cite{CD1}.
Rappelons qu'il s'agit d'un préfaisceau en $K$-algèbres différentielles
graduées sur la catégorie des $k$-schémas affines lisses
dont
l'hypercohomologie Nisnevich peut être étendue en un foncteur
covariant monoïdal
$$
R_E:\DM \rightarrow \Der(K).
$$
Plus précisément,
 on obtient avec ces notations, pour tout schéma lisse $X$ et tout entier $p \in \ZZ$, 
$$
H^p(X,E)=H^p\big( R_E(M(X))^\vee \big).
$$
Par ailleurs, on associe à $E$ un spectre motivique $\E$
(\emph{cf.} \cite[2.7.6, 2.7.9]{CD1}) tel que:
$$
H^p(X,E)=\Hom_{\DM}(M(X),\E[p]).\footnote{
Avec ces notations, $R_E(\F)=\derR \Hom_{\DM}(\un,\E \otimes \F)$.}
$$
Notons que le faisceau avec transferts $\E_0$ s'identifie,
 après oubli des transferts,
  au faisceau $E_\nis$ associé au préfaisceau $E$. \\

Pour tout entier $n \geq 0$, on pose $K(n)=H^1_\nis(\GG,E)^{\otimes n}$, $K(-n)=K(n)^\vee$
-- par définition de $E$, ces espaces vectoriels sont de dimension $1$.
Pour tout $K$-espace vectoriel $V$, on pose $V(\pm n)=V \otimes K(\pm n)$.
L'isomorphisme canonique ci-dessus s'étend avec ces notations:
$$
H^p(X,E)(n)=\Hom_{\DM}(M(X),\E(n)[p]).\footnote{Ainsi, le spectre $\E$
 est \og $\un(1)$-périodique \fg: il existe un isomorphisme (non canonique)
 $E\simeq E(1)$.}
$$

On en déduit donc la suite spectrale du coniveau à coefficiens dans $\E$:
\begin{equation} \label{eq:coniveau_MWC}
E_{1,c}^{p,q}=\bigoplus_{x \in X^{(p)}} H^{q-p}(\kappa(x),E)(-p)
 \Rightarrow H^{p+q}(X,E).
\end{equation}
D'après le théorème précédent, cette suite spectrale est canoniquement
isomorphe -- à partir du terme $E_2$ -- à la suite spectrale 
d'hyper-cohomologie pour la t-structure homotopique sur $\DM$:
\begin{equation} \label{eq:htp_MWC}
E_{2,t}^{p,q}=H^p(X,\uH^q_0(\E)) \Rightarrow H^{p+q}(X,E).
\end{equation}
Il en résulte que la filtration par coniveau sur $H^*(X,E)$
 coincide avec la filtration donnée par la t-structure homotopique
 relativement à $\E$.

Ce résultat est à comparer avec la proposition (6.4) de \cite{BO},
 d'autant plus que d'après la démonstration de \ref{thm:comparaison_ssp},
 la suite spectrale \eqref{eq:htp_MWC} s'identifie à la suite spectrale
 d'hypercohomologie Nisnevich associée au complexe $E_\nis$ sur le site
 $X_\nis$.
Le faisceau $\uH^q_0(\E)$ s'identifie avec le faisceau Nisnevich
 associé au préfaisceau
$$
\upH^q_0(E):X \mapsto H^q(X,E).
$$
Comme ce dernier est un préfaisceau invariant par homotopie avec transferts,
 $\uH^q_0(\E)$ s'identifie encore au faisceau Zariski associé
 à $\upH^q_0(\E)$ (\emph{cf.} \cite[4.4.16]{Deg3}) -- 
il coincide donc avec le faisceau noté $\mathcal H^q$ dans \cite{BO} une
fois oublié les transferts.

Comme $E$ est sans torsion,
 il résulte de \cite[5.24]{V2} que $\uH^q_0(\E)$ est un faisceau étale.
De plus, d'après \cite[5.7, 5.28]{V2},
$$
H^p(X,\uH^q_0(\E))=H^p_\zar(X,\uH^q_0(\E))=H^p_{\et}(X,\uH^q_0(\E)).
$$
On peut démontrer de plus que la suite spectrale \eqref{eq:htp_MWC}
 coincide avec la suite spectrale d'hyper-cohomologie étale (resp. Zariski)
 associé au complexe $E_\nis$. Pour la topologie étale,
 cela résulte directement de l'équivalence
$$
DM_-^{eff}(k) \otimes \QQ \xrightarrow{\ \sim \ }
 DM_{-,\et}^{eff}(k) \otimes \QQ
$$
prouvé par Voevodsky (\emph{cf.} \cite[3.3.2]{V1}).
\end{num}

Pour résumer\footnote{Signalons que le premier point de cette proposition
 a été obtenu dans \cite{Deg6}.}:
\begin{cor}
Soit $E$ une théorie de Weil mixte à coefficients dans $K$,
 $E_\nis$ le faisceau Nisnevich
 et $\E$ le spectre motivique qui lui sont associés.

\begin{enumerate}
\item
Pour tout schéma lisse $X$, la filtration par coniveau sur $X$
 induit une suite spectrale convergente
$$
E_{1,c}^{p,q}(X,\E) \Rightarrow H^{p+q}(X,E)
$$
dont le complexe sur la ligne $q$ est $E_{1,c}^{*,q}(X,\E)=C^*(X,\umH^q_* \E)_0$.
\item Cette suite spectrale s'identifie à partir du terme $E_2$
 avec la suite spectrale \eqref{eq:htp_MWC}
 induite par la filtration sur $\E$ donnée par
 la t-structure homotopique de $\DM$.
\item Elle s'identifie encore avec les suites spectrales
 d'hyper-cohomologie Nisnevich et étale de $X$ à coefficients
 dans $E_\nis$.
\end{enumerate}
\end{cor}

\part{Applications et compléments} 

\section{Morphismes de Gysin et correspondances}

Soit $f:Y \rightarrow X$ un morphisme fini équidimensionnel entre schémas lisses.
Suivant \cite[2.7]{Deg6}, on associe à $f$ un morphisme de Gysin $f^*:M(X) \rightarrow M(Y)$
 dans $\DM$.
Par ailleurs, la transposée du graphe de $f$ définit une correspondance finie 
 $\tra f:X \doto Y$, qui a son tour induit un morphisme $\tra f_*:M(X) \rightarrow M(Y)$
 dans $\DM$.
Dans \cite[2.13]{Deg6}, on a montré que $f^*=\tra f_*$ si $f$ est étale.
On obtient ici la généralisation dans le cas ramifié:
\begin{prop}\label{prop:Gysin&tanspose}
Avec les notations ci-dessus, $f^*=\tra f_*$.
\end{prop}
\begin{proof}
On pose $\alpha=f^*-\tra f_*$ et on montre que $\alpha=0$.
Il suffit de montrer que pour tout $\E \in \DM$, le morphisme
$$
\Hom_{\DM}(M(Y),\E) \xrightarrow{\alpha^*} \Hom_{\DM}(M(X),\E)
$$
est nul.
Comme la t-structure homotopique sur $\DM$ est non dégénérée,
 il suffit de traiter le cas où $\E$ est dans le coeur homotopique.
On peut donc supposer que $\E$ est un module homotopique.
Ce cas résulte finalement de la proposition \ref{Gysin&pushout}
 et de la formule \eqref{transferts_Chow}.
\end{proof}

\section{Borne inférieure et constructibilité des modules de cycles}

On introduit l'hypothèse suivante sur le corps $k$:
\begin{enumerate}
\item[$(\cM_k)$]
Pour tout corps de fonctions $E/k$, il existe
 un $k$-schéma projectif lisse dont le corps des fonctions
 est $k$-isomorphe à $E$.
\end{enumerate}
Cette hypothèse est évidemment une conséquence de la résolution
 des singularités au sens classique pour $k$.

Le résultat suivant est bien connu\footnote{On obtient une
preuve très élégante en utilisant un argument dû à J. Riou
facilement adapté de la preuve de \cite[th. 1.4]{Riou}.}:
\begin{prop} \label{prop:eng_proj_lisse}
Soit $d$ un entier et $\sP_{\leq d}$
 la sous-catégorie triangulée de $\DM$ engendrée par les motifs
 de schémas projectifs lisses de dimension inférieure à $d$.

Soit $X$ un schéma lisse de dimension inférieure à $d$.
\begin{enumerate}
\item[(i)] Si $(\cM_k)$ est vérifiée,
 $M(X)$ appartient à $\sP_{\leq d}$.
\item[(ii)] Le motif rationel $M(X) \otimes \QQ$ appartient
 à $\sP_{\leq d} \otimes \QQ$.
\end{enumerate}
\end{prop}

On en déduit le résultat suivant:
\begin{prop}
Soit $X$ un schéma de dimension $d$ et $(n,i) \in \ZZ^2$
 un couple d'entiers.
\begin{enumerate}
\item[(i)] Si $X$ est projectif lisse,
 $\hat h_{i,-n}(X)=0$ si $n>d$.
\item[(ii)] Si $(\cM_k)$ est vérifiée, 
 $\hat h_{i,-n}(X)=0$ si $n>d$.
\item[(iii)] Dans tous les cas,
 $\hat h_{i,-n}(X) \otimes \QQ=0$ si $n>d$.
\end{enumerate}
\end{prop}
\begin{proof}
Le point $(i)$ est un corollaire de la formule \eqref{eq:rep_dual_cohm}
 et du théorème de simplification de Voevodsky car ce dernier affirme
 qu'il n'y a pas de cohomologie motivique en poids strictement négatif.

Soit $\C_{\leq d}$ la sous-catégorie pleine de $\DM$ formée
 des motifs $\cM$ tel que pour tout corps de fonctions $E$
 et tout couple $(n,i) \in \ZZ^2$, $n>d$,
$$
\Hom_{\DM}(M(E),\cM\{-n\}[-i])=0.
$$
Cette catégorie est une sous-catégorie triangulée.
D'après (i), elle contient les motifs $M(P)$ pour $P$ projectif
lisse de dimension inférieure à $d$. La proposition précédente
permet donc de conclure.
\end{proof}

\begin{df} \label{df:constructible}
Nous dirons qu'un module homotopique 
(resp. module de cycles)
est \emph{constructible}
 s'il appartient à la sous-catégorie épaisse\footnote{\emph{i.e.} 
 stable par noyau, conoyau, extension, sous-objet et quotient.}
 de $\hmtr$ (resp. $\modl$)
 engendrée par les objets $\hStab h_i(X)\{n\}$ (resp. $\hat h_i(X)\{n\}$)
 pour un schéma lisse $X$
 et un couple d'entiers $(n,i) \in \ZZ^2$.
\end{df}

\rem 
\begin{enumerate}
\item Grâce à la t-structure homotopique,
 on peut considérer une autre condition de finitude sur les modules homotopiques.
 Un module homotopique $F_*$ est dit \emph{fortement constructible}
  s'il est de la forme
 $\uH_*^0(\E)$ pour un motif géométrique $\E$.\footnote{
  De même, un module de cycles est fortement constructible
   si le module homotopique associé l'est.}
 Dans ce cas, $F_*$ est constructible dans le sens précédent
 mais la réciproque n'est pas claire.
\item Les modules homotopiques constructibles ne jouissent pas des propriétés
 de finitude de leur analogue $l$-adique. 
Ainsi, il y a lieu de considérer parallèlement la notion plus forte 
 de module homotopique
 \emph{de type fini}\footnote{Cette notion, 
 introduite dans la thèse de l'auteur \cite{Deg1}, 
 a été étudiée indépendamment par J. Ayoub dans l'appendice de \cite{HKA}.}:
 $F_*$ est de type fini s'il existe un épimorphisme
  $\hStab h_0(X) \rightarrow F_*$.
Ces subtilités interviennent car le foncteur $\uH^0$ ne préserve pas
 la propriété d'être géométrique (\emph{i.e.} compact) -- contrairement
 à son analogue $l$-adique, 
 le foncteur cohomologique associé à la t-structure canonique,
 qui lui préserve la constructibilité.
\item Dans le prolongement de la remarque précédente,
 notons qu'il est probable que la plupart des modules homotopiques constructibles
 ne soient pas fortement dualisables. La seule exception que l'on connaisse
 à cette règle est le cas d'un $k$-schéma étale $X$ 
 et du module homotopique $\hStab h_0(X)$. 
 Ce dernier est fortement dualisable dans $\hmtr$ (ou même dans $\hftr$) 
 et il est son propre dual fort.
\end{enumerate}

\begin{cor}
La graduation canonique d'un module de cycles constructible $M$
 est bornée inférieurement dès que l'une des deux propriétés
 suivantes est réalisée:
\begin{itemize}
\item La propriété $(\cM_k)$ est satisfaite.
\item $M$ est sans torsion.
\end{itemize}
\end{cor}

\section{Homologie de Borel-Moore}

On suppose $k$ de caractéristique $0$.
Dans \cite[par. 4]{V1}, Voevodsky associe à tout schéma
algébrique $X$ un motif à support compact $M^c(X)$ dans $\DMgm$
 vérifiant les propriétés suivantes:
\begin{enumerate}
\item[(C1)] Si $X$ est propre, $M^c(X)=M(X)$.
\item[(C2)] $M^c(X)$ est covariant par rapport aux morphismes
 propres, contravariant par rapport aux morphismes étales.
\item[(C3)] Si $i:Z\rightarrow X$ est une immersion fermée
 d'immersion ouverte complémentaire $j:U \rightarrow X$,
 il existe un triangle distingué canonique:
$$
M^c(Z) \xrightarrow{i_*} M^c(X) \xrightarrow{j^*} M^c(U)
 \xrightarrow{+1}
$$
\item[(C4)] Si $X$ est lisse de dimension pure $d$, $M(X)$
 est fortement dualisable avec pour dual fort $M^c(X)(-d)[-2d]$.
\end{enumerate}

Considérons un module de cycles $\phi$ ainsi que le module
 homotopique $F_*$ qui lui est associé d'après le théorème
 \ref{thm:main}.

Considérons un schéma lisse $X$ connexe de dimension $d$.
Pour un couple d'entier $(n,r) \in \ZZ$, on obtient les isomorphismes
suivants:
\begin{align*}
&A_{d-n}(X,\phi)_{r-d}=A^n(X,\phi)_r
 \stackrel{(1)}=H^n(X,F_r) \\
 &=\Hom_{\DM}(M(X),F_*\{r\}[n])
 \stackrel{(2)}=\Hom_{\DM}(\un[d-n],M^c(X)\{r-d\}).
\end{align*}
où $(1)$ résulte de \ref{iso_coh_gpe_chow}
 et $(2)$ de la propriété (C4). \\
Nous dirons qu'un schéma $X$ est lissifiable si il
 existe une immersion fermée de $X$ dans un schéma lisse.
Utilisant la propriété (C3) ci-dessus et la suite exacte longue
 de localisation pour les groupes de Chow à coefficients,
 on en déduit aisément:
\begin{cor}
Soit $\phi$ un module de cycles et $F_*$ le module homotopique
lui correspondant par l'équivalence de \ref{thm:main}.
Pour tout schéma lissifiable $X$ et tout couple $(i,s) \in \ZZ^2$,
 il existe un isomorphisme canonique\footnote{On démontre notamment qu'il ne dépend
  pas du plongement dans un schéma lisse.}:
$$
A_i(X,\phi)_s \simeq \Hom_{\DM}(\un[i],M^c(X) \otimes F_*\{s\}).
$$
\end{cor}
En utilisant la fonctorialité de la suite exacte longue de localisation
 et du motif à support compact (\emph{cf.} (C2)),
 cet isomorphisme est fonctoriel covariant par rapport aux morphismes
 propres, contravariant par rapport aux morphismes étales.
Si on utilise de plus la fonctorialité étendue aux morphismes plats
 du motif à support compact (\emph{cf.} \cite[4.2.4]{V1}), 
 on obtient même la fonctorialité contravariante de cet isomorphisme 
 par rapport aux morphismes plats.

\rem
\begin{enumerate}
\item Le membre de droite de l'isomorphisme du corollaire précédent
 mérite le nom d'homologie de Borel-Moore de $F_*$ en degré $i$
 (et $\GG$-twist $s$). L'isomorphisme est alors à comparer avec
 l'isomorphisme entre homologie motivique de Borel-Moore et
 groupes de Chow supérieurs (\emph{cf.} \cite[4.2.9]{V1}).
\item On peut se passer de l'hypothèse que $X$ admet un plongement
 dans un schéma lisse en utilisant le \og truc de Jouanolou \fg.
\item Le corollaire précédent utilise uniquement les propriétés
(C1) à (C4) du motif à supports compacts. Quitte à travailer
rationellement, on peut obtenir ces propriétés (\emph{cf.} \cite{CD3})
et prolonger ainsi la conclusion du corollaire au cas d'un corps
parfait quelconque.
\end{enumerate}

\section{Motifs birationnels}

Rappelons la définition des \emph{complexes motiviques birationnels}
 due à B. Kahn et R. Sujatha (\emph{cf.} \cite{KS}):
\begin{df}[Kahn, Sujatha]
On note $\DMb$ (resp. $\DMgmb$) la localisation de la catégorie $\DMe$
 (resp. enveloppe pseudo-abélienne de la localisation de $\DMgme$)
 par rapport à la classe de flèches formée des morphismes
 $M(U) \xrightarrow{j_*} M(X)$ où $j$ est une immersion ouverte
 et $X$ un schéma lisse.
\end{df}
Suivant \cite{KS}, le foncteur canonique $\DMe \rightarrow \DMb$
est noté $\nu_{\leq 0}$ -- \emph{cf.} \cite[6.3]{KS}.
La catégorie $\DMb$ est une localisation de Bousfield à
 gauche de la catégorie homotopique $\DMe$. On en déduit que
 $\nu_{\leq 0}$ admet un adjoint à droite $i:\DMb \rightarrow \DMe$
 pleinement fidèle qui identifie
 la catégorie $\DMb$ à la sous-catégorie
 de $\DMe$ formée des complexes motiviques $K$ tels que
 pour toute immersion ouverte $j:U \rightarrow X$ dans un schéma lisse $X$,
 le morphisme induit 
 $H^*(X,K) \xrightarrow{j^*} H^*(U,K)$  est un isomorphisme.
Un tel complexe $K$ est appelé suivant Kahn et Sujatha un
 \emph{complexe motivique birationnel}.

\rem D'après un fait général sur les localisations de Bousfield (\emph{cf.} \cite[\S 5]{CD1}),
 la catégorie $\DMgmb$ s'identifie à la sous-catégorie pleine des objets compacts
 de $\DMb$ par le foncteur canonique $\DMgmb \rightarrow \DMb$ -- \emph{cf.} \cite[6.4.c]{KS}.
Par ailleurs, la catégorie $\DMb$ s'identifie à la localisation de $\DMe$ par
 rapport à la sous-catégorie pleine\footnote{Cette sous-catégorie est pleine
 d'après le théorème de simplification de Voeovdsky: voir \ref{num:simplification}.}
 $\DMe(1)$ -- \emph{cf.} \cite[6.4.a]{KS}.

\num Soit $X$ (resp. $Y$) un schéma lisse connexe de corps des fonctions $K$
 (resp. $L$).
Pour tout ouvert $U$ (resp. $V$) de $X$ (resp. $Y$),
 on obtient un morphisme
\begin{align*}
\varphi_{U,V}:\Hom_{\DMe}(M(U),M(V)) \rightarrow & \Hom_{\DMb}(M(U),M(V)) \\
& \simeq \Hom_{\DMb}(M(X),M(Y)).
\end{align*}
Passant à la limite, on en déduit un morphisme canonique
$$
\varphi:\Hom_{\DMgmo}(M(K),M(L)) \rightarrow \Hom_{\DMb}(M(X),M(Y)).
$$
\begin{prop}
Le morphisme $\varphi$ est surjectif.
\end{prop}
\begin{proof}
Le foncteur $\nu_{\leq 0}$ est plein. Donc pour tout couple $(U,V)$,
 $\varphi_{U,V}$ est surjectif. Le fonteur limite inductive préserve
 les épimorphismes ; ainsi, pour tout $V \subset Y$, le morphisme
$$
\varphi_V:\ilim{U \subset X} \Hom_{\DMe}(M(U),M(V)) \rightarrow \Hom_{\DMb}(M(X),M(Y))
$$
est surjectif.
D'après \cite[3.4.4]{Deg7}, la source de ce morphisme est canoniquement isomorphe
à $\hhlrep_0(V)(K)$.
Si on considère de plus une immersion ouverte $j:V' \rightarrow V$,
 il résulte de la proposition \ref{prop:h_0(imm_ouv)=epi} que le morphisme induit
$$
\hhlrep_0(V')(K) \xrightarrow{j_*} \hhlrep_0(V)(K)
$$
est un épimorphisme.
Il en résulte que le système projectif
 $(\Ker(\varphi_V))_{V \subset Y}$
 vérifie la condition de Mittag-Leffler ce qui permet de conclure
  (\emph{cf.} par exemple \cite[0-13.2.2]{EGA3}).
\end{proof}
Le morphisme $\varphi$ n'est pas injectif en général car pour une
 valuation $v$ sur un corps de fonctions $E$ et une uniformisante
 $\pi$ de $v$, le morphisme
$$
s_v^\pi=\gamma_\pi \circ \partial_v\{-1\}:M(\kappa_v)
 \rightarrow M(E)
$$
dépend de $\pi$ alors qu'il n'en dépend pas dans les motifs birationnels
 d'après \cite{KS}.
On en déduit donc un morphisme plein, non fidèle, de la catégorie
des motifs génériques \emph{sans twist} vers la catégorie des
motifs birationnels.

\num Soit $E$ un corps de fonctions et $\nu$ un valuation éventuellement
 non discrète sur $E$. On note $\mathcal O$ l'anneau des entiers de $\nu$
 qui n'est donc pas nécessairement noethérien.
La proposition précédente est intéressante
 car elle montre que $\nu$ détermine au moins un morphisme de spécialisation 
 $s_\nu:M(\kappa_v) \rightarrow M(E)$ -- \emph{cf.} \cite{KS}.

Plus canoniquement, on peut attacher à $\nu$ un morphisme résidu.
En effet, $\mathcal O$ est réunion filtrante de ses sous-anneaux
 de valuations $\mathcal O_v$ qui sont des $k$-algèbres de type fini,
 correspondant à une valuation géométrique $v$.
Notons $\cI_\nu$ l'ensemble ordonné de ces sous-$k$-algèbres $\mathcal O_v$,
 indexé par les valuations $v$ correspondantes.
Quitte à remplacer $\cI_\nu$ par une partie cofinale,
 on peut supposer que pour tout $v \in \cI_\nu$, $\mathrm{Frac}(\cO_v)=E$
 et la fibre spéciale de $\spec{\cO} \rightarrow \spec{\cO_v}$ est non
 vide.
A tout valuation $v$ de $\cI_\nu$, de corps résiduel $\kappa_v$,
 on associe un morphisme résidu
$$
\partial_v:M(\kappa_v)\{1\} \rightarrow M(E).
$$
Ce morphisme est naturel par rapport à toute inclusion $\cO_v \subset \cO_w$:
 d'après la formule (R3b) de \cite[4.4.7]{Deg5}, le diagramme suivant est commutatif
$$
\xymatrix@R=15pt@C=22pt{
M(\kappa_w)\{1\}\ar^-{\partial_w}[r]\ar_{\bar \varphi^\sharp}[d] & M(E)\ar@{=}[d] \\
M(\kappa_v)\{1\}\ar^-{\partial_v}[r] & M(E),
}
$$
où $\bar \varphi:\kappa_v \rightarrow \kappa_w$ est le morphisme induit sur
 les corps résiduel. En effet, l'extension $\cO_w/\cO_v$ est non ramifiée.
Finalement, on peut donc poser
$$
\partial_\nu=\pplim{v \in \cI_\nu} \partial_v:M(\kappa_\nu)\{1\} \rightarrow M(E).
$$

\rem Cet exemple montre que les morphismes (D1) à (D4) dégagés
 dans la catégorie des motifs génériques (\emph{cf.} \cite{Deg5}) n'engendrent pas,
 tous les morphismes entre motifs génériques -- du moins par composition \emph{finie}.
Il est d'autant plus remarquable qu'ils permettent de décrire complètement
 un module homotopique en tant que \emph{morphismes de spécialisations}
 entre ses fibres.

\bibliographystyle{smfalpha}
\bibliography{modhtp}

\end{document}